\documentclass{amsart}

\usepackage{xr}
\usepackage{lineno,hyperref}
\usepackage{amssymb}
\usepackage{amsmath}
\usepackage{amsthm}
\usepackage{mathrsfs}
\usepackage{bbm}
\usepackage{cite}
\usepackage{color}
\usepackage{enumerate}

\newtheorem{thm}{Theorem}[section]
\newtheorem{cor}[thm]{Corollary}
\newtheorem{defn}[thm]{Definition}

\newtheorem{prop}[thm]{Proposition}
\newtheorem{exa}[thm]{Example}
\newtheorem{rmk}[thm]{Remark}

\newcommand{\diag}{\mathop{\mathrm{diag}}}
\newcommand{\id}{\mathop{\mathrm{id}}}

\begin{document}

\title{Riemann slice-domains over quaternions II}
\author{Xinyuan Dou}
\email[X.~Dou]{douxy@mail.ustc.edu.cn}
\author{Guangbin Ren}
\email[G.~Ren]{rengb@ustc.edu.cn}
\date{\today}
\address{Department of Mathematics, University of Science and Technology of China, Hefei 230026, China}


\begin{abstract}
	We generalize the representation formula from slice-domains of regularity to general Riemann slice-domains. This result allows us to extend the $*$-product of slice regular functions on axially symmetric domains to certain Riemann slice-domains by introducing holomorphic stem systems and tensor holomorphic functions. In particular, we construct a power series expansions of slice regular functions on certain Riemann slice-domains, which implies a relation between stem holomorphic, tensor holomorphic and slice regular functions.
 \end{abstract}

\maketitle

\section{Introduction}

In order to solve the problem of multivalued functions originated from the analytic continuation of slice regular functions on quaternions, we develop a theory of Riemann slice-domains (see \cite{Dou2018001}) following the classical one in complex analysis. In this paper, we will generalize the representation formula and the $*$-product to some suitable Riemann slice-domains.

Gentili and Struppa introduce the theory of slice regular functions in \cite{Gentili2006001,Gentili2007001}, which has developed rapidly in the past 10 years (see \cite{Gentili2008002,Colombo2010003,Colombo2015001,Ghiloni20171001}). A representation formula on axially symmetric domains \cite{Colombo2009001,Colombo2010001}, applied in \cite{Colombo2011001,Gentili2016001,Colombo2018001}, plays an important role in this theory. This formula relates the value that a slice regular function assumes at one point to the values it assumes on two certain axially symmetric points. Therefore we can extend several results from complex analysis to the theory of slice regular functions by this formula.

We have proved the representation formula over slice-domains of regularity (see \cite[Theorem \ref{th-dere}]{Dou2018001}). In this paper, we will provide a weak condition  (see Section \ref{sc-2sli}), and prove a representation formula over general Riemann slice-domains (see Theorem \ref{tm-rf}). Then we show some symmetry of slice-domains of existence in Section \ref{sc-2rf}. In Section \ref{sc-2ese}, two examples show that how the representation formula works. With the increase of the dimension of the imaginary units, the representation formula is more complicated. Even these two simple examples are somewhat intricate.

We next consider the $*$-product in Riemann slice-domains. In the theory of slice-regular functions, the $*$-product is introduced in \cite{Gentili2008001} and generalized \cite{Colombo2009001} to axially symmetric slice domains in $\mathbb{H}$. The $*$-product also has been generalized to Clifford algebras \cite{Colombo2010002} and real alternative algebras \cite{Ghiloni2011001}. It has important applications in Schur analysis \cite{Alpay2012001,Alpay2013001}, Schwarz lemma \cite{Bisi2012001,Ren2017002}, twistor transforms of quaternionic functions \cite{Gentili2014001}, and the theory of quaternionic operators \cite{Alpay2015001,Alpay2016001,Ghiloni2018001}.
In our setting, we need to introduce a theory of holomorphic stem systems (see Section \ref{sc-hss}), following a method proposed by Fueter in \cite{Fueter1934/35}. This method has been extended by Ghiloni and Perotti in order to develop a theory of slice regular functions on real alternative algebras in \cite{Ghiloni2011001}. Our definition of holomorphic stem systems relies heavily on the representation formula. In the representation formula, one factor does not depend on the choice of imaginary units, which  motivates us to define the concept of the stem holomorphic functions. For this concept, due to the topological intricacy of Riemann slice-domains, more compatibility conditions (see Definition \ref{df-hssy}) are needed than the complex intrinsic in \cite[Definition 4]{Ghiloni2011001}. Thanks to these compatibility conditions, we can lift each holomorphic stem system to a  slice regular function on some suitable Riemann slice-domain (see Theorem \ref{th-sfsr}).

However, notice that stem holomorphic functions are vector-valued (more precisely, $M_{2^N\times 1}(\mathbb{H})$-valued for some $N\in\mathbb{N}^+$). We can not define simply their multiplication. Following \cite{Ghiloni2011001}, we define tensor product-valued (more precisely, $\mathbb{C}^{\otimes N}\otimes\mathbb{H}$-valued for some $N\in\mathbb{N}^+$) function, called tensor holomorphic functions. Thanks to an isomorphism between $M_{2^N\times 1}(\mathbb{H})$ and $\mathbb{C}^{\otimes N}\otimes\mathbb{H}$ (see Proposition \ref{pr-vt}), then we can define a $*$-product of stem holomorphic functions, induced by the product of tensor holomorphic functions. As a consequence, we can define a $*$-product of holomorphic stem systems (see Theorem \ref{th-sp}).

Now we consider the $*$-product of slice regular functions on Riemann slice-domains. In the classical case, the $*$-product of two slice regular functions on an axially symmetric slice domain in $\mathbb{H}$ is essentially the unique slice regular extension to the whole domain of their pointwise product on the real axis. We attempt to generalize this $*$-product to Riemann slice-domains. In this general case, we find a phenomenon that even if a Riemann slice-domain $(G,\pi,x_0)$ has certain symmtries, the pointwise product of the restrictions to $G^{x_0}_\mathbb{R}$ of two slice regular functions on $G$ may not necessarily extend regularly to the whole domain $G$; for the precise definitions of $(G,\pi,x_0)$ and $G^{x_0}_\mathbb{R}$, see \cite[Definition \ref{df-rsd}]{Dou2018001} and Definition \ref{df-psd}. In fact, it is a problem related to analytic continuations, which in the case of several complex variables initiates the theories of analytic spaces and sheaf cohomology. At present, this may be a hard problem. So we only consider some Riemann slice-domains on which the above phenomenon does not occur. Such Riemann slice-domains are called $*$-preserving (see Definition \ref{df-psd}). In particular, axially symmetric domains in $\mathbb{H}$ are $*$-preserving (see Remark \ref{rm-asp}). But it is hard to identify whether a Riemann slice-domain is $*$-preserving or not, even the Riemann slice-domain is a domain of existence of the square root function $f_\mathbb{R}$ (see Example \ref{ex-sr}).

As applications of the $*$-product for slice regular functions on $*$-preserving Riemann slice-domains, we can define the conjugates and symmetrizations of slice regular functions, as well as the inverses of some suitable slice regular functions (see Proposition \ref{pr-iesr}). We write out the power series expansions of slice regular functions   on $*$-preserving Riemann slice-domains (see Theorem \ref{th-t}) and give a relation between the power series expansions of stem holomorphic, tensor holomorphic and slice regular functions (see Theorem \ref{th-str}).

\section{Preliminaries}

In this section we collect some the basic definitions. We mention that for all those notations unexplained, we refer to our preceding article \cite{Dou2018001}.

\subsection{Slice-domains in $\mathbb{H}$}

In this subsection, we will recall some definitions on slice-domains in $\mathbb{H}$.

Let $\mathbb{S}$ be the 2-sphere of imaginary units of quaternions $\mathbb{H}$, i.e.,
\begin{equation*}
\mathbb{S}:=\{q\in\mathbb{H}:q^2=-1\}.
\end{equation*}
For each subset $U$ of $\mathbb{H}$ and $I\in\mathbb{S}$, we set
\begin{equation*}
\mathbb{C}_I:=\{x+yI:x,y\in\mathbb{R}\},\qquad U_\mathbb{R}:=U\cap\mathbb{R}\qquad\mbox{and}\qquad U_I:=U\cap\mathbb{C}_I.
\end{equation*}

\begin{defn}
	Let  $I\in\mathbb{S}$	and  $\Omega$  be an open set in $\mathbb{C}_I$.  A function $f:\Omega\rightarrow\mathbb{H}$ is said to be (left) holomorphic, if f has continuous partial derivatives and satisfies
	\begin{equation*}
	(\frac{\partial}{\partial x}+I \frac{\partial}{\partial y}) f(x+yI)=0,\qquad\mbox{on}\qquad\Omega.
	\end{equation*}
\end{defn}

Let $i$ ba a fixed imaginary unit in the complex field $\mathbb{C}$. We set
\begin{equation*}
	P_I:\mathbb{C}\rightarrow\mathbb{C}_I,\qquad x+yi\rightarrow x+yI,\qquad\forall\ x,y\in\mathbb{R}\ \mbox{and}\ I\in\mathbb{S}.
\end{equation*}

For each $x\in\mathbb{R}$, let
\begin{equation*}
\lfloor x\rfloor:=\max\{N\in\mathbb{Z}:N\le x\}\qquad(\mbox{resp.}\ \lceil x\rceil:=\min\{N\in\mathbb{Z}:N\ge x\})
\end{equation*}
be the floor (resp. ceiling) integral part of $x$, and
\begin{equation*}
\{x\}:=x-\lfloor x\rfloor
\end{equation*}
the fractional part of $x$.

For each $I\in\mathbb{S}$, let
\begin{equation*}
\mathbb{C}'_I:=(\mathbb{C}_I,I)=\{(z,I):z\in\mathbb{C}_I\}
\end{equation*}
be a field, which is isomorphic to $\mathbb{C}$.
Let $\tau(\mathbb{C}'_I)$ be the  induced topology on $\mathbb{C}'_I$ by $\mathbb{C}$,  and $\tau(\sqcup_{I\in\mathbb{S}} \mathbb{C}'_I)$ be the disjoint union topology. We set
\begin{equation*}
\varphi:\bigsqcup_{I\in\mathbb{S}} \mathbb{C}'_I\rightarrow\mathbb{H},\qquad (q,I)\mapsto q.
\end{equation*}
And let $\tau_s(\mathbb{H})$ be the quotient space topology induced by $\varphi$.

\begin{defn}\label{df-2st}
	We call the topology $\tau_s$, the slice-topology of $\mathbb{H}$.
\end{defn}

Open sets, connectedness and paths in the slice-topology are  called respectively as slice-open sets, slice-connectedness and slice-paths, and so on.

\begin{defn}\label{df-2sr}
	Let $\Omega$ be a slice-open set in $\mathbb{H}$. A function $f:\Omega\rightarrow\mathbb{H}$ is called (left) slice regular, if for each $I\in\mathbb{S}$, $f_I:=f|_{\Omega_I}$ is left holomorphic.
\end{defn}

\subsection{Riemann slice-domains over $\mathbb{H}$} In this subsection, we will recall some definitions on Riemann slice-domains over $\mathbb{H}$.

\begin{defn}\label{df-2rsd}
	A (Riemann) slice-domain over $\mathbb{H}$ is a pair $(G,\pi)$ with the following properties:
	
	1. $(G,\tau(G))$ is a connected Hausdorff   space,
	
	2. $\pi:G\rightarrow\mathbb{H}$ is a local slice-homeomorphism.
\end{defn}

\begin{defn}
	A (Riemann) slice-domain over $\mathbb{H}$ with distinguished point is a triple $\mathcal{G}=(G,\pi,x)$ for which $(G,\pi)$ is a slice-domain over $\mathbb{H}$ and $x\in G$.
\end{defn}

\begin{defn}
	Let $\mathcal{G}_\lambda=(G_\lambda,\pi_\lambda,x_\lambda)$, $\lambda=1,2$, be two slice-domains over $\mathbb{H}$ with distinguished point. We say that $\mathcal{G}_1$ is contained in $\mathcal{G}_2$ (denoted by $\mathcal{G}_1\prec\mathcal{G}_2$), if there exists a continuous map $\varphi:G_1\rightarrow G_2$ with the following properties:
	
	1. $\pi_2\circ\varphi=\pi_1$ (called ``$\varphi$ preserves fibers").
	
	2. $\varphi(x_1)=x_2$.
\end{defn}

\begin{defn}
	Two slice-domains $\mathcal{G}_1,\mathcal{G}_2$ over $\mathbb{H}$ with distinguished point are called equivalent (symbolically $\mathcal{G}_1\cong\mathcal{G}_2$), if $\mathcal{G}_1\prec\mathcal{G}_2$ and $\mathcal{G}_2\prec\mathcal{G}_1$. We denote the equivalence class of $\mathcal{G}_1$ by $[\mathcal{G}_1]$.
\end{defn}

\begin{defn}\label{eq-2lgb}
	Let $\Lambda$ be an index set, and $\mathcal{G}$, $\mathcal{G}_\lambda$, $\lambda\in\Lambda$, be slice-domains over $\mathbb{H}$ with distinguished point. $\mathcal{G}$ is called a upper bound of $\{\mathcal{G}_\lambda\}_{\lambda\in\Lambda}$, if $\mathcal{G}_\lambda\prec\mathcal{G}$ for each $\lambda\in\Lambda$.
\end{defn}

\begin{defn}\label{de-2sdui}
	Let $\Lambda$ be an index set, and $\mathcal{G}$, $\mathcal{G}_\lambda$, $\lambda\in\Lambda$ be slice-domains over $\mathbb{H}$ with distinguished point. $\mathcal{G}$ is called a union of $\{\mathcal{G}_\lambda\}_{\lambda\in\Lambda}$, if $\mathcal{G}\prec\mathcal{G}'$ for each upper bound $\mathcal{G}'$ of $\{\mathcal{G}_\lambda\}_{\lambda\in\Lambda}$.
	
	We denote the set of all unions of $\{\mathcal{G}_\lambda\}_{\lambda\in\Lambda}$ by $\cup_{\lambda\in\Lambda}\mathcal{G}_\lambda$.
\end{defn}

\begin{defn}  \label{def:2slicefunction}
	Let $(G,\pi)$ be a slice-domain over $\mathbb{H}$. A function $f:G\rightarrow\mathbb{H}$ is called slice regular at a point $x\in G$ if there exists an open neighborhood $U\subset G$ of $x$ and a slice-open set $V\subset\mathbb{H}$ such that $\pi|_U:U\rightarrow V$ is a slice-homeomorphism and $f\circ\pi|_U^{-1}:V\rightarrow\mathbb{H}$ is slice regular.
	
	The function $f$ is called slice regular on $G$ if $f$ is slice regular at every point $x\in G$. We denote the set of all slice regular functions on $G$ by $\mathcal{SR}(G)$.
\end{defn}

\begin{defn}
	Let $\mathcal{G_\lambda}=(G_\lambda,\pi_\lambda,x_\lambda)$, $\lambda=1,2$, be slice-domains over $\mathbb{H}$ with distinguished point, and $\mathcal{G}_1\prec\mathcal{G}_2$ by the fiber preserving mapping $\varphi:G_1\rightarrow G_2$ with $\varphi(x_1)=x_2$. For every function $f$ on $G_2$, we define
	\begin{equation*}
	f|_{G_1}:=f\circ\varphi.
	\end{equation*}
\end{defn}

\begin{defn}
	1. Let $(G,\pi)$ be a slice-domain over $\mathbb{H}$, and $x\in G$ be a point. If $f$ is a slice regular function   near $x$, then the pair $(f,x)$ is called a local slice regular function at $x$.
	
	2. Let $(G_1,\pi_1)$, $(G_2,\pi_2)$ be slice-domains over $\mathbb{H}$, and $x_\imath\in G_\imath$, $\imath=1,2$ with $\pi_1(x_1)=\pi_2(x_2)$. Two locally holomorphic functions $(f_1,x_1)$, $(f_2,x_2)$ are called equivalent if there exist an open neighborhood $U$ of $x_1$, an open neighborhood $V$ of $x_2$ and a slice domain $W$ in $\mathbb{H}$, such that
	\begin{equation*}
	\pi_1|_U:U\rightarrow W\qquad\mbox{and}\qquad\pi_2|_V:V\rightarrow W
	\end{equation*}
	are slice-homeomorphisms, and
	\begin{equation*}
	f_1\circ\pi_1|_U^{-1}=f_2\circ\pi_2|_V^{-1}.
	\end{equation*}
	
	3. The equivalence class of a local slice regular function $(f,x)$ is denoted by $f_x$.
\end{defn}

\begin{defn}
	Let $\mathcal{G}=(G,\pi,x)$ be a slice-domain over $\mathbb{H}$ with distinguished point, and $f$ be slice regular functions on $G$. We say that $f$ can extend slice regularly to a slice-domain $\breve{\mathcal{G}}=(\breve{G},\breve{\pi},\breve{x})$ over $\mathbb{H}$ with distinguished point, if $\mathcal{G}\prec \breve{\mathcal{G}}$ and there exists a slice regular function $F$ on $\breve G$ such that $F|_{\breve G}=f$. We say that $\breve{\mathcal G}$ is $\{f\}$-extendible.
	
	Let $\mathscr{G}$ be the system of all $\{f\}$-extendible slice-domains over $\mathbb{H}$ with distinguished point. We call
	\begin{equation*}
	H_f(\mathcal{G}):=\bigcup_{\mathcal{\breve G}\in\mathscr{G}}\mathcal{\breve G}
	\end{equation*}
	the set of slice-domains of existence of the function $f$ with respect to $\mathcal{G}$.
\end{defn}

\begin{defn}
	A slice-domain $\mathcal{G}=(G,\pi,x)$ over $\mathbb{H}$ with distinguished point is called a slice-domain of (slice) regularity if there exists a slice regular function $f$ on $G$ such that $\mathcal{G}$ is a slice-domain of existence of $f$ with respect to $\mathcal{G}$.
\end{defn}

\subsection{Finite-part paths} In this subsection, we will recall some definitions with respect to finite-part paths.

A path in a topological space $X$ is a continuous function $f$ from the unit interval $[0,1]$ to $X$. For each $N\in\mathbb{N}^+$, topological space $X$ and paths $\gamma_\imath$, $\imath=1,2,...,N$ in $X$, we denote the composition of paths $\{\gamma_\imath\}_{\imath=1}^N$ by
\begin{equation*}
\prod_{l=1}^N \gamma_l=\gamma_1\gamma_2....\gamma_N,
\end{equation*}
i.e.,
\begin{equation*}
(\prod_{l=1}^N\gamma_l)(t):=\left\{
\begin{split}
&\gamma_{\lfloor tN\rfloor}(\{tN\}), &&t\in[0,1),
\\&\gamma(1), &&t=1.
\end{split}
\right.
\end{equation*}

\begin{defn}
	Let $U$ be a subset of $\mathbb{H}$. A slice-path $\gamma$ in $\mathbb{H}$ is called slice preserving path in $U$, if there exists $I\in\mathbb{S}$ such that $\gamma([0,1])\subset U_I$.
\end{defn}

\begin{defn}
	Let $(G,\pi)$ be a slice-domain over $\mathbb{H}$. A path $\gamma$ in $G$ is called slice preserving, if $\pi\circ\gamma$ is a slice preserving path in $\mathbb{H}$.
\end{defn}

\begin{defn}\label{df-2npp}
	Let $N\in\mathbb{N}^+$ and $\gamma_\imath$ be a path in $\mathbb{C}$ for each $\imath\in\{1,2,...,N\}$.
	\begin{equation*}
	\gamma:=(\gamma_1,\gamma_2,...,\gamma_N)
	\end{equation*}
	is called an $N$-part path in $\mathbb{C}$, if
	\begin{equation*}
	\gamma_\imath(1)=\gamma_{\imath+1}(0)\in\mathbb{R},\qquad \imath=1,2,...,N-1.
	\end{equation*}
	
	We call $\gamma_1(0)$ the initial point of $\gamma$.
	
	For each $z\in\mathbb{C}$ and $N\in\mathbb{N}^+$, we denote by $\mathcal{P}^{\infty}_z(\mathbb{C})$ (resp. $\mathcal{P}^{N}_z(\mathbb{C})$) the set of all the finite-part (resp. $N$-part) paths in $\mathbb{C}$ with the initial point $z$.
\end{defn}

For each slice-domain $(G,\pi)$ over $\mathbb{H}$ (resp. slice-domain $(G,\pi,x)$ over $\mathbb{H}$ with distinguished point), $I\in\mathbb{S}$, and $U\subset G$, we set
\begin{equation*}
U_I:=\{q\in U:\pi(q)\in\mathbb{C}_I\}\qquad\mbox{and}\qquad U_\mathbb{R}:=\{q\in U:\pi(q)\in\mathbb{R}\}.
\end{equation*}

\begin{defn}
	Let $(G,\pi)$ be a slice-domain over $\mathbb{H}$, $N\in\mathbb{N}^+$ and $\gamma_\imath$ be a slice preserving path in $\mathbb{H}$ (resp. $G$) for each $\imath\in\{1,2,...,N\}$.
	\begin{equation*}
	\gamma:=(\gamma_1,\gamma_2,...,\gamma_N)
	\end{equation*}
	is called an $N$-part path in $\mathbb{H}$ (resp. $G$), if
	\begin{equation*}
	\gamma_\imath(1)=\gamma_{\imath+1}(0)\in\mathbb{R}\qquad(\mbox{resp.}\ \gamma_\imath(1)=\gamma_{\imath+1}(0)\in G_{\mathbb{R}}),\qquad \imath=1,2,...,N-1.
	\end{equation*}
\end{defn}

\begin{defn}
	Let $(G,\pi)$ be a slice-domain over $\mathbb{H}$, $N\in\mathbb{N}^+$ and $\gamma$ be the $N$-part path in $\mathbb{C}$ (resp. $\mathbb{H}$ or $G$). We set
	\begin{equation*}
	\gamma(t):=\left\{
	\begin{aligned}
	&\gamma_{\lfloor tN\rfloor+1}(\{tN\}),\qquad &&t\in[0,1),
	\\&\gamma_N(1), &&t=1.
	\end{aligned}
	\right.
	\end{equation*}
	We call $\gamma$ is from $\gamma(0)$ to $\gamma(1)$.
	
	The set of all $N$-part paths in $\mathbb{C}$ (resp. $\mathbb{H}$ or $G$) is denoted by $\mathcal{P}^N(\mathbb{C})$ (resp. $\mathcal{P}^N(\mathbb{H})$ or $\mathcal{P}^N(G)$). We define the set of finite-part paths in $\mathbb{C}$ (resp. $\mathbb{H}$ or $G$), by
	\begin{equation*}
	\mathcal{P}^{\infty}(\mathbb{C}):=\bigsqcup_{\imath\in\mathbb{N}^+} \mathcal{P}^\imath(\mathbb{C})\quad
	(\mbox{resp.}\ 	\mathcal{P}^{\infty}(\mathbb{H}):=\bigsqcup_{\imath\in\mathbb{N}^+} \mathcal{P}^\imath(\mathbb{H})\ or\ 
	\mathcal{P}^{\infty}(G):=\bigsqcup_{\imath\in\mathbb{N}^+} \mathcal{P}^\imath(G)).
	\end{equation*}
\end{defn}

\begin{defn}\label{df:2lni}
For each $N\in\mathbb{N}^+$, $I=(I_1,I_2,...I_N)\in\mathbb{S}^N$ and $\gamma=(\gamma_1,\gamma_2,...,\gamma_N)\in\mathcal{P}^N(\mathbb{C})$, we define a map
\begin{equation*}
\phi_I:\bigsqcup_{\imath\in\mathbb{N}^+}\mathcal{P}^\imath (\mathbb{C})\rightarrow\bigsqcup_{\imath\in\mathbb{N}^+}\mathcal{P}^\imath (\mathbb{H}),\qquad\gamma\mapsto(P_{I_1} (\gamma_1),P_{I_2} (\gamma_2),...,P_{I_N} (\gamma_N)).
\end{equation*}
We call $\phi_I(\gamma)$ the $I$-lifting of $\gamma$ to $\mathbb{H}$, denoted by $\gamma^I$.
\end{defn}

Let $(G,\pi)$ is a slice-domain over $\mathbb{H}$, we define a map $\pi:\mathcal{P}^{\infty} (G)\rightarrow\mathcal{P}^{\infty} (\mathbb H)$ by
\begin{equation*}
\pi(\beta):=(\pi(\beta_1),\pi(\beta_2),...,\pi(\beta_m)),\qquad\forall\ m\in\mathbb{N}^+\ \mbox{and}\ \beta\in\mathcal{P}^m(G).
\end{equation*}

\begin{defn}
	Let $N\in\mathbb{N}^+$, $\gamma$ be an $N$-part path (resp. path) in $\mathbb{H}$ and $\mathcal{G}=(G,\pi,x)$ be a slice-domain over $\mathbb{H}$ with distinguished point. We say that $\gamma$ is contained in $\mathcal{G}$ (denoted by $\gamma\prec\mathcal{G}$), if there exists an (unique) $N$-part path (resp. path) $\alpha$ in $G$ such that $\alpha(0)=x$ and $\pi(\alpha)=\gamma$.
	
	We call $\alpha$ the lifting of $\gamma$ to $\mathcal{G}$, denoted by $\gamma_{\mathcal{G}}$.
\end{defn}

Let $\mathcal{G}$ be a slice-domain over $\mathbb{H}$ with distinguished point, $N\in\mathbb{N}^+$, and $\gamma=(\gamma_1,\gamma_2,...,\gamma_N)$ be an $N$-part path in $\mathbb{C}$ (resp. $\mathbb{H}$ or $\mathcal{G}$). For each $t\in[0,1]$, we define a finite-part path $\gamma[t]$ in $\mathbb{C}$ (resp. $\mathbb{H}$ or $\mathcal{G}$) by
\begin{equation*}
\gamma[t]:=\left\{
\begin{split}
&(\gamma_1,\gamma_2,...,\gamma_{\lfloor Nt\rfloor},\gamma_{(Nt)}),\qquad&&t\in[0,1),\\
&\gamma,&&t=1,\\
\end{split}\right.
\end{equation*}
where $\gamma_{(Nt)}$ is the path in $\mathbb{C}$ (resp. $\mathbb{H}$ or $\mathcal{G}$), defined by
\begin{equation*}
\gamma_{(Nt)}(s):=\left\{
\begin{split}
&\gamma_{\lfloor Nt+1\rfloor}((Nt-\lfloor Nt\rfloor)s),\qquad&&Nt\neq\lfloor Nt\rfloor,\\
&\gamma(t),&&otherwise,\\
\end{split}\right.\qquad\forall\ s\in[0,1].
\end{equation*}
Let $\gamma[t^-]$ be a finite-part path in $\mathbb{C}$, defined by
\begin{equation*}
\gamma[t^-]:=\left\{
\begin{split}
&(\gamma_1,\gamma_2,...,\gamma_{tN}),\qquad&&t\in\{\frac{1}{N},\frac{2}{N},...,\frac{N}{N}\},\\
&\gamma[t],&&t\in[0,1]\backslash\{\frac{1}{N},\frac{2}{N},...,\frac{N}{N}\}.\\
\end{split}\right.
\end{equation*}

\subsection{Technical notations}

In this subsection, we will recall some technical notations.

For each $N\in\mathbb{N}^+$, $I=(I_1,I_2,...,I_N)\in\mathbb{S}^N$ and $m\in\{1,2,...,2^N\}$, we set
\begin{equation}\label{eq-2im}
I(m):=\prod_{\imath=N}^1 (I_\imath I_{\imath-1})^{m_\imath}=(I_N I_{N-1})^{m_N}(I_{N-1} I_{N-2})^{m_{N-1}}...(I_{1} I_{0})^{m_{1}},
\end{equation}
where $I_{0}:=1$ and $(m_N m_{N-1} ... m_1)_2$ is the binary number of $m-1$.

We define a map
\begin{equation*}
\zeta:\bigsqcup_{\imath\in\mathbb{N}}\mathbb{S}^\imath\rightarrow\bigsqcup_{\imath\in\mathbb{N}}\mathbb{H}^{2^\imath},\qquad I\mapsto(I(1),I(2),...,I(2^\jmath)),\qquad\forall\ \jmath\in\mathbb{N}^+\ \mbox{and}\ I\in\mathbb{S}^\jmath.
\end{equation*}

For each set $A$ and $\imath,\jmath\in\mathbb{N}^+$, we denote the set of all $\imath\times\jmath$ matrices of $A$ by $M_{\imath\times\jmath}(A)$, and the set of all $\imath\times\imath$ matrices of $A$ by $M_\imath(A)$. We denote the $\imath\times\imath$ identity matrix by $\mathbb{I}_\imath$, and the $\imath\times\imath$ zero matrix by $0_\imath$ for each $\imath\in\mathbb{N}^+$.

For each matrix $E$, we denote the transpose of $E$ by $E^T$.

For each $\imath\in\mathbb{N}^+$, we say that $A\in M_\imath(\mathbb{H})$ is invertible, if there exists a matrix $B\in M_\imath(\mathbb{H})$, such that $AB=BA=\mathbb{I}_\imath$.

For each $n,m\in\mathbb{N}^+$, $A=\{a_{\imath,\jmath}\}_{n\times m}\in M_{n\times m}(\mathbb{H})$ and $q\in\mathbb{H}$, we set
\begin{equation*}
qA:=\{q\cdot a_{\imath,\jmath}\}_{n\times m}\qquad\mbox{and}\qquad Aq:=\{a_{\imath,\jmath}\cdot q\}_{n\times m}.
\end{equation*}

\begin{defn}\label{df-2fsr}
	Let $N\in\mathbb{N}^+$ and $J=(J_{\imath,\jmath})_{2^N\times N}\in M_{2^N\times N}(\mathbb{S})$. We call $J$ has full slice-rank, if $\mathcal{M}(J^{(\imath)})$ is invertible for each $\imath\in\{1,2,...,N\}$
\end{defn}

For each $N\in\mathbb{N}^+$ and $J=\{J_{\imath,\jmath}\}_{2^N\times N}\in M_{2^N\times N}(\mathbb{S})$, we denote the $\imath$-th row vector of $J$ by
\begin{equation*}
J_\imath:=(J_{\imath,1},J_{\imath,2},...,J_{\imath,N})\in\mathbb{S}^N,\qquad\forall\ \imath\in\{1,2,...,2^N\}.
\end{equation*}
For each $l\in\{1,2,...,N\}$ and $\imath\in\{1,2,...,2^N\}$, we set
\begin{equation*}
J^{(l)}:=\{J_{\imath,\jmath}\}_{2^l\times l}\qquad\mbox{and}\qquad J_\imath^{(l)}:=(J_{\imath,1},J_{\imath,2},...,J_{\imath,l}).
\end{equation*}
We define a map
\begin{equation*}
\mathcal{M}:\bigsqcup_{\imath\in\mathbb{N}^+}M_{2^\imath\times\imath}(\mathbb{S})\rightarrow \bigsqcup_{\imath\in\mathbb{N}^+}M_{2^\imath}(\mathbb{H})
\end{equation*}
by
\begin{equation*}
\mathcal{M}(K):=(\zeta(K_1)^T,\zeta(K_2)^T,...,\zeta(K_{2^\jmath})^T)^T
\end{equation*}
for each $\jmath\in\{1,2,...,N\}$ and $K\in M_{2^\jmath\times\jmath}(\mathbb{S})$.

\section{Slice-linearly independent}\label{sc-2sli}

The representation formula over slice-domains of regularity \cite[Theorem \ref{th-dere}]{Dou2018001} demands $J\in M_{2^N\times N}(\mathbb{S})$ with full slice-rank (see Definition \ref{df-2fsr}). In this section, we prove Proposition \ref{pr-tfb} to provide a weak condition, so-called left slice-linearly independent (see Definition \ref{df-sli}), for a general representation formula over Riemann slice-domains (see Theorem \ref{tm-rf}).

\begin{defn}
	Let $N\in\mathbb{N}^+$, and $v_\imath\in\mathbb{H}^N$, $\imath=1,2,...,N$. The set of vectors $\{v_\imath\}_{\imath=1}^N$ is called left (resp. right) linearly independent, if the equation
	\begin{equation*}
	\sum_{\imath=1}^N q_\imath v_\imath=0\qquad (\mbox{resp.}\ \sum_{\imath=1}^N v_\imath q_\imath=0)
	\end{equation*}
	can only be satisfied by $q_\imath=0$, $\imath=1,2,...,N$, where $q_\imath\in\mathbb{H}$, $\imath=1,2,...,N$.
\end{defn}

The rank of a quaternion matrix $A$ is defined to be the maximum number of columns of $A$ which are right linearly independent.

\begin{prop}\label{pr-qmi}
	(\cite[Page 43]{Zhang1997001})
	If a quaternion matrix $A$ is of rank $\imath\in\mathbb{N}^+$, then $\imath\in\mathbb{N}^+$ is also the maximum number of rows of $A$ that are left linearly independent.
\end{prop}

\begin{prop}\label{pr-aii}
	(\cite[Page 43]{Zhang1997001})
	Let $\imath\in\mathbb{N}^+$, if $A\in M_{\imath\times\imath}(\mathbb{H})$, then $A$ is invertible if and only if $A$ is of (full) rank $\imath$.
\end{prop}

\begin{defn}\label{df-sli}
	We say that $J\in M_{2^N\times N}(\mathbb{S})$ is left (resp. right) slice-linearly independent, if $\{\zeta(J_\imath)\}_{\imath=1}^{2^N}$ are left (resp. right) linearly independent. 
\end{defn}

\begin{prop}\label{pr-tfb}
	Let $N\in\mathbb{N}^+$, and $J\in M_{2^N\times N}(\mathbb{S})$. If $J\in M_{2^N\times N}(\mathbb{S})$ is left slice-linearly independent, then there exists a permutation $\sigma$ in the symmetric group of degree $2^N$, $S_{2^N}$, such that
	\begin{equation*}
	((J_{\sigma(1)})^T,(J_{\sigma(2)})^T,...,(J_{\sigma(2^N)})^T)^T
	\end{equation*}
	has full slice-rank.
\end{prop}

\begin{proof}
	We will prove this proposition by induction. Obviously, Proposition \ref{pr-tfb} holds for $N=1$. If Proposition \ref{pr-tfb} holds for $N=k$ with $k\in\mathbb{N}^+$, we will prove that Proposition \ref{pr-tfb} also holds for $N=k+1$.
	
	Since $J$ is left slice-linearly independent, $\{\zeta(J_\imath)\}_{\imath=1}^{2^{k+1}}$ (the rows of  $\mathcal{M}(J)$) are left linearly independent. And thanks to Propositions \ref{pr-qmi} and \ref{pr-aii}, the rank of $\mathcal{M}(J)$ is $2^{k+1}$ (full rank) and $\mathcal{M}(J)$ is invertible. Therefore the columns of $\mathcal{M}(J)$ are right linearly independent, and the first $2^K$ columns of $\mathcal{M}(J)$ are also right linearly independent. It follows that, the rank of
	\begin{equation*}
	A:=(\zeta(J_1^{(k)})^T,\zeta(J_2^{(k)})^T,...,\zeta(J_{2^{k+1}}^{(k)})^T)^T
	\end{equation*}
	is $2^K$. And according to Proposition \ref{pr-qmi}, the maximum number of rows of $A$, that are left linearly independent, is $2^K$. Thence there exists a permutation $\rho$ in the symmetric group of degree $2^{k+1}$, such that the rows of
	\begin{equation*}
	B:=(\zeta(J_{\rho(1)}^{(k)})^T,\zeta(J_{\rho(2)}^{(k)})^T,...,\zeta(J_{\rho(2^k)}^{(k)})^T)
	\end{equation*}
	are left linearly independent. It follows that
	\begin{equation*}
	((J_{\rho(1)}^{(k)})^T,(J_{\rho(2)}^{(k)})^T,...,(J_{\rho(2^{k})}^{(k)})^T)^T
	\end{equation*}
	is left slice-linearly independent. By induction hypothesis, there exists a permutation $\omega$ in the symmetric group of degree $2^{k+1}$ such that
	\begin{equation*}
	((J_{\omega\circ\rho(1)}^{(k)})^T,(J_{\omega\circ\rho(2)}^{(k)})^T,...,(J_{\omega\circ\rho(2^{k})}^{(k)})^T)^T
	\end{equation*}
	has full slice-rank with $w(\imath)=\imath$ for each $\imath\in\{2^k+1,2^K+2,...,2^{k+1}\}$. Then
	\begin{equation*}
	\sigma=\omega\circ\rho
	\end{equation*}
	is the permutation in the symmetric group of degree $2^{k+1}$ such that
	\begin{equation*}
	((J_{\sigma(1)})^T,(J_{\sigma(2)})^T,...,(J_{\sigma(2^{k+1})})^T)^T
	\end{equation*}
	has full slice-rank.
\end{proof}

\section{Representation Formula}\label{sc-2rf}

In this section, we generalize the representation formula \cite[Theorem \ref{th-dere}]{Dou2018001} to Riemann slice-domains. We discover an invariant from the representation formula (see Theorem \ref{tm-rf}), which does not depend on the choice of the left slice-linearly independent matrix. This invariant will be treated as a stem holomorphic function in Section \ref{sc-hss} for introducing the $*$-product over some suitable Riemann slice domains. And then we get some extension results of slice regular function on Riemann slice-domains over $\mathbb{H}$, which indicate some symmetry of slice-domains of existence.

\begin{thm}\label{tm-rf} 
	(Representation Formula)
	Let $N\in\mathbb{N}^+$, $J\in M_{2^N\times N}(\mathbb{S})$ be a left slice-linearly independent matrix, $\mathcal{G}=(G,\pi,x)$ be a slice-domain over $\mathbb{H}$ with distinguished point with $\pi(x)\in\mathbb{R}$, and $\gamma$ be an N-part path in $\mathbb{C}$. If
	\begin{equation*}
	\gamma^{J_\imath}\prec\mathcal{G},\qquad\imath=1,2,...,2^N,
	\end{equation*}
	$f$ is a slice regular function on $G$, and $K\in\mathbb{S}^N$ with $\gamma^K\prec\mathcal{G}$, then
	\begin{equation}\label{eq-mrk}
	f(\gamma^K_{\mathcal{G}}(1))=\zeta(K)\mathcal{M}(J)^{-1}f(\gamma^J_{\mathcal{G}}(1)),
	\end{equation}
	where
	\begin{equation*}
	f(\gamma^J_{\mathcal{G}}(1)):=(f(\gamma^{J_1}_{\mathcal{G}}(1)),f(\gamma^{J_2}_{\mathcal{G}}(1)),...,\gamma^{J_{2^N}}_{\mathcal{G}}(1))^T.
	\end{equation*}
	
	Moreover,
	\begin{equation}\label{eq-gfg}
	{\mathcal{G}}^f_{\gamma}:=\mathcal{M}(J)^{-1}f(\gamma^J_{\mathcal{G}}(1))\in M_{1\times 2^N}(\mathbb{H})
	\end{equation}
	does not depend on the choice of the left slice-linearly independent matrix $J\in M_{2^N\times N}(\mathbb{S})$ with $\gamma^{J_\imath}\prec\mathcal{G}$, $\imath=1,2,...,2^N$.
\end{thm}

\begin{proof}
	1. Thanks to Proposition \ref{pr-tfb}, there exists $\sigma\in S_{2^N}$ such that
	\begin{equation*}
	J_\sigma:=(J_{\sigma(1)}^T,J_{\sigma(2)}^T,...,J_{\sigma(2^N)}^T)^T
	\end{equation*}
	has full slice-rank. Let
	\begin{equation*}
	\widehat{\mathcal{G}}:=(\widehat{G},\widehat{\pi},\widehat{x})
	\end{equation*}
	be a slice-domain of existence of $f$ with respect $\mathcal{G}$. According to \cite[Theorem \ref{th-exsd}]{Dou2018001}, we have
	\begin{equation*}
	\mathcal{G}\prec\widehat{\mathcal{G}}\qquad\mbox{and}\qquad\widetilde{f}|_G=f.
	\end{equation*}
	Let $\varphi:G\rightarrow \widehat{G}$ be the fiber preserving map from $\mathcal{G}$ to $\widehat{\mathcal{G}}$. We notice that
	\begin{equation}\label{eq-rkg}
	\gamma^K\prec\widehat{\mathcal{G}},\qquad\varphi(\gamma^K_{\mathcal{G}})=\gamma^K_{\widehat{\mathcal{G}}}\qquad\mbox{and}\qquad f(\gamma^K_{\mathcal{G}}(1))=\widetilde{f}(\varphi(\gamma^K_{\mathcal{G}}(1)))=\widetilde{f}(\gamma^K_{\widehat{\mathcal{G}}}(1))
	\end{equation}
	for each $K\in\mathbb{S}^N$, and thanks to \cite[Theorem \ref{th-dere}]{Dou2018001}, it follows that
	\begin{equation}\label{eq-mjf}
	\begin{split}
	&\mathcal{M}(J)^{-1}f(\gamma^{J}_{\mathcal{G}}(1))
	\\=&\mathcal{M}(J)^{-1}(f(\gamma^{J_1}_{\mathcal{G}}(1)),f(\gamma^{J_2}_{\mathcal{G}}(1)),...,f(\gamma^{J_{2^N}}_{\mathcal{G}}(1)))^T
	\\=&\mathcal{M}(J)^{-1}(\widetilde{f}(\gamma^{J_1}_{\widehat{\mathcal{G}}}(1)),\widetilde{f}(\gamma^{J_2}_{\widehat{\mathcal{G}}}(1)),...,\widetilde{f}(\gamma^{J_{2^N}}_{\widehat{\mathcal{G}}}(1)))^T
	\\=&\mathcal{M}(J)^{-1}(\zeta(J_1)^T,\zeta(J_2)^T,...,\zeta(J_{2^N})^T)^T\mathcal{M}(J_\sigma)^{-1}\widetilde{f}(\gamma^{J_\sigma}_{\widehat{\mathcal{G}}}(1))
	\\=&\mathcal{M}(J)^{-1}\mathcal{M}(J)\mathcal{M}(J_\sigma)^{-1}\widetilde{f}(\gamma^{J_\sigma}_{\widehat{\mathcal{G}}}(1))
	\\=&\mathcal{M}(J_\sigma)^{-1}\widetilde{f}(\gamma^{J_\sigma}_{\widehat{\mathcal{G}}}(1)),
	\end{split}
	\end{equation}
	where
	\begin{equation*}
	J_{\sigma}:=(J_{\sigma(1)}^T,J_{\sigma(2)}^T,...,J_{\sigma(2^N)}^T)^T
	\end{equation*}
	and
	\begin{equation*}
	\widetilde{f}(\gamma^{J_\sigma}_{\widehat{\mathcal{G}}}(1)):=(\widetilde{f}(\gamma^{J_{\sigma(1)}}_{\widehat{\mathcal{G}}}(1)),\widetilde{f}(\gamma^{J_{\sigma(2)}}_{\widehat{\mathcal{G}}}(1)),...,\widetilde{f}(\gamma^{J_{\sigma(2^N)}}_{\widehat{\mathcal{G}}}(1)))^T.
	\end{equation*}
	According to (\ref{eq-rkg}), (\ref{eq-mjf}) and \cite[Theorem \ref{th-dere}]{Dou2018001}, we have
	\begin{equation*}
	\begin{split}
	f(\gamma^K_{\mathcal{G}}(1))=&\widetilde{f}(\gamma^K_{\widehat{\mathcal{G}}}(1))
	\\=&\zeta(K)\mathcal{M}(J_\sigma)^{-1}\widetilde{f}(\gamma^{J_\sigma}_{\widehat{\mathcal{G}}}(1))
	\\=&\zeta(K)\mathcal{M}(J)^{-1}f(\gamma^J_{\mathcal{G}}(1)).
	\end{split}
	\end{equation*}
	Thence (\ref{eq-mrk}) holds.
	
	2. For each left slice-linearly independent matrix $I\in M_{2^N\times N}(\mathbb{S})$ with
	\begin{equation*}
	\gamma^{I_\imath}\prec\mathcal{G},\qquad\imath=1,2,...,2^N.
	\end{equation*}
	We notice that
	\begin{equation*}
	\begin{split}
	&\mathcal{M}(I)^{-1}f(\gamma^{I}_{\mathcal{G}}(1))
	\\=&\mathcal{M}(I)^{-1}(f(\gamma^{I_1}_{\mathcal{G}}(1)),f(\gamma^{I_2}_{\mathcal{G}}(1)),...,f(\gamma^{I_{2^N}}_{\mathcal{G}}(1)))^T
	\\=&\mathcal{M}(I)^{-1}(\zeta(I_1)^T,\zeta(I_2)^T,...,\zeta(I_{2^N})^T)^T\mathcal{M}(J)^{-1}f(\gamma^{J}_{\mathcal{G}}(1))
	\\=&\mathcal{M}(J)^{-1}f(\gamma^J_{\mathcal{G}}(1)),
	\end{split}
	\end{equation*}
	where
	\begin{equation*}
	f(\gamma^I_{\mathcal{G}}(1)):=(f(\gamma^{I_1}_{\mathcal{G}}(1)),f(\gamma^{I_2}_{\mathcal{G}}(1)),...,\gamma^{I_{2^N}}_{\mathcal{G}}(1))^T.
	\end{equation*}
	It is clear that $\mathcal{G}^f_\gamma$ does not depend on the choice of the left slice-linearly independent matrix $J\in M_{2^N\times N}(\mathbb{S})$ with $\gamma^{J_\imath}\prec\mathcal{G}$, $\imath=1,2,...,2^N$.
\end{proof}

\begin{defn}
	Let $\mathcal{G}=(G,\pi,x)$ be a slice-domain over $\mathbb{H}$ with distinguished point. $\mathcal{G}$ is called slice-symmetric with respect to distinguished point, if the following properties hold:
	\begin{enumerate}
		\item $\pi(x)\in\mathbb{R}$,
		
		\item If there exist $N\in\mathbb{N}^+$, an $N$-part path $\gamma$ in $\mathbb{C}$ and a left slice-linearly independent matrix $J\in M_{2^N\times N}(\mathbb{S})$, such that
		\begin{equation*}
		\gamma^{J_\imath}\prec\mathcal{G},\qquad\imath=1,2,...,2^N,
		\end{equation*}
		then
		\begin{equation*}
		\gamma^K\prec\mathcal{G},\qquad\forall\ K\in\mathbb{S}^N.
		\end{equation*}
	\end{enumerate}
	
	The slice-domain $(G,\pi)$ is called slice-symmetric if $(G,\pi,y)$ is slice-symmetric with respect to distinguished point, for each $y\in G_{\mathbb{R}}$.
	
	$\mathcal{G}$ is called slice-symmetric, if $(G,\pi)$ is slice-symmetric.
\end{defn}

\begin{rmk}
	``slice-symmetric with respect to distinguished point" and ``slice-symmetric" are different. Let $\mathcal{G}=(G,\pi,x)$ be a slice-domain over $\mathbb{H}$ with distinguished point.
	
	If $\mathcal{G}$ is slice-symmetric with respect to distinguished point, this does not lead to $(G,\pi,y)$ is also slice-symmetric with respect to distinguished point for each $y\in G_\mathbb{R}$.
	
	On the other hand, if $\mathcal{G}=(G,\pi,x)$ is slice-symmetric, then for each $y\in G_\mathbb{R}$, $(G,\pi,y)$ is also slice-symmetric with respect to distinguished point. However, if $y\notin G_\mathbb{R}$, $(G,\pi,y)$ is not slice-symmetric with respect to distinguished point.
\end{rmk}

\begin{prop}\label{pr-sdhss}
	All the slice-domain of regularity are slice-symmetric.
\end{prop}

\begin{proof}
	Let $\mathcal{G}=(G,\pi,x)$ is a slice-domain of regularity. According to \cite[Proposition \ref{pr-tsh}]{Dou2018001}, $(G,\pi,y)$ is also a slice-domain of regularity for each $y\in G_\mathbb{R}$. Thanks to Proposition \ref{pr-tfb} and \cite[Theorem \ref{th-dere}]{Dou2018001}, $(G,\pi,y)$ is slice-symmetric with respect to distinguished point. Thus $(G,\pi)$ and $\mathcal{G}$ is slice-symmetric.
\end{proof}

\begin{cor}
	Let $\mathcal{G}=(G,\pi,x)$ be a slice-domain over $\mathbb{H}$ with distinguished point, and $f$ be a slice-regular function on $G$. Then $f$ can extend slice regularly to each slice-domain of existence of $f$ with respect to $\mathcal{G}$, which is slice-symmetric.
\end{cor}

\begin{proof}
	This corollary follows immediately from Proposition \ref{pr-sdhss}.
\end{proof}

\section{Examples of slice-domains of existence}\label{sc-2ese}

Now, we describe two examples, slice-domains of existence of $x^{\frac{1}{2}}$ and $ln(x)$ to show that how the Representation Formula \ref{tm-rf} works. Before that, we talk about some notations (see Proposition \ref{pr-umm}) and a proposition of slice-domains of existence (see Proposition \ref{pr-srfid}).

For each $N\in\mathbb{N}^+$ and $m\in\{1,2,...,2^N\}$, we define a map
\begin{equation*}
\eta_N^m:\mathbb{S}\rightarrow\mathbb{S}^N
\end{equation*}
by
\begin{equation*}\label{eq-mni}
\eta_N^m(I):=((-1)^{m_N}I,(-1)^{m_{N-1}}I,...,(-1)^{m_1}I),\qquad\forall\ I\in\mathbb{S},
\end{equation*}
where $(m_Nm_{N-1}...m_1)_2$ is the binary number of $m-1$.

For each $N\in\mathbb{N}^+$, we define a map
\begin{equation*}
\eta_N:\mathbb{S}\rightarrow M_{2^N\times N}(\mathbb{S})
\end{equation*}
by
\begin{equation*}
\eta_N(I):=(\eta_N^1(I)^T,\eta_N^2(I)^T,...,\eta_N^{2^N}(I)^T)^T,\qquad\forall I\in\mathbb{S}.
\end{equation*}
Then the following proposition holds.

\begin{prop}\label{pr-umm}
	For each $N\in\mathbb{N}^+$ and $I\in\mathbb{S}$, the matrix $2^{-\frac{N}{2}}\mathcal{M}(\eta_N(I))$ in $M_{2^N}(\mathbb{C}_I)$ is unitary, i.e.,
	\begin{equation}\label{eq-mqi}
	\mathcal{M}(\eta_N(I))^{-1}=\frac{1}{2^N}\overline{\mathcal{M}(\eta_N(I))^T}.
	\end{equation}
\end{prop}

\begin{proof}
	We will prove this proposition by induction. When $N=1$, we have
	\begin{equation*}
	\frac{1}{2^1}\left(\begin{matrix}1&I\\1&-I\end{matrix}\right)
	\left(\begin{matrix}1&1\\-I&I\end{matrix}\right)
	=\left(\begin{matrix}1&0\\0&1\end{matrix}\right)
	\end{equation*}
	for each $I\in\mathbb{S}$. Then (\ref{eq-mqi}) holds.
	
	If (\ref{eq-mqi}) holds when $N=m$, where $m\in\mathbb{N}^+$ and $m\ge1$. We will prove that (\ref{eq-mqi}) holds in the case of $N=m+1$. For each $I\in\mathbb{S}$, there exist matrices $A,B\in M_{2^{m-1}\times 2^m}(\mathbb{H})$ such that
	\begin{equation*}
	\mathcal{M}(\eta_m(I))=(A^T,B^T)^T.
	\end{equation*}
	By induction hypothesis, we have:
	\begin{equation*}
	\frac{1}{2^m}\left(\begin{matrix}A\\B\end{matrix}\right)
	\left(\begin{matrix}\overline{A^T}&\overline{B^T}\end{matrix}\right)
	=\mathbb{I}_{2^m}.
	\end{equation*}
	It follows that
	\begin{equation}\label{eq-2abi}
	A\overline{A^T}=B\overline{B^T}=2^m\mathbb{I}_{2^{m-1}}\qquad\mbox{and}\qquad A\overline{B^T}=B\overline{A^T}=0_{2^{m-1}}.
	\end{equation}
	According to (\ref{eq-2im}) and (\ref{eq-mni}), we have
	\begin{equation*}
	\eta_{m+1}^\imath(\jmath)=\eta_{m+1}^{\imath+2^m}(\jmath)=-\eta_{m+1}^\imath(\jmath+2^m)=\eta_{m+1}^{\imath+2^m}(\jmath+2^m),
	\end{equation*}
	and
	\begin{equation*}
	\eta_{m+1}^{\imath+2^{m-1}}(\jmath)=\eta_{m+1}^{\imath+2^m+2^{m-1}}(\jmath)=\eta_{m+1}^{\imath+2^{m-1}}(\jmath+2^m)=-\eta_{m+1}^{\imath+2^m+2^{m-1}}(\jmath+2^m)
	\end{equation*}
	for each $\imath\in\{1,2,...,2^{m-1}\}$ and $\jmath\in\{1,2,...,2^m\}$. 
	Then
	\begin{equation*}
	\mathcal{M}(\eta_{m+1}(I))=\left(\begin{matrix}A&-A\\B&B\\A&A\\B&-B\end{matrix}\right).
	\end{equation*}
	Thanks to (\ref{eq-2abi}),
	\begin{equation*}
	\frac{1}{2^{m+1}}\left(\begin{matrix}A&-A\\B&B\\A&A\\B&-B\end{matrix}\right)
	\left(\begin{matrix}\overline{A^T}&\overline{B^T}&\overline{A^T}&\overline{B^T}\\-\overline{A^T}&\overline{B^T}&\overline{A^T}&-\overline{B^T}\end{matrix}\right)
	=\mathbb{I}_{2^{m+1}}.
	\end{equation*}
	It follows that
	\begin{equation*}
	2^{-\frac{m+1}{2}}\mathcal{M}(\eta_N(I))
	\end{equation*}
	is an unitary matrix for each $I\in\mathbb{S}$. Hence
	\begin{equation*}
	2^{-\frac{N}{2}}\mathcal{M}(\eta_N(I))
	\end{equation*}
	is an unitary matrix for each $N\in\mathbb{N}^+$ and $I\in\mathbb{S}$, by induction.
\end{proof}

\begin{prop}\label{pr-srfid}
	Let $\mathcal{G}=(G,\pi,x)$ be a slice-domain over $\mathbb{H}$ with distinguished point, $p,q\in G$ with $\pi(p)=\pi(q)$, and $f$ be a slice regular function on $G$. If $\mathcal{G}$ is a slice-domain of existence of $f$ with respect to $\mathcal{G}$, and $f_p=f_q$, then $p=q$.
\end{prop}

\begin{proof}
	We define an equivalence $\sim$ on $G$ by $p'\sim q'$, if and only if, $\pi(p')=\pi(q')$ and $f_{p'}=f_{q'}$. We set
	\begin{equation*}
	\widehat{G}:=G/\sim,
	\end{equation*}
	and let
	\begin{equation*}
	\varphi:G\rightarrow\widehat{G}
	\end{equation*}
	be the quotient map of $\sim$. Let the topology of $\widehat{G}$ be the quotient space topology, then $\varphi$ is continuous. since
	\begin{equation*}
	\pi(p_1)=\pi(p_2),\qquad\forall\ q'\in G\ \mbox{and}\ p_1,p_2\in\varphi^{-1}(q'),
	\end{equation*}
	it follows that there exists a map $\widehat{\pi}:\widehat{G}\rightarrow\mathbb{H}$, such that
	\begin{equation}\label{eq-2pqp}
	\widehat{\pi}(q')=\pi(p'),\qquad\forall\ q'\in G\ \mbox{and}\ p'\in\varphi^{-1}(q').
	\end{equation}
	Similarly, there exists a map $\widehat{f}:\widehat{G}\rightarrow\mathbb{H}$, such that
	\begin{equation}\label{eq-2fqp}
	\widehat{f}(q')=f(p'),\qquad\forall\ q'\in G\ \mbox{and}\ p'\in\varphi^{-1}(q').
	\end{equation}
	We set
	\begin{equation}\label{eq-2xfx}
	\widehat{x}:=\varphi(x).
	\end{equation}
	We will prove that
	\begin{equation*}
	\widehat{\mathcal{G}}:=(\widehat{G},\widehat{\pi},\widehat{x})
	\end{equation*}
	is a slice-domain of existence of $f$ (following the proof of \cite[Theorem \ref{th-union}]{Dou2018001}), and
	\begin{equation*}
	f|_{\widehat{G}}=\widehat{f}.
	\end{equation*}
	
	1). For each $p'\in\widehat{G}$, there exists $q'\in\varphi^{-1}(p')$. Since $G$ is connected, there exists a path $\gamma$ in $G$ from $x$ to $q'$. Because $\varphi$ is continuous, $\varphi\circ\gamma$ is a path in $\widehat{G}$ from $\widehat{x}$ to $p'$. Then $\widehat{G}$ is path-connected. It is also connected.
	
	2). For each $p'\in\widehat{G}$, there exists $q'\in\varphi^{-1}(p')$. There exists a domain $U$ in $G$ containing $q'$ such that $\pi|_U:U\rightarrow\pi(U)$ is a slice-homeomorphism. We notice that $\varphi|_U:U\rightarrow\varphi(U)$ is a continuous bijection. We will prove that $\varphi|_U^{-1}$ is continuous.
	
	For each domain $V$ in $U$, we set
	\begin{equation*}
	V':=\varphi(V).
	\end{equation*}
	For each $y\in\varphi^{-1}(V')$, there exists a domain $W$ in $G$ containing $y$, such that $\pi|_W:W\rightarrow\pi(W)$ is a slice-homeomorphism and $\pi(W)\subset\pi(V)$. We denote the unique element in $\varphi|_U^{-1}\circ\varphi(y)$ by $y'$. Since
	\begin{equation*}
	\varphi(y')=\varphi(\varphi|_U^{-1}\circ\varphi(y))=\varphi(y),
	\end{equation*}
	it is clear that
	\begin{equation*}
	y\sim y'.
	\end{equation*}
	According to $f_y=f_{y'}$ and the identity principle (see \cite[Theorem \ref{pt-ipsd}]{Dou2018001}), we have
	\begin{equation*}
	f\circ\pi|_U^{-1}=f\circ\pi|_W^{-1}\qquad\mbox{on}\qquad\pi(W).
	\end{equation*}
	Thence
	\begin{equation*}
	f_{\pi|_U^{-1}(z)}=f_{\pi|_W^{-1}(z)},\qquad\forall\ z\in\pi(W).
	\end{equation*}
	It follows that
	\begin{equation*}
	\pi|_U^{-1}(z)\sim\pi|_W^{-1}(z)\qquad\forall\ z\in\pi(W).
	\end{equation*}
	Thereby
	\begin{equation*}
	W\subset\varphi^{-1}(V').
	\end{equation*}
	And since $W$ is open in $G$, $y$ is an interior point of $\varphi^{-1}(V')$. It follows that each point in $\varphi^{-1}(V')$ is an interior point. Then $\varphi^{-1}(V')$ is open in $G$. And since $\varphi$ is continuous, $V'$ is open in $\widehat{G}$. So
	\begin{equation*}
	V'=(\varphi|_U^{-1})^{-1}(V)
	\end{equation*}
	is open in $\widehat{G}$ for each domain $V$ in $U$. It is clear that $\varphi|_U^{-1}$ is continuous.
	
	Then $\varphi|_U$ is a homeomorphism. And since $\pi|_U$ is a slice-homeomorphism, it follows that
	\begin{equation*}
	\widehat{\pi}|_{\varphi(U)}=\pi|_U\circ\varphi|_U^{-1}
	\end{equation*}
	is a slice-homeomorphism from $\varphi(U)$ to $\pi(U)$. Then $\widehat{\pi}$ is local slice-homeomorphism.
	
	3). For each $p',q'\in\widehat{G}$ with
	\begin{equation*}
	p'\neq q'.
	\end{equation*}
	If $\widehat{\pi}(p')\neq\widehat{\pi}(q')$, then there exist a slice-domain $U$ in $\mathbb{H}$ containing $\widehat{\pi}(p')$ and a slice-domain $V$ in $\mathbb{H}$ containing $\widehat{\pi}(q')$, such that $U\cap V=\varnothing$. Since $\widehat{\pi}$ is local slice-homeomorphism, $\widehat{\pi}^{-1}(V)$ and $\widehat{\pi}^{-1}(U)$ are two disjoint open sets in $\widehat{G}$ containing $p'$ and $q'$ respectively. Then $\widehat{G}$ is Hausdorff.
	
	If $\widehat{\pi}(p')=\widehat{\pi}(q')$, then there exist a domain $U$ in $\widehat{G}$ containing $p'$, a domain $V$ in $\widehat{G}$ containing $q'$, and a slice-domain $W$ in $\mathbb{H}$, such that
	\begin{equation*}
	\widehat{\pi}|_U:U\rightarrow W\qquad\mbox{and}\qquad\widehat{\pi}|_V:V\rightarrow W
	\end{equation*}
	are slice-homeomorphisms. We notice that $\widehat{\pi}|_U^{-1}$ and $\widehat{\pi}|_V^{-1}$ are two continuous mappings from $W$ to $\widehat{G}$. If there exists $y\in U\cap V$, then
	\begin{equation*}
	\widehat{\pi}|_U^{-1}(\widehat{\pi}(y))=y=\widehat{\pi}|_V^{-1}(\widehat{\pi}(y)),
	\end{equation*}
	and
	\begin{equation*}
	\widehat{\pi}\circ\widehat{\pi}|_U^{-1}={\id}_W=\widehat{\pi}\circ\widehat{\pi}|_V^{-1}.
	\end{equation*}
	And thanks to \cite[Proposition \ref{pr-ul}]{Dou2018001}, we have
	\begin{equation*}
	\widehat{\pi}|_U^{-1}=\widehat{\pi}|_V^{-1}.
	\end{equation*}
	It is clear that
	\begin{equation*}
	p'=\widehat{\pi}|_U^{-1}(\widehat{\pi}(p'))=\widehat{\pi}|_V^{-1}(\widehat{\pi}(q'))=q',
	\end{equation*}
	which is a contradiction. It implies that $U\cap V=\varnothing$. Then $\widehat{G}$ is Hausdorff. And since 1) and 2), $\widehat{\mathcal{G}}$ is a slice-domain over $\mathbb{H}$ with distinguished point.
	
	4). According to (\ref{eq-2pqp}) and (\ref{eq-2xfx}), $\varphi$ is the fiber preserving map from $\mathcal{G}$ to $\widehat{\mathcal{G}}$. It follows that
	\begin{equation}\label{eq-2gg}
	\mathcal{G}\prec\widehat{\mathcal{G}}.
	\end{equation}
	Thanks to (\ref{eq-2fqp}), we have
	\begin{equation*}
	\widehat{f}|_G=\widehat{f}\circ\varphi=f.
	\end{equation*}
	It is clear that $\widehat{f}$ is a slice regular extension of $f$. And since $\mathcal{G}$ is a slice-domain of existence of $f$, it follows that $\widehat{\mathcal{G}}\prec\mathcal{G}$ by \cite[Theorem \ref{th-exsd}]{Dou2018001}. And according to (\ref{eq-2gg}), we have
	\begin{equation*}
	\widehat{\mathcal{G}}\cong\mathcal{G}.
	\end{equation*}
	Thence $\widehat{G}$ is also a slice-domain of existence of $f$, $\varphi$ is invertible, and
	\begin{equation*}
	f|_{\widehat{G}}=\widehat{f}.
	\end{equation*}
	
	We notice that $p\sim q$, i.e., $\varphi(p)=\varphi(q)$. It follows that
	\begin{equation*}
	p=\varphi^{-1}\circ\varphi(p)=\varphi^{-1}\circ\varphi(q)=q.
	\end{equation*}
\end{proof}

Let $f_0:\mathbb{R}^+\rightarrow\mathbb{R}$ be the square root function, defined by
\begin{equation*}
f_0(x):=\sqrt{x},
\end{equation*}
We set
\begin{equation*}
A:=\{yi\in\mathbb{C}:y\in\mathbb{R}\ \mbox{and}\ y\le 0\}\qquad\mbox{and}\qquad\Omega:=\mathbb{C}\backslash A.
\end{equation*}
Then for each $I\in\mathbb{S}$, there exists a holomorphic extension
\begin{equation*}
f^{(I)}_0:\Omega^I\rightarrow\mathbb{H}
\end{equation*}
of $f_0$, where
\begin{equation*}
\Omega^I:=P_I(\Omega).
\end{equation*}
We notice that there also exists a slice regular extension
\begin{equation*}
f_1:U_0\rightarrow\mathbb{H}
\end{equation*}
of $f_0$, where
\begin{equation*}
U_0:=\mathbb{H}\backslash(\mathbb{R}^-\cup\{0\}).
\end{equation*}

For each $I\in\mathbb{S}$ and $a\in\mathbb{H}$, we set
\begin{equation*}
\Omega^I_a:=\{(z,I,a)\in\Omega\times\mathbb{S}\times\mathbb{H}:z\in\Omega\},
\end{equation*}
and $\Omega^I_a$ has an induced topology from $\Omega$. We set
\begin{equation*}
E:=\{q\in\mathbb{H}:|q|=1\}\qquad\mbox{and}\qquad X:=\bigcup_{(I,a)\in\mathbb{S}\times E}\Omega^I_a.
\end{equation*}
Let the topology of $X$ be the disjoint union topology. Then we define a function $F:X\rightarrow\mathbb{H}$, by
\begin{equation*}
F((z,I,a)):=f^{(I)}_0\circ P_I(z)\cdot a,\qquad\forall\ I\in\mathbb{S}\ \mbox{and}\ a\in E.
\end{equation*}
We define a map
\begin{equation*}
\pi:X\rightarrow\mathbb{H},\qquad (z,I,a)\mapsto P_I(z),\qquad\forall\ (z,I,a)\in X.
\end{equation*}
We define an equivalence $\sim$ on $X$ by $x\sim y$, if and only if, $F(x)=F(y)$ and $\pi(x)=\pi(y)$, where $x,y\in X$. We notice that, if $x,y\in X$ with $x\sim y$, there exist domains $W_x$ containing $x$ and $W_y$ containing $y$ in $X$, such that
\begin{equation*}
\pi(W_x)=\pi(W_y)\qquad\mbox{and}\qquad F\circ\pi|_{W_x}^{-1}=F\circ\pi|_{W_y}^{-1}.
\end{equation*}
We set
\begin{equation*}
\widehat{X}:=X/\sim.
\end{equation*}
Let $\varphi:X\rightarrow\widehat{X}$ be the quotient map induced by $\sim$, and the topology of $\widehat{X}$ be the quotient space topology. Then there exist functions
\begin{equation*}
\widehat{F}:\widehat{X}\rightarrow\mathbb{H},\qquad\widehat{\pi}:\widehat{X}\rightarrow\mathbb{H}\qquad\mbox{and}\qquad\widehat{x_0}\in\widehat{X},
\end{equation*}
such that
\begin{equation*}
\widehat{F}\circ\varphi=F,\qquad\widehat{\pi}\circ\varphi=\pi\qquad\mbox{and}\qquad\widehat{x_0}=\varphi((1,I,1)),\quad\forall\ I\in\mathbb{S}.
\end{equation*}

\begin{exa}\label{ex-sr}
	$\mathcal{G}_0:=(\widehat{X},\widehat{\pi},\widehat{x_0})$ is a slice-domain of existence of $f_1$ with respect to slice-domain $(U_0,id_{U_0},1)$ over $\mathbb{H}$ with distinguished point.	
\end{exa}

\begin{proof}
	We set
	\begin{equation*}
	\widehat{X}_{\mathbb{R}^+}:=\{q\in\widehat{X}:\widehat{\pi}(q)\in\mathbb{R}^+\}\qquad\mbox{and}\qquad\widehat{X}_{\mathbb{R}^-}:=\{q\in\widehat{X}:-\widehat{\pi}(q)\in\mathbb{R}^+\}.
	\end{equation*}
	Let $q\in\widehat{X}_{\mathbb{R}^+}$, then there exists $a\in E$ with $\widehat{F}(q)=f_1(\pi(q))a$. And for each $I\in\mathbb{S}$ there exist
	\begin{equation*}
	q^{(I)}\in\Omega^I_a
	\end{equation*}
	such that
	\begin{equation*}
	\{q^{(I)}\}=\varphi^{-1}(q)\cap\Omega^I_a\qquad\mbox{and}\qquad\varphi^{-1}(q)=\bigcup_{I\in\mathbb{S}}\{q^{(I)}\}.
	\end{equation*}
	Then
	\begin{equation*}
	V_q:=\bigcup_{I\in\mathbb{S}}\varphi(B^I_a(q^{(I)},|q|))
	\end{equation*}
	is a domain in $\widehat{X}$, and $\widehat{\pi}|_{V_p}:V_p\rightarrow B_\mathbb{H}(\widehat{\pi}(q),|q|)$ is a slice-homeomorphism, where $B^I_a(q^{(I)},|q|)$ is the open ball in $\Omega^I_a$ of radius $|q|$ and center $\widehat{\pi}(q)$.
	
	It is similar when $q\in\widehat{X}_{\mathbb{R}^-}$. There exists $a\in E$ with
	\begin{equation*}
	\widehat{F}(q)=f_1(-\pi(q))a.
	\end{equation*}
	And for each $I\in\mathbb{S}$, there exist
	\begin{equation*}
	q^{(I)}\in\Omega^I_{-Ia},
	\end{equation*}
	such that
	\begin{equation*}
	\{q^{(I)}\}=\varphi^{-1}(q)\cap\Omega^I_{-Ia}\qquad\mbox{and}\qquad\varphi^{-1}(q)=\bigcup_{I\in\mathbb{S}}\{q^{(I)}\}.
	\end{equation*}
	Then
	\begin{equation*}
	V_q:=\bigcup_{I\in\mathbb{S}}\varphi(B^I_{-Ia}(q^{(I)},|q|))
	\end{equation*}
	is a domain in $\widehat{X}$, and $\widehat{\pi}|_{V_p}:V_p\rightarrow B(\widehat{\pi}(q),|q|)$ is a slice-homeomorphism.
	
	If $q\notin\widehat{X}_{\mathbb{R}}$, we set
	\begin{equation*}
	V_q:=\phi(B_{-Ia}^I(q^{(I)},|Im(q)|)).
	\end{equation*}
	We notice that $\widehat{\pi}|_{V_p}:V_p\rightarrow B(\widehat{\pi}(q),|Im(q)|)$ is a slice-homeomorphism. Then $\mathcal{G}_0$ is local slice-homeomorphism.
	
	Following the proof of Proposition \ref{pr-srfid}, we can get $\widehat{X}$ is a Hausdorff space. We notice that for each $p\in E$, there exist $I,J\in\mathbb{S}$ and $x,y\in\mathbb{R}$ with $I\perp J$ and $x^2+y^2=1$, such that $p=x+yI$. Then
	\begin{equation}\label{eq-2pjx}
	p=J(-xJ-yJI)\qquad\mbox{and}\qquad J,\ -xJ-yJI\in\mathbb{S}.
	\end{equation}
	For each $q\in\widehat{X}_{\mathbb{R}^+}$, there exists $a\in E$ with
	\begin{equation*}
	\widehat{F}(q)=f_1(\pi(q))a.
	\end{equation*}
	According to (\ref{eq-2pjx}), there exists $I,J\in\mathbb{S}$ such that
	\begin{equation*}
	a=IJ.
	\end{equation*}
	It is trivial to prove that, there exist a path $\alpha$ in $\widehat{X}_I$ and a path $\beta$ in $\widehat{X}_J$, such that $\alpha\circ\beta$ is a path in $\widehat{X}$, from $\widehat{x}$ to $q$. And we notice that for each $q''\in\widehat{X}$, there exists a point $q'\in\widehat{X}_\mathbb{R}$ and a path in $\widehat{X}$, from $q''$ to $q'$. It follows that $\widehat{X}$ is connected. In fact, there exists a $3$-part path in $\widehat{X}$ from $\widehat{x}$ to $q$, for each $q\in\widehat{X}$. Then $\mathcal{G}_0$ is a slice-domain over $\mathbb{H}$ with distinguished point. And $\widehat{F}$ is a slice regular extension of $f_1$.
	
	Let $\mathcal{G}'=(G',\pi',x')$ be a slice-domain of existence of $f_1$, and we denote the slice regular extension of $f$ by $f'$. For each $q\in G'$, there exist $N\in\mathbb{N}^+$ and an $N$-part path $\gamma^q$ in $G'$ from $x'$ to $q$. We notice that
	\begin{equation*}
	f'_{x'}=\widehat{F}_{\widehat{x}},
	\end{equation*}
	it follows that
	\begin{equation*}
	f'_{\gamma(\frac{1}{N})}=\widehat{F}_{\alpha^q(\frac{1}{N})},
	\end{equation*}
	where
	\begin{equation*}
	\alpha^q:=(\pi\circ\gamma^q)_{\mathcal{G}_0}.
	\end{equation*}
	Then we have
	\begin{equation*}
	f'_q=f'_{\gamma^q(1)}=\widehat{F}_{\alpha^q(1)}
	\end{equation*}
	by induction. We notice that for each $q_1,q_2\in\widehat{X}$ with $f_{q_1}=f_{q_2}$, we have
	\begin{equation*}
	q_1=q_2.
	\end{equation*}
	Then $\alpha^q(1)$ is not depend on the choice of $\gamma^q$. So there exists a map $\varphi:G'\rightarrow\widehat{X}$ with $\varphi(q)=\alpha^q(1)$. We notice that $\varphi$ is the fiber preserving from $\mathcal{G}'$ to $\mathcal{G}_0$. It follows that $\mathcal{G}'\prec\mathcal{G}_0$. Then $\mathcal{G}_0$ is a slice-domain of existence of $f_1$, by \cite[Theorem \ref{th-exsd}]{Dou2018001}.
\end{proof}

We define a path $\alpha$ in $\mathbb{C}$, by
\begin{equation*}
\alpha(t):=e^{\pi it},\qquad\forall\ t\in[0,1].
\end{equation*}
Then
\begin{equation*}
\beta:=(\alpha,\alpha^{(-1)})
\end{equation*}
is a $2$-part path in $\mathbb{C}$. We notice that
\begin{equation*}
\widehat{F}(\alpha^I_{\mathcal{G}_0}(1))=I,\qquad\forall\ I\in\mathbb{S},
\end{equation*}
and since
\begin{equation*}
\beta^K\circ\alpha^{K_2}=\alpha^{K_1},\qquad\forall\ K=(K_1,K_2)\in\mathbb{S}^2
\end{equation*}
then
\begin{equation*}
K_1=\widehat{F}(\alpha^{K_1}_{\mathcal{G}_0}(1))=K_2\widehat{F}(\beta^K_{\mathcal{G}_0}(1)),\qquad\forall\ K=(K_1,K_2)\in\mathbb{S}^2.
\end{equation*}
It follows that
\begin{equation}\label{eq-es}
\widehat{F}(\beta^K_{\mathcal{G}_0}(1))=K_2^{-1}K_1,\qquad\forall\ K=(K_1,K_2)\in\mathbb{S}^2.
\end{equation}
Let $I\in\mathbb{S}$, and set
\begin{equation*}
J:=\eta_2(I).
\end{equation*}
We also can get (\ref{eq-es}) by the Representation Formula \ref{tm-rf}:
\begin{equation*}
\begin{split}
\widehat{F}(\beta^K_{\mathcal{G}_0}(1))=&\zeta(K)\mathcal{M}(J)f(\beta^J_{\mathcal{G}_0}(1))
\\=&(1,K_1,K_2K_1,-K_2)
\left(\begin{matrix}
1&I&-1&-I\\1&I&1&I\\1&-I&1&-I\\1&-I&-1&I
\end{matrix}\right)^{-1}
\left(\begin{matrix}
I^{-1}I\\I^{-1}(-I)\\(-I)^{-1}I\\(-I)^{-1}(-I)
\end{matrix}\right)
\\=&(1,K_1,K_2K_1,-K_2)\frac{1}{4}
\left(\begin{matrix}
1&1&1&1\\-I&-I&I&I\\-1&1&1&-1\\I&-I&I&-I
\end{matrix}\right)
\left(\begin{matrix}
1\\-1\\-1\\1
\end{matrix}\right)
\\=&(1,K_1,K_2K_1,-K_2)\left(\begin{matrix}0\\0\\-1\\0\end{matrix}\right)
\\=&K_2^{-1}K_1
\end{split}
\end{equation*}
for each $K\in\mathbb{S}^2$.

Now we describe another example without proof. Let $f_2:\mathbb{R}^+\rightarrow\mathbb{R}^+$ be the natural logarithm function, i.e.,
\begin{equation*}
f_2(x)=ln(x),\qquad\forall\ x\in\mathbb{R}.
\end{equation*}
There exist a slice regular function $f_3:U_0\rightarrow\mathbb{H}$ with $f_3|_{\mathbb{R}^+}=f_2$, and a holomorphic extension
\begin{equation*}
f^{(I)}_2:\Omega^I\rightarrow\mathbb{H}
\end{equation*}
of $f_2$ for each $I\in\mathbb{S}$. For each $I\in\mathbb{S}$ and $a\in\mathbb{H}$, we define a function $F':X'\rightarrow\mathbb{H}$, by
\begin{equation}\label{eq-2fzi}
F'(z,I,a)\mapsto f_2^{(I)}\circ P_I(z)+a.
\end{equation}
We set
\begin{equation*}
E':=\{q\in\mathbb{H}:Im(q)=0\}\qquad\mbox{and}\qquad X':=\bigcup_{(I,a)\in\mathbb{S}\times E'}\Omega^I_a.
\end{equation*}
We define an equivalence $\cong$ on $X'$ by $x\cong y$, if and only if, $\pi(x)=\pi(y)$ and $F'(x)=F'(y)$. We set
\begin{equation}\label{eq-2xx}
\widetilde{X}:=X'/\cong,
\end{equation}
and let
\begin{equation*}
\phi:X'\rightarrow X'/\cong
\end{equation*}
be the quotient map induced by $\cong$. Then there exist
\begin{equation}\label{eq-2fxh}
\widetilde{F}:\widetilde{X}\rightarrow\mathbb{H}\qquad\mbox{and}\qquad\widetilde{\pi}:\widetilde{X}\rightarrow\mathbb{H},
\end{equation}
such that
\begin{equation*}
\widetilde{F}\circ\phi=F'\qquad\mbox{and}\qquad\widetilde{\pi}\circ\phi=\pi.
\end{equation*}
Then
\begin{equation*}
\mathcal{G}_2:=(\widetilde{X},\widetilde{\pi},\widetilde{x})
\end{equation*}
is a slice-domain of existence of $f_3$, where
\begin{equation}\label{eq-2xpi}
\widetilde{x}:=\phi(1,I,0).
\end{equation}
And $\widetilde{F}$ is a slice regular extension of $f_3$. We notice that
\begin{equation*}
\widetilde{F}(\beta^K_{\mathcal{G}_2(1)})=\pi K_1-\pi K_2,\qquad\forall\ K\in\mathbb{S}^2.
\end{equation*}
Let $I\in\mathbb{S}$, and set
\begin{equation*}
J:=\eta_2(I).
\end{equation*}
We also can get by the Representation Formula:
\begin{equation*}
\begin{split}
\widehat{F}(\beta^K_{\mathcal{G}_2}(1))=&\zeta(K)\mathcal{M}(J)f(\beta^J_{\mathcal{G}_2}(1))
\\=&(1,K_1,K_2K_1,-K_2)
\left(\begin{matrix}
1&I&-1&-I\\1&I&1&I\\1&-I&1&-I\\1&-I&-1&I
\end{matrix}\right)^{-1}
\left(\begin{matrix}
\pi I-\pi I\\\pi I-(-\pi I)\\(-\pi I)-\pi I\\(-\pi I)-(-\pi I)
\end{matrix}\right)
\\=&(1,K_1,K_2K_1,-K_2)\frac{1}{4}
\left(\begin{matrix}
1&1&1&1\\-I&-I&I&I\\-1&1&1&-1\\I&-I&I&-I
\end{matrix}\right)
\left(\begin{matrix}
0\\2\pi I\\-2\pi I\\0
\end{matrix}\right)
\\=&(1,K_1,K_2K_1,-K_2)\left(\begin{matrix}0\\\pi\\0\\\pi\end{matrix}\right)
\\=&\pi K_1-\pi K_2
\end{split}
\end{equation*}
for each $K\in\mathbb{S}^2$.

\begin{defn}
	Let $\mathcal{G}=(G,\pi,x)$ be a slice-domain over $\mathbb{H}$ with distinguished point, and $N\in\mathbb{N}^+$. $\mathcal{G}$ is called $N$-axially symmetric, if the following statements hold
	\begin{enumerate}
		\item $\pi(x)\in\mathbb{R}$
		
		\item For each $\gamma\in\mathcal{P}^N(\mathbb{C})$ and $I\in\mathbb{S}^N$ with $\gamma^I\prec\mathcal{G}$, we have
		\begin{equation*}
		\gamma^K\prec\mathcal{G},\qquad\forall\ K\in\mathbb{S}^N.
		\end{equation*}
	\end{enumerate}
	
	$\mathcal{G}$ is called $\infty$-axially symmetric, if $\mathcal{G}$ is $m$-axially symmetric for each $m\in\mathbb{N}^+$.
\end{defn}

We notice that $\mathcal{G}_0$ and $\mathcal{G}_2$ are $\infty$-axially symmetric.

Let $\mathcal{G}=(G,\pi,x)$ be a slice-domain over $\mathbb{H}$ with distinguished point. For each $I\in\mathbb{S}$ and $y\in G_I$, we denote the connected component in $G_I$ containing $y$, by $G_I^y$. For each $y\in G_\mathbb{R}$, we denote the the connected component in $G_\mathbb{R}$ containing $y$, by $G_\mathbb{R}^y$.

\begin{prop}\label{pr-srip}
	Let $\mathcal{G}=(G,\pi,x)$ be a $\infty$-axially symmetric slice-domain over $\mathbb{H}$ with distinguished point, $I\in\mathbb{S}$ and $f:G_I^x\rightarrow\mathbb{H}$ be a holomorphic function. If there exists a slice regular function $F$ on $G$ with $F|_{G_I^x}=f$. Then
	\begin{equation*}
	F(\gamma^K_{\mathcal{G}}(1))=\zeta(K)\mathcal{M}(\eta_N(I))f(\gamma^{\eta_N(I)}_{\mathcal{G}}(1))
	\end{equation*}
	for each $N\in\mathbb{N}^+$, $\gamma\in\mathcal{P}^N(\mathbb{C})$ and $K\in\mathbb{S}^N$ with $\gamma^K\prec\mathcal{G}$.
\end{prop}

\begin{proof}
	We notice that $\mathcal{G}$ is $\infty$-axially symmetric and $K\in\mathbb{S}^N$, it follows that
	\begin{equation*}
	\gamma^{\eta^\imath_N(I)}\prec\mathcal{G},\qquad\forall\ \imath\in\{1,2,...,2^N\}.
	\end{equation*}
	Then this proposition immediately from the Representation Formula \ref{tm-rf}.
\end{proof}

Let $\mathcal{G}=(G,\pi,x)$ be a $\infty$-axially symmetric slice-domain over $\mathbb{H}$ with distinguished point, $I\in\mathbb{S}$ and $q\in\mathcal{G}$. Thanks to \cite[Theorem \ref{th-efpsp}]{Dou2018001}, there exists $N\in\mathbb{N}^+$, $K\in\mathbb{S}^N$ and $\gamma\in\mathcal{P}^N(\mathbb{C})$, such that $\gamma^K_{\mathcal{G}}$ is an $N$-part path in $G$ from $x$ to $q$. If $F$ is a slice regular function on $G$. According Proposition \ref{pr-srip}, $F$ is determined by $F|_{G^x_I}$. Conversely, if $f:G_I^x\rightarrow\mathbb{H}$ is a holomorphic function, does there exist a slice regular $F$ on $G$ such that
\begin{equation*}
F|_{G^x_I}=f?
\end{equation*}

We attempt to define a slice regular $F$ on $G$, by Representation Formula \ref{tm-rf}, i.e.,
\begin{equation}\label{eq-rficrs}
F(\gamma^K_{\mathcal{G}}(1))=\zeta(K)\mathcal{M}(\eta_N(I)))f(\gamma^{\eta_N(I)}_{\mathcal{G}}(1)),
\end{equation}
for each $N\in\mathbb{N}^+$, $\gamma\in\mathcal{P}^N(\mathbb{C})$, and $K\in\mathbb{S}^N$ with $\gamma^K\prec\mathcal{G}$. However $F$ constructed by (\ref{eq-rficrs}) may be a multivalued function. For example, there exists a holomorphic function $f_4:\widetilde{X}_I^{\widetilde{x}}\rightarrow\mathbb{H}$ such that
\begin{equation*}
f_4\circ\widetilde{\pi}|_{\widetilde{X}_{\mathbb{R}}^{\widetilde{x}}}^{-1}=f_0,
\end{equation*}
where $\widetilde{X}$ is defined in (\ref{eq-2xx}), $\widetilde{\pi}$ is defined by (\ref{eq-2fxh}) and $\widetilde{x}$ is defined in (\ref{eq-2xpi}). We define a path $\alpha^{(m)}$ in $\mathbb{C}$, by
\begin{equation*}
\alpha^{(m)}(t):=e^{m\pi it},\qquad\forall\ m\in\mathbb{N}^+\ \mbox{and}\ t\in[0,1].
\end{equation*}
We define a $1$-part path
\begin{equation*}
\gamma_1:=\alpha^{(5)}
\end{equation*}
in $\mathbb{C}$, and a $2$-part path
\begin{equation*}
\gamma_2:=(\alpha^{(4)},\alpha^{(3)})
\end{equation*}
in $\mathbb{C}$. Let $J_1,J_2\in\mathbb{S}$ with $J_1\perp J_2$. We set
\begin{equation*}
K:=\frac{4J_1+3J_2}{5}\in\mathbb{S}\qquad\mbox{and}\qquad J:=(J_1,J_2)\in\mathbb{S}^2.
\end{equation*}
Since paths $\alpha^{(m)}$, $m\in\mathbb{N}^+$, do not pass through the origin of $\mathbb{C}$, it follows that
\begin{equation*}
(\gamma_1)^K\prec\mathcal{G}_2\qquad\mbox{and}\qquad (\gamma_2)^J\prec\mathcal{G}_2.
\end{equation*}
Then we set
\begin{equation*}
q_1:=(\gamma_1)^K_{\mathcal{G}_2}(1)\qquad\mbox{and}\qquad q_2:=(\gamma_2)^J_{\mathcal{G}_2}(1).
\end{equation*}
Suppose that there exists a slice regular function $F_4$ on $\widetilde{X}$ with
\begin{equation*}
F_4|_{\widetilde{X}_I^{\widetilde{x}}}=f_4,
\end{equation*}
Since
\begin{equation*}
\widetilde{F}_{q_1}=(F'+4\pi K)_{(-1,K,0)}=(F'-\pi K+4\pi J_1+3\pi J_2)_{(-1,K,0)}=\widetilde{F}_{q_2},
\end{equation*}
\begin{equation*}
\widetilde{\pi}(q_1)=P_K(\gamma_1(1))=P_K(-1)=-1=P_{J_2}(-1)=P_{J_2}(\gamma_2(1))=\widetilde{\pi}(q_2),
\end{equation*}
and Proposition \ref{pr-srfid}, it follows that
\begin{equation*}
q_1=q_2,
\end{equation*}
where $F'$ is defined by (\ref{eq-2fzi}) and $\widetilde{F}$ is defined by (\ref{eq-2fxh}). We notice that
\begin{equation*}
F_4(q_2)=F_4(q_1)=f_4(q_1)=\sqrt{1}K^5=K,
\end{equation*}
and
\begin{equation*}
\begin{split}
F_4(q_2)=&\zeta(J)\mathcal{M}(\eta_2(I))F_4((\gamma_1)_{\mathcal{G}_2}^{\eta_2(I)}(1))
\\=&(1,J_1,J_2 J_1,-J_2)
\left(\begin{matrix}
1&I&-1&-I\\1&I&1&I\\1&-I&1&-I\\1&-I&-1&I
\end{matrix}\right)^{-1}
\left(\begin{matrix}
f_4((\gamma_1)^{(I,I)}_{\mathcal{G}_2}(1))\\f_4((\gamma_1)^{(I,-I)}_{\mathcal{G}_2}(1))\\f_4((\gamma_1)^{(-I,I)}_{\mathcal{G}_2}(1))\\f_4((\gamma_1)^{(-I,-I)}_{\mathcal{G}_2}(1))
\end{matrix}\right)
\\=&(1,J_1,J_2 J_1,-J_2)\frac{1}{4}
\left(\begin{matrix}
1&1&1&1\\-I&-I&I&I\\-1&1&1&-1\\I&-I&I&-I
\end{matrix}\right)
\left(\begin{matrix}
\sqrt{1}I^4I^3\\\sqrt{1}I^4(-I)^3\\\sqrt{1}(-I)^4 I^3\\\sqrt{1}(-I)^4(-I)^3
\end{matrix}\right)
\\=&(1,J_1,J_2 J_1,-J_2)\frac{1}{4}
\left(\begin{matrix}
1&1&1&1\\-I&-I&I&I\\-1&1&1&-1\\I&-I&I&-I
\end{matrix}\right)
\left(\begin{matrix}
-I\\I\\-I\\I
\end{matrix}\right)
\\=&(1,J_1,J_2 J_1,-J_2)
\left(\begin{matrix}
0\\0\\0\\1
\end{matrix}\right)
\\=&-J_2,
\end{split}
\end{equation*}
which is a contradiction. So $f_4$ can not ``extend slice regularly" to $\widetilde{X}$.

\section{Holomorphic stem system}\label{sc-hss}

The notion of the stem function was introduced by Fueter in \cite{Fueter1934/35}. Ghiloni and Perotti extend the stem function to develop the theory of slice regular functions on the real alternative *-algebra in \cite{Ghiloni2011001}. In this section, we will define a similar notion, so-called holomorphic stem systems, to introduce the $*$-product of slice regular functions on Riemann slice-domains (see Section \ref{sc-psrf}).

\begin{defn}\label{df-css}
	Let $\mathcal{G}$ be a slice-domain over $\mathbb{H}$ with distinguished point, $N\in\mathbb{N}^+$, and $\gamma$ be an $N$-part path in $\mathbb{C}$. We say that $\gamma$ is contained in $\mathcal{G}$ (denoted by $\gamma\prec\mathcal{G}$), if $\gamma^I\prec\mathcal{G}$ for each $I\in\mathbb{S}^N$.
	
	We denote by $\mathcal{P}^{\infty}_{\mathbb{C}}(\mathcal{G})$ (resp. $\mathcal{P}^N_{\mathbb{C}}(\mathcal{G})$) the set of all the finite-part (resp. $N$-part) paths in $\mathbb{C}$ contained in $\mathcal{G}$.
\end{defn}

Let $\mathcal{G}=(G,\pi,x)$ be a slice-domain over $\mathbb{H}$ with distinguished point, if $\pi(x)\notin\mathbb{R}$, then there exists $I\in\mathbb{S}$ such that $\pi(x)\in\mathbb{C}_I$. Let $K\in\mathbb{S}$ with $K\perp I$, then
\begin{equation*}
\gamma^K(0)\in\mathbb{C}_K\backslash\mathbb{R},\qquad\forall\ \gamma\in\mathcal{P}^{\infty}(\mathbb{C}).
\end{equation*}
It follows that
\begin{equation*}
\gamma^K(0)\neq\pi(x)\quad\mbox{and}\quad\gamma\nprec\mathcal{G},\qquad\forall\ \gamma\in\mathcal{P}^{\infty}(\mathbb{C}).
\end{equation*}
Then
\begin{equation*}
\mathcal{P}^{\infty}_{\mathbb{C}}(\mathcal{G})=\varnothing.
\end{equation*}

\begin{prop}\label{pr-spa}
	Let $\mathcal{G}$ be a $\infty$-axially symmetric slice-domain over $\mathbb{H}$ with distinguished point, $N\in\mathbb{N}^+$, and $\gamma$ be an $N$-part path in $\mathbb{C}$. If there exists $I\in\mathbb{S}^N$ such that $\gamma^I\prec\mathcal{G}$, then $\gamma\prec\mathcal{G}$.
\end{prop}

\begin{proof}
	If $\gamma^I\prec\mathcal{G}$, and since $\mathcal{G}$ is $\infty$-axially symmetric, it follows that $\gamma\prec\mathcal{G}$.
\end{proof}

\begin{defn}
	We call a subset $\Gamma$ of $\mathcal{P}^{\infty}(\mathbb{C})$ is radial, if $\gamma[t]\in\Gamma$ and $\gamma[t^-]\in\Gamma$, for each $\gamma\in\Gamma$ and $t\in[0,1]$.
\end{defn}

\begin{defn}
	Let $\Omega$ be an open set in $\mathbb{C}$ and $N\in\mathbb{N}^+$. A function $F:\Omega\rightarrow M_{2^N\times 1}(\mathbb{H})$ is called stem ($N$-)holomorphic or (stem finite-)holomorphic, if $f$ has continuous partial derivatives and satisfies
	\begin{equation*}
	(\frac{\partial}{\partial x}+\sigma_N\frac{\partial}{\partial y})F(x+yi)=0_{2^N\times 1}
	\end{equation*}
	for each $x,y\in\mathbb{R}$ with $x+yi\in\Omega$.
	
	We denote by $\mathcal{SH}^N(\Omega)$ the set of all the stem $N$-holomorphic function on $\Omega$ for each $N\in\mathbb{N}^+$. And we denote by $\mathcal{SH}^{\infty}(\Omega)$ the set of all the stem finite-holomorphic functions.
\end{defn}

Let $x_0\in\mathbb{R}$. We recall $\mathcal{P}^{\infty}_{x_0}(\mathbb{C})$ is the set of all finite-part path in $\mathbb{C}$ with an initial point $x_0$. If $\Gamma$ is a radial subset of $\mathcal{P}^{\infty}_{x_0}(\mathbb{C})$, then
\begin{equation*}
\gamma[0]=\alpha,\qquad\forall\ \gamma\in\Gamma,
\end{equation*}
where $\alpha$ is a $1$-part path in $\mathbb{C}$ with
\begin{equation}\label{eq-2asx}
\alpha(s):=x_0,\qquad\forall\ s\in[0,1].
\end{equation}
We say that $\alpha$ is the initial path of $\Gamma$.

For each $q\in\mathbb{H}$, $z\in\mathbb{C}$, $r\in\mathbb{R}^+$, and $I\in\mathbb{S}$, we set
\begin{equation*}
\begin{split}
\mathbb{B}_{\mathbb{H}}(q,r):=\{p\in\mathbb{H}:|p-q|<r\},\qquad
\mathbb{B}_{\mathbb{C}}(z,r):=\{p\in\mathbb{C}:|p-z|<r\},
\\\mathbb{B}_I(q,r):=\mathbb{B}_\mathbb{H}(q,r)\cap\mathbb{C}_I\qquad \mbox{and}\qquad \mathbb{B}_{\mathbb{R}}(q,r):=\{p\in\mathbb{R}:|p-q|<r\}.
\end{split}
\end{equation*}

Let $\Gamma$ be a subset of $\mathcal{P}^{\infty}(\mathbb{C})$, and $r:\Gamma\rightarrow\mathbb{R}^+$ be a map, we set
\begin{equation*}
\mathbb{B}^r_\gamma:=B_{\mathbb{C}}(\gamma(1),r(\gamma)),\qquad\forall\ \gamma\in\Gamma,
\end{equation*}
and
\begin{equation*}
\mathcal{SH}^{\infty}(\mathbb{B}_{\Gamma}^r):=\bigcup_{\gamma\in\Gamma}\mathcal{SH}^{\infty}(\mathbb{B}_{\gamma}^r).
\end{equation*}

Let $(G,\pi)$ be a slice-domain over $\mathbb{H}$, and $\Gamma$ be a subset of $\mathcal{P}^{\infty}(\mathbb{C})$ (resp. $\mathcal{P}^{\infty}(\mathbb{H})$ or $\mathcal{P}^{\infty}(G)$). We set
\begin{equation*}
\Gamma^N:=\Gamma\cap\mathcal{P}^N(\mathbb{C})\qquad(\mbox{resp.}\quad\Gamma^N:=\Gamma\cap\mathcal{P}^N(\mathbb{H})\quad\mbox{or}\quad\Gamma^N:=\Gamma\cap\mathcal{P}^N(G)).
\end{equation*}

\begin{defn}\label{df-hssy}
	Let $x_0\in\mathbb{R}$, $\Gamma$ be a radial subset of $\mathcal{P}^{\infty}_{x_0}(\mathbb{C})$, $r:\Gamma\rightarrow\mathbb{R}^+$ be a map, and $F:\Gamma\rightarrow\mathcal{SH}^{\infty}(\mathbb{B}_{\Gamma}^r)$ be a map. $(F,r,\Gamma)$ is called a (holomorphic) stem system, if the following statements hold:
	\begin{enumerate}
		\item (Local holomorphy) For each $N\in\mathbb{N}^+$ and $\gamma\in\Gamma^N$, we have
		\begin{equation*}
		F_\gamma\in\mathcal{SH}^N(\mathbb{B}_{\gamma}^r),
		\end{equation*}
		where
		\begin{equation*}
		F_\gamma:=F(\gamma).
		\end{equation*}
		
		\item (Local compatibility) For each $N\in\mathbb{N}^+$, $N$-part path $\gamma$ in $\Gamma$, $m\in\{1,2,...,N\}$, and $t_1,t_2\in[\frac{m-1}{N},\frac{m}{N}]$ with $t_1<t_2$. If there exists a real number $t\in[t_1,t_2]$ such that
		\begin{equation*}
		\gamma([t_1,t])\subset\mathbb{B}_{\gamma[t_1]}^r\qquad\mbox{and}\qquad\gamma([t,t_2])\subset\mathbb{B}_{\gamma[t_2^-]}^r,
		\end{equation*}
		then $F_{\gamma[t_1]}$ and $F_{\gamma[t_2^-]}$ coincide on $\mathbb{B}_{\gamma[t_1]}^r\cap\mathbb{B}_{\gamma[t_2^-]}^r$.
		
		\item (Axial compatibility) For each $N\in\mathbb{N}^+$, $\gamma\in\Gamma^N$, $m\in\{1,2,...,N-1\}$ and
		\begin{equation*}
		x\in\mathbb{B}_{\gamma[\frac{m}{N}]}^r\cap\mathbb{B}_{\gamma[\frac{m}{N}^-]}^r\cap\mathbb{R},
		\end{equation*}
		we have
		\begin{equation*}
		F_{\gamma[\frac{m}{N}]}(x)=(F_{\gamma[\frac{m}{N}^-]}(x)^T,0_{1\times 2^m})^T.
		\end{equation*}
		
		\item (Initial compatibility) There exists a function $f:\mathbb{B}_{\gamma[0]}^r\cap\mathbb{R}\rightarrow\mathbb{H}$ such that
		\begin{equation*}
		F_{\gamma[0]}(x)=(f(x),0)^T
		\end{equation*}
		for each $x\in\mathbb{B}_{\gamma[0]}^r\cap\mathbb{R}$ and $\gamma\in\Gamma$.
	\end{enumerate}
\end{defn}

Let  $\Gamma$ be a subset of $\mathcal{P}^{\infty}(\mathbb{C})$, and $r:\Gamma\rightarrow\mathbb{R}^+$ be a map. For each $N\in\mathbb{N}^+$, $\gamma\in\Gamma^N$ and $K\in\mathbb{S}^N$, we set
\begin{equation*}
(\mathbb{B}_{\gamma}^r)^K:=P_{K_N}(\mathbb{B}_{\gamma}^r).
\end{equation*}

\begin{prop}\label{pr-sfsl}
	Let $(F,r,\Gamma)$ be a holomorphic stem system, $N\in\mathbb{N}^+$, $\gamma\in\Gamma^N$, and $K\in\mathbb{S}^N$. Then $F_{\gamma}^K:(\mathbb{B}_{\gamma}^r)^K\rightarrow\mathbb{H}$ is a holomorphic function defined by
	\begin{equation}\label{eq-2fkr}
	F_{\gamma}^K(x+yK_N):=\zeta(K)F_\gamma(x+yi)
	\end{equation}
	for each $x,y\in\mathbb{R}$ with $x+yi\in\mathbb{B}_\gamma^r$.
\end{prop}

\begin{proof}
	According to \cite[Proposition \ref{pr-ct}]{Dou2018001}, we have
	\begin{equation*}
	\begin{split}
	(\frac{\partial}{\partial x}+K_N\frac{\partial}{\partial y})F_{\gamma}^K(x+yK_N)
	=&(\frac{\partial}{\partial x}+K_N\frac{\partial}{\partial y})\zeta(K)F_\gamma(x+yi)
	\\=&\zeta(K)(\frac{\partial}{\partial x}+\sigma_N\frac{\partial}{\partial y})F_\gamma(x+yi)
	\\=&\zeta(K)0_{2^N\times 1}
	\\=&0
	\end{split}
	\end{equation*}
	for each $N\in\mathbb{N}^+$, $\gamma\in\Gamma^N$, and $x,y\in\mathbb{R}$ with $x+yi\in\mathbb{B}_\gamma^r$.
\end{proof}

\begin{defn}\label{df-hss}
	Let $\mathcal{F}=(F,r,\Gamma)$ be a holomorphic stem system, and $\mathcal{G}=(G,\pi,x)$ be a slice-domain over $\mathbb{H}$ with distinguished point. We say that $\mathcal{F}$ is a holomorphic stem system on $\mathcal{G}$ (denoted by $\mathcal{F}\prec\mathcal{G}$), if there exists a slice regular function $f$ on $G$, such that the following statements hold:
	\begin{enumerate}
		\item $\gamma\prec\mathcal{G}$ for each $\gamma\in\Gamma$.
		
		\item For each $N\in\mathbb{N}^+$, $\gamma\in\Gamma^N$ and $K\in\mathbb{S}^N$, there exists a domain $U^K_\gamma$ in $G_{K_N}$ containing $\gamma^K_{\mathcal{G}}(1)$, such that $\pi|_{U^K_\gamma}:U^K_\gamma\rightarrow(\mathbb{B}_{\gamma}^r)^K$ is a homeomorphism with respect to topologies $\tau(G_{K_N})$ and $\tau(\mathbb{C}_{K_N})$, and
		\begin{equation}\label{eq-2fpi}
		f\circ\pi|_{U^K_\gamma}^{-1}=F_{\gamma}^K,
		\end{equation}
		where $F_{\gamma}^K$ is defined by (\ref{eq-2fkr}).
	\end{enumerate}
	We call $f$ the slice regular function induced by $\mathcal{F}$ on $\mathcal{G}$ and write
	\begin{equation*}
	\mathcal{I}_{\mathcal{G}}[\mathcal{F}]:=f.
	\end{equation*}
	We denote by $\mathcal{HS}({\mathcal{G}})$ the set of all the holomorphic stem systems on $\mathcal{G}$. We define a map
	\begin{equation*}
	\mathcal{I}_{\mathcal{G}}:\mathcal{HS}(\mathcal{G})\rightarrow\mathcal{SR}(G),\qquad\mathcal{L}\mapsto\mathcal{I}_{\mathcal{G}}[\mathcal{L}],\qquad\forall\ \mathcal{L}\in\mathcal{HS}(\mathcal{G}).
	\end{equation*}
\end{defn}

	Let $\mathcal{F}$ be a holomorphic stem system, and $\mathcal{G}$ be a slice-domain over $\mathbb{H}$ with distinguished point with $\mathcal{F}\prec\mathcal{G}$. According to \cite[Identity Principle \ref{th-idh}]{Dou2018001}, the slice regular function induced by $\mathcal{F}$ on $\mathcal{G}$ is unique. So $\mathcal{I}_{\mathcal{G}}$ is well defined.

\begin{thm}\label{th-sfsr}
	For each holomorphic stem system $\mathcal{F}$, there exists a slice-domain $\mathcal{G}$ over $\mathbb{H}$ with distinguished point, such that $\mathcal{F}\prec\mathcal{G}$.
\end{thm}

\begin{proof}
	We write $\mathcal{F}=(F,r,\Gamma)$, and let $x_1\in\mathbb{R}$ be the initial point of the paths in $\Gamma$. Let $\alpha$ be the initial path of $\Gamma$, defined by (\ref{eq-2asx}). We define a function $f_0:\mathbb{B}_0\rightarrow\mathbb{H}$, by
	\begin{equation}\label{eq-2fxy}
		f_0(x+yI):=F_\alpha^I(x+yI)=(1,I)F_\alpha(x+yi)
	\end{equation}
	for each $x,y\in\mathbb{R}$ with $x+yi\in\mathbb{B}_\alpha^r$, and $I\in\mathbb{S}$, where $\mathbb{B}_0:=B_{\mathbb{H}}(x_1,r_{\alpha})$.  According $\mathcal{F}$ is initial compatible, then there exists a function $f_1:\mathbb{B}_\alpha^r\cap\mathbb{R}\rightarrow\mathbb{H}$, such that
	\begin{equation*}
	F_\alpha(x)=(f_1(x),0)^T,\qquad\forall\ x\in\mathbb{B}_\alpha^r\cap\mathbb{R}.
	\end{equation*}
	It follows that
	\begin{equation*}
	f_0(x)=(1,I)F_\alpha(x)=(1,I)(f_1(x),0)^T=f_1(x),\qquad\forall x\in\mathbb{B}_\alpha^r\cap\mathbb{R}
	\end{equation*}
	does not depend on the choice of $I$. Therefore $f_0$ is well defined.
	
	According to Proposition \ref{pr-sfsl}, $f_0$ is a slice regular function on $\mathbb{B}_0$, and
	\begin{equation*}
	\mathcal{G}_0:=(\mathbb{B}_0,id_{\mathbb{B}_0},x_1)
	\end{equation*}
	is a slice-domain over $\mathbb{H}$ with distinguished point. Let $\mathcal{G}=(G,\pi,x_0)$ be a slice-domain of existence of the function $f_0$ with respect to $\mathcal{G}_0$, and $f$ be a slice regular function on $G$ with
	\begin{equation}\label{eq-2fbz}
	f|_{\mathbb{B}_0}=f_0.
	\end{equation}
	Let $\Gamma'$ be a subset of $\Gamma$, such that $\gamma\in\Gamma'$, if and only if the following properties hold:
	\begin{enumerate}[i)]
		\item $\gamma\in\Gamma$ and $\gamma\prec\mathcal{G}$.
		
		\item\label{it-2nki} For each $N\in\mathbb{N}^+$ and $K\in\mathbb{S}^N$. If $\gamma\in\Gamma^N$, then there exists a domain $U^K_\gamma$ in $G_{K_N}$ containing $\gamma^K_{\mathcal{G}}(1)$, such that $\pi|_{U^K_\gamma}:U^K_\gamma\rightarrow(\mathbb{B}_{\gamma}^r)^K$ is a homeomorphism with respect to topologies $\tau(G_{K_N})$ and $\tau(\mathbb{C}_{K_N})$, and
		\begin{equation}\label{eq-2fpu}
		f\circ\pi|_{U^K_\gamma}^{-1}=F_{\gamma}^K.
		\end{equation}
	\end{enumerate}
	For each $N\in\mathbb{N}^+$ and $\gamma\in\Gamma^N$, we will prove that $\gamma\in\Gamma'$, following the proof of \cite[Theorem \ref{th-dere}]{Dou2018001} with ignoring some of the details which is similar to \cite[Theorem \ref{th-dere}]{Dou2018001}.
	
	We notice that for each $N\in\mathbb{N}^+$, $K\in\mathbb{S}^N$ and $\gamma\in(\Gamma')^N$, there exists a unique domain satisfing \ref{it-2nki}), denoted by $U^K_\gamma$.
	
	Suppose $\Gamma'\neq\Gamma$, then we set
	\begin{equation*}
	t_1:=\inf\{t\in[0,1]:\gamma[t]\notin\Gamma'\ or\ \gamma[t^-]\notin\Gamma'\}.
	\end{equation*}
	
	\textbf{1). We will prove that $t_1\neq 0$, in this step.}
	
	Thanks to (\ref{eq-2fxy}) and (\ref{eq-2fbz}), we have
	\begin{equation*}
	\gamma[0]=\gamma[0^-]=\alpha\in\Gamma.
	\end{equation*}
	
	Since $\gamma$ is continuous, there exists $t_2\in(0,\frac{1}{N})$ such that
	\begin{equation*}
	\gamma([0,t_2])\subset\mathbb{B}_\alpha^r.
	\end{equation*}
	According to $\mathcal{F}$ is locally compatible, we have
	\begin{equation*}
	F_{\gamma[t_3]}=F_{\alpha},\qquad\mbox{on}\qquad\mathbb{B}_{\gamma[t_3]}^r\cap\mathbb{B}_\alpha^r,\qquad\forall\ t_3\in[0,t_2].
	\end{equation*}
	And thanks to (\ref{eq-2fkr}) and (\ref{eq-2fpu}), it follows that
	\begin{equation*}
	\begin{split}
	F_{\gamma[t_3]}^K(x+yK)=&\zeta(K)F_{\gamma[t_3]}(x+yi)
	\\=&\zeta(K)F_{\alpha}(x+yi)
	\\=&F_{\alpha}^K(x+yK)
	\\=&f\circ\pi|_{U_\alpha^K}^{-1}(x+yK)
	\end{split}
	\end{equation*}
	for each $K\in\mathbb{S}$, and $x,y\in\mathbb{R}$ with $x+yi\in\mathbb{B}_{\gamma[t_3]}^r\cap\mathbb{B}_\alpha^r$.
	
	For each $K\in\mathbb{S}$, according to \cite[Proposition \ref{pr-xfh}]{Dou2018001}, there exists a slice-domain $U_3^K$ in $\mathbb{H}$ containing $(\mathbb{B}_{\gamma[t_3]}^r)^K$, and a slice regular function $f_3^K$ on $U_3^K$ with
	\begin{equation*}
	f_3^K|_{(\mathbb{B}_{\gamma[t_3]}^r)^K}=F_{\gamma[t_3]}^K.
	\end{equation*}
	And according to \cite[Theorem \ref{th-exsd} and Proposition \ref{pr-tsh}]{Dou2018001},
	\begin{equation*}
	(U_3^K,id_{U_3^K},q_3)\prec(G,\pi,\pi|_{U_\alpha^K}^{-1}(q_3)),\qquad\forall\ q_3\in(\mathbb{B}_{\gamma[t_3]}^r)^K\cap(\mathbb{B}_\alpha^r)^K.
	\end{equation*}
	It follows that
	\begin{equation*}
	\gamma[t_3]=\gamma[t_3^-]\in\Gamma',\qquad\forall\ t_3\in(0,t_2).
	\end{equation*}
	Then
	\begin{equation*}
	0\neq\inf\{t\in[0,1]:\gamma[t]\notin\Gamma'\ or\ \gamma[t^-]\notin\Gamma'\}=t_1.
	\end{equation*}
	
	\textbf{2). We will prove that $\{Nt_1\}=0$, in this step.}
	
	Suppose that $\{Nt_1\}\neq 0$. There exists
	\begin{equation*}
	t_4\in(\frac{\lfloor Nt_1\rfloor}{N},t_1)\qquad\mbox{and}\qquad t_5\in(t_1,\frac{\lceil Nt_1\rceil}{N})
	\end{equation*}
	such that
	\begin{equation*}
	\gamma([t_4,t_5])\subset\mathbb{B}_{\gamma[t_1]}^r.
	\end{equation*}
	And thanks to (\ref{eq-2fkr}) and (\ref{eq-2fpu}), we have
	\begin{equation}\label{eq-2fkrt}
	\begin{split}
	F_{\gamma[t_1]}^K(x+yK_{N_1})=&\zeta(K)F_{\gamma[t_1]}(x+yi)
	\\=&\zeta(K)F_{\gamma[t_4]}(x+yi)
	\\=&F_{\gamma[t_4]}^K(x+yK_{N_1})
	\\=&f\circ\pi|_{U^K_{\gamma[t_4]}}^{-1}(x+yK_{N_1})
	\end{split}
	\end{equation}
	for each $x,y\in\mathbb{R}$ with $x+yi\in\mathbb{B}_{\gamma[t_1]}^r\cap\mathbb{B}_{\gamma[t_4]}^r$ and $K\in\mathbb{S}^{N_1}$, where
	\begin{equation}\label{eq-2non}
	N_1:=\lceil Nt_1\rceil.
	\end{equation}
	According to \cite[Theorem \ref{th-exsd}, Proposition \ref{pr-tsh} and Proposition \ref{pr-xfh}]{Dou2018001}, we have
	\begin{equation*}
	\gamma[t_1]=\gamma[t_1^-]\in\mathcal{G}.
	\end{equation*}
	And thanks to (\ref{eq-2fkrt}), it follows that
	\begin{equation*}
	\gamma[t_1]=\gamma[t_1^-]\in\Gamma'.
	\end{equation*}
	Similarly, for each $t_6\in(t_1,t_5)$, since $\gamma(t_6)\in\mathbb{B}_{\gamma[t_1]}^r$, then
	\begin{equation}\label{eq-2fkx}
	\begin{split}
	F_{\gamma[t_6]}^K(x+yK_{N_1}))=&F_{\gamma[t_1]}^K(x+yK_{N_1})
	\\=&f\circ\pi|_{U^K_{\gamma[t_1]}}^{-1}(x+yK_{N_1})
	\end{split}
	\end{equation}
	for each $x,y\in\mathbb{R}$ with $x+yi\in\mathbb{B}_{\gamma[t_1]}^r\cap\mathbb{B}_{\gamma[t_6]}^r$ and $K\in\mathbb{S}^{N_1}$. According to \cite[Theorem \ref{th-exsd}, Proposition \ref{pr-tsh} and Proposition \ref{pr-xfh}]{Dou2018001}, we have
	\begin{equation*}
	\gamma[t_6]=\gamma[t_6^-]\in\mathcal{G}.
	\end{equation*}
	And thanks to (\ref{eq-2fkx}), it follows that
	\begin{equation*}
	\gamma[t_6]=\gamma[t_6^-]\in\Gamma',\qquad\forall\ t_6\in(t_1,t_5).
	\end{equation*}
	Then
	\begin{equation*}
	t_1=\inf\{t\in[0,1]:\gamma[t]\notin\Gamma'\ or\ \gamma[t^-]\notin\Gamma'\}\ge t_5>t_1,
	\end{equation*}
	which is a contradiction. It follows that $\{Nt_1\}=0$.
	
	\textbf{3). We will prove that $\Gamma'=\Gamma$, in this step.}
	
	We can prove that $\gamma[t_1^-]\in\Gamma'$, by the same method to 2). If $t_1=1$, then
	\begin{equation*}
	\gamma=\gamma[t_1^-]\in\Gamma'.
	\end{equation*}
	Then
	\begin{equation*}
	t_1=\inf\{t\in[0,1]:\gamma[t]\notin\Gamma'\ or\ \gamma[t^-]\notin\Gamma'\}\neq 1,
	\end{equation*}
	which is a contradiction.
	
	Otherwise, then $t_1\neq 1$. We notice that
	\begin{equation*}
	\begin{split}
	F^K_{\gamma[t_1]}(x)=&\zeta(K)F_{\gamma[t_1]}(x)
	\\=&\zeta(K)(F_{\gamma[t_1^-]}(x)^T,0_{1\times 2^m})^T
	\\=&\zeta(K^{N_1})F_{\gamma[t_1^-]}(x)
	\\=&F_{\gamma[t_1^-]}^{K^{(N_1)}}(x)
	\\=&f\circ\pi|_{U^{K^{(N_1)}}_{\gamma[t_1^-]}}^{-1}(x)
	\end{split}
	\end{equation*}
	for each $x\in\mathbb{B}_{\gamma[t_1]}^r\cap\mathbb{B}_{\gamma[t_1^-]}^r\cap\mathbb{R}$, and $K\in\mathbb{S}^{N_1+1}$, where $N_1$ is defined by (\ref{eq-2non}).
	
	According to \cite[Theorems \ref{th-idh}, \ref{th-exsd} and Propositions \ref{pr-tsh}, \ref{pr-xfh}]{Dou2018001}, we have
	\begin{equation*}
	\gamma[t_1]\in\Gamma'.
	\end{equation*}
	By the same method to 2), we can get that there exists $t_7\in(t_1,t_1+\frac{1}{N})$ such that
	\begin{equation*}
	\gamma[t_8]=\gamma[t_8^-]\in\Gamma',\qquad\forall\ t_8\in(t_1,t_7).
	\end{equation*}
	Then
	\begin{equation*}
	t_1=\inf\{t\in[0,1]:\gamma[t]\notin\Gamma'\ or\ \gamma[t^-]\notin\Gamma'\}\ge t_7>t_1,
	\end{equation*}
	which is a contradiction.
	
	In summary,
	\begin{equation*}
	\gamma\in\Gamma',\qquad\forall\ \gamma\in\Gamma.
	\end{equation*}
	It follows that $\Gamma'=\Gamma$ and Theorem \ref{th-sfsr} holds.
\end{proof}

\begin{defn}\label{df-srs}
	Let $\mathcal{G}=(G,\pi,x)$ be a $\infty$-axially symmetric slice-domain over $\mathbb{H}$ with distinguished point.
	
	A map $r:\mathcal{P}^{\infty}_{\mathbb{C}}(\mathcal{G})\rightarrow\mathbb{H}$ is called a stem radius system of $\mathcal{G}$, if for each $N\in\mathbb{N}^+$, $\gamma\in\mathcal{P}^N_{\mathbb{C}}(\mathcal{G})$ and $K\in\mathbb{S}^N$, there exists a domain $U^K_{\gamma}$ in $G_{K_N}$ containing $\gamma^K_{\mathcal{G}}(1)$ such that $\pi|_{U^K_{\gamma}}:U^K_{\gamma}\rightarrow(\mathbb{B}_{\gamma}^r)^K$ is a homeomorphism with respect to topologies $\tau(G_{K_N})$ and $\tau(\mathbb{C}_{K_N})$.
	
	We denote $U^K_{\gamma}$ by $(\mathcal{G}_\gamma^r)^K$.
\end{defn}

For each $z,w\in\mathbb{C}$, we define a path $\alpha[z,w]$ in $\mathbb{C}$ by
\begin{equation*}
\alpha[z,w](t):=z+tw,\qquad\forall\ t\in[0,1].
\end{equation*}

For each $z\in\mathbb{C}$ and $\gamma\in\mathcal{P}^N(\mathbb{C})$, we define an $N$-part path $\gamma\langle z\rangle$ in $\mathbb{C}$, by
\begin{equation*}
\gamma\langle z\rangle:=(\gamma_1,\gamma_2,...,\gamma_{N-1},\gamma_N\cdot\alpha[\gamma_N(1),z]).
\end{equation*}

Let $\mathcal{G}=(G,\pi,x)$ be a $\infty$-axially symmetric slice-domain over $\mathbb{H}$ with distinguished point, and $r$ be a stem radius system of $\mathcal{G}$. According to for each $N\in\mathbb{N}^+$, $\gamma\in\mathcal{P}^N_{\mathbb{C}}(\mathcal{G})$ and $K\in\mathbb{S}^N$,
\begin{equation*}
\pi|_{(\mathcal{G}_\gamma^r)^K}:(\mathcal{G}_\gamma^r)^K\rightarrow(\mathbb{B}_{\gamma}^r)^K
\end{equation*}
is a homeomorphism with respect to topologies $\tau(G_{K_N})$ and $\tau(\mathbb{C}_{K_N})$, we have
\begin{equation*}
\gamma\langle z\rangle\prec\mathcal{G},\qquad\forall\ z\in\mathbb{B}_{\gamma}^r.
\end{equation*}

\begin{defn}
	Let $\mathcal{G}=(G,\pi,x)$ be a $\infty$-axially symmetric slice-domain over $\mathbb{H}$ with distinguished point, $f$ be a slice regular function on $G$, and $r$ be a stem radius system of $\mathcal{G}$. We define a map
	\begin{equation*}
	F:\mathcal{P}^{\infty}_{\mathbb{C}}(\mathcal{G})\rightarrow\mathcal{SH}^{\infty}(\mathbb{B}^r_{\mathcal{P}^\infty_{\mathbb{C}}(\mathcal{G})}),\qquad \gamma\mapsto F_{\gamma},\qquad\forall\ \gamma\in\mathcal{P}^{\infty}_{\mathbb{C}}(\mathcal{G}),
	\end{equation*}
	where $F_{\gamma}:\mathbb{B}_{\gamma}^r\rightarrow M_{2^N\times 1}(\mathbb{H})$ defined by
	\begin{equation}\label{eq-2fgz}
	F_{\gamma}(z):=\mathcal{G}^f_{\gamma\langle z\rangle},\qquad\forall\ z\in\mathbb{B}_\gamma^r,
	\end{equation}
	and $\mathcal{G}^f_{\gamma\langle z\rangle}$ is defined by (\ref{eq-gfg}).
	
	We call $F$ is the stem function system of $f$ with respect to $r$, denoted by $f^r$.
\end{defn}

\begin{thm}\label{th-srhs}
	Let $\mathcal{G}=(G,\pi,x_0)$ be a $\infty$-axially symmetric slice-domain over $\mathbb{H}$ with distinguished point, $f$ be a slice regular function on $G$, and $r$ be a stem radius system of $\mathcal{G}$. Then
	\begin{equation*}
	\mathcal{F}:=(f^r,r,\mathcal{P}^{\infty}_{\mathbb{C}}(\mathcal{G}))
	\end{equation*}
	is a holomorphic stem system.
	
	We call $\mathcal{F}$ the holomorphic stem system of $f$ with respect to $r$, denoted by $f^r_{\mathcal{G}}$.
\end{thm}

\begin{proof}
	For each $N\in\mathbb{N}^+$, $K\in\mathbb{S}^N$ and $\gamma\in\mathcal{P}^N_{\mathbb{C}}(\mathcal{G})$, we set
	\begin{equation*}
	U^K_{\gamma}:=(\mathcal{G}^r_{\gamma})^K.
	\end{equation*}
	
	\textbf{1). We will prove that $\mathcal{F}$ is locally holomorphic, in this step.}
	
	Let $N\in\mathbb{N}^+$, $\gamma\in\mathcal{P}^N_{\mathbb{C}}(\mathcal{G})$, $I\in\mathbb{S}$ and
	\begin{equation*}
	J:=\eta_N(I).
	\end{equation*}
	According to (\ref{eq-2fgz}), \cite[Proposition \ref{pr-ct}]{Dou2018001} and Theorem \ref{tm-rf}, we have
	\begin{equation*}
	\begin{split}
	(\frac{\partial}{\partial x}+\sigma_N\frac{\partial}{\partial y})f_{\gamma}^r(x+yi)
	=&(\frac{\partial}{\partial x}+\sigma_N\frac{\partial}{\partial y})\mathcal{G}^f_{\gamma\langle x+yi\rangle}
	\\=&(\frac{\partial}{\partial x}+\sigma_N\frac{\partial}{\partial y})\mathcal{M}(J)^{-1}f(\gamma\langle x+yi\rangle^J_{\mathcal{G}}(1))
	\\=&(\frac{\partial}{\partial x}+\sigma_N\frac{\partial}{\partial y})\mathcal{M}(J)^{-1}F_1(x+yJ)
	\\=&\mathcal{M}(J)^{-1}(\frac{\partial}{\partial x}+D_N(J)\frac{\partial}{\partial y})F_1(x+yJ)
	\\=&\mathcal{M}(J)^{-1}0_{2^N\times 1}
	\\=&0_{2^N\times 1}
	\end{split}
	\end{equation*}
	for each $x,y\in\mathbb{R}$ with $x+yi\in\mathbb{B}_{\gamma}^r$, where
	\begin{equation*}
	f(\gamma\langle x+yi\rangle^J_{\mathcal{G}}(1)):=\left(
	\begin{matrix}
	&f(\gamma\langle x+yi\rangle^{J_1}_{\mathcal{G}}(1))
	\\&f(\gamma\langle x+yi\rangle^{J_2}_{\mathcal{G}}(1))
	\\&\vdots
	\\&f(\gamma\langle x+yi\rangle^{J_{2^N}}_{\mathcal{G}}(1))
	\end{matrix}\right),
	\end{equation*}
	and
	\begin{equation*}
	F_1(x+yJ):=\left(
	\begin{matrix}
	&f\circ\pi|_{U^{J_1}_\gamma}^{-1}(x+yJ_{1,N})
	\\&f\circ\pi|_{U^{J_2}_\gamma}^{-1}(x+yJ_{2,N})
	\\&\vdots
	\\&f\circ\pi|_{U^{J_{2^N}}_\gamma}^{-1}(x+yJ_{2^N,N})
	\end{matrix}\right).
	\end{equation*}
	Thence
	\begin{equation*}
	f^r_\gamma\in\mathcal{SH}^N(\mathbb{B}^r_{\gamma}),\qquad\forall N\in\mathbb{N}^+\ \mbox{and}\ \gamma\in\mathcal{P}^N_{\mathbb{C}}(\mathcal{G}).
	\end{equation*}
	It follows that $\mathcal{F}$ is locally holomorphic (see Definition \ref{df-hssy} (1)).
	
	\textbf{2). We will prove that $\mathcal{F}$ is local compatible, in this step.}
	
	For each $N\in\mathbb{N}^+$, $m\in\{1,2,...,N\}$, $\gamma\in\mathcal{P}^N_{\mathbb{C}}(\mathcal{G})$, and $t_1,t_2\in[\frac{m-1}{N},\frac{m}{N}]$ with $t_1<t_2$. If there exists $t\in[t_1,t_2]$ with
	\begin{equation*}
	\gamma[t_1,t]\subset\mathbb{B}_{\gamma[t_1]}^r\qquad\mbox{and}\qquad\gamma[t,t_2]\subset\mathbb{B}_{\gamma[t_2]}^r,
	\end{equation*}
	then for each $z\in\mathbb{B}_{\gamma[t_1]}^r\cap\mathbb{B}_{\gamma[t_2]}^r$, there exists a path $\alpha$ in $\mathbb{B}_{\gamma[t_1]}^r\cap\mathbb{B}_{\gamma[t_2]}^r$ from $\gamma(t)$ to $z$.
	
	We define a $m$-part path $\beta$ in $\mathbb{C}$ from $\pi(x_0)$ to $z$, by
	\begin{equation*}
	\beta:=(\gamma_1,\gamma_2,...,\gamma_{m-1},\gamma_{(Nt)}\circ\alpha).
	\end{equation*}
	We notice that
	\begin{equation*}
	\beta^K_{\mathcal{G}}(1)=\pi|_{U^K_{\gamma[t_1]}}^{-1}\circ P_{K_m}(z)=\pi|_{U^K_{\gamma[t_2]}}^{-1}\circ P_{K_m}(z),\qquad\forall\ K\in\mathbb{S}^m.
	\end{equation*}
	And according to (\ref{eq-2fgz}), we have
	\begin{equation*}
	\begin{split}
	f^r_{\gamma[t_1]}(z)=&\mathcal{M}(J)F_2(x_1+y_1J)
	\\=&\mathcal{M}(J)f(\beta^J_{\mathcal{G}}(1))
	\\=&\mathcal{M}(J)F_3(x_1+y_1J)
	\\=&f^r_{\gamma[t_2]}(z),
	\end{split}
	\end{equation*}
	where $x_1,y_1\in\mathbb{R}$ with $z=x_1+y_1i$, $I\in\mathbb{S}$, $J=\eta_m(I)$,
	\begin{equation*}
	F_2(x_1+y_1J):=\left(
	\begin{matrix}
	&f\circ\pi|_{U^{J_1}_{\gamma[t_1]}}^{-1}(x_1+y_1J_{1,m})
	\\&f\circ\pi|_{U^{J_2}_{\gamma[t_1]}}^{-1}(x_1+y_1J_{2,m})
	\\&\vdots
	\\&f\circ\pi|_{U^{J_{2^N}}_{\gamma[t_1]}}^{-1}(x_1+y_1J_{2^N,m})
	\end{matrix}\right),
	\end{equation*}
	\begin{equation*}
	f(\beta^J_{\mathcal{G}}(1))=\left(
	\begin{matrix}
	&f(\beta^{J_1}_{\mathcal{G}}(1))
	\\&f(\beta^{J_2}_{\mathcal{G}}(1))
	\\&\vdots
	\\&f(\beta^{J_{2^N}}_{\mathcal{G}}(1))
	\end{matrix}\right),
	\end{equation*}
	and
	\begin{equation*}
	F_3(x_1+y_1J):=\left(
	\begin{matrix}
	&f\circ\pi|_{U^{J_1}_{\gamma[t_2]}}^{-1}(x_1+y_1J_{1,m})
	\\&f\circ\pi|_{U^{J_2}_{\gamma[t_2]}}^{-1}(x_1+y_1J_{2,m})
	\\&\vdots
	\\&f\circ\pi|_{U^{J_{2^N}}_{\gamma[t_2]}}^{-1}(x_1+y_1J_{2^N,m})
	\end{matrix}\right).
	\end{equation*}
	It follows that $f^r_{\gamma[t_1]}$ and $f^r_{\gamma[t_2]}$ are coincide on $\mathbb{B}_{\gamma[t_1]}^r\cap\mathbb{B}_{\gamma[t_2]}^r$. Then $\mathcal{F}$ is local compatible (see Definition \ref{df-hssy} (2)).
	
	\textbf{3). We will prove that $\mathcal{F}$ is axially compatible, in this step.}
	
	For each $N\in\mathbb{N}^+$, $m\in\{1,2,...,N\}$, $\gamma\in\mathcal{P}^N_{\mathbb{C}}(\mathcal{G})$, and
	\begin{equation*}
	x\in\mathbb{B}_{\gamma[\frac{m}{N}]}^r\cap\mathbb{B}_{\gamma[\frac{m}{N}^-]}^r\cap\mathbb{R}.
	\end{equation*}
	Let $I\in\mathbb{S}$, and we set
	\begin{equation*}
	J:=\eta_{m+1}(I)\qquad\mbox{and}\qquad K:=\eta_m(I).
	\end{equation*}
	
	For each $\imath\in\{1,2,...,2^{m+1}\}$, we notice that
	\begin{equation*}
	\gamma[\frac{m}{N}]\langle x\rangle^{J_\imath}_{\mathcal{G}}(1)=\gamma[\frac{m}{N}^-]\langle x\rangle^{J_\imath'}_{\mathcal{G}}(1)\qquad\mbox{and}\qquad\zeta(J_\imath)=(\zeta(J_\imath'),0_{1\times 2^m}).
	\end{equation*}
	where
	\begin{equation*}
	J_\imath'=(J_{\imath,1},J_{\imath,1},...,J_{\imath,m}).
	\end{equation*}
	And according to Theorem \ref{tm-rf}, we have
	\begin{equation*}
	\begin{split}
	&f(\gamma[\frac{m}{N}]\langle x\rangle^J_{\mathcal{G}}(1))
	\\:=&(f(\gamma[\frac{m}{N}]\langle x\rangle^{J_1}_{\mathcal{G}}(1)),f(\gamma[\frac{m}{N}]\langle x\rangle^{J_2}_{\mathcal{G}}(1)),...,f(\gamma[\frac{m}{N}]\langle x\rangle^{J_{2^{m+1}}}_{\mathcal{G}}(1)))^T
	\\=&(f(\gamma[\frac{m}{N}^-]\langle x\rangle^{J_1'}_{\mathcal{G}}(1)),f(\gamma[\frac{m}{N}^-]\langle x\rangle^{J_2'}_{\mathcal{G}}(1)),...,f(\gamma[\frac{m}{N}^-]\langle x\rangle^{J_{2^{m+1}}'}_{\mathcal{G}}(1)))^T
	\\=&(\zeta(J_1')^T,\zeta(J_2')^T,...,\zeta(J_{2^{m+1}}')^T)^T\mathcal{M}(K)f(\gamma[\frac{m}{N}^-]\langle x\rangle^K_{\mathcal{G}}(1))
	\\=&\mathcal{M}(J)(\mathbb{I}_{2^m},0_{2^m})^T\mathcal{M}(K)f(\gamma[\frac{m}{N}^-]\langle x\rangle^K_{\mathcal{G}}(1)),
	\end{split}
	\end{equation*}
	where
	\begin{equation*}
	f(\gamma[\frac{m}{N}^-]\langle x\rangle^K_{\mathcal{G}}(1)):=\left(
	\begin{matrix}
	&f(\gamma[\frac{m}{N}^-]\langle x\rangle^{K_1}_{\mathcal{G}}(1))
	\\&f(\gamma[\frac{m}{N}^-]\langle x\rangle^{K_2}_{\mathcal{G}}(1))
	\\&\vdots
	\\&f(\gamma[\frac{m}{N}^-]\langle x\rangle^{K_{2^m}}_{\mathcal{G}}(1))
	\end{matrix}\right).
	\end{equation*}
	And thanks to (\ref{eq-mrk}), (\ref{eq-gfg}) and (\ref{eq-2fgz}), it is clear that
	\begin{equation*}
	\begin{split}
	f^r_{\gamma[\frac{m}{N}]}(x)=&\mathcal{G}^f_{\gamma[\frac{m}{N}]\langle x\rangle}
	\\=&\mathcal{M}(J)^{-1}f(\gamma[\frac{m}{N}]\langle x\rangle^J_{\mathcal{G}}(1))
	\\=&\mathcal{M}(J)^{-1}\mathcal{M}(J)(\mathbb{I}_{2^m},0_{2^m})^T\mathcal{M}(K)f(\gamma[\frac{m}{N}^-]\langle x\rangle^K_{\mathcal{G}}(1))
	\\=&((\mathcal{M}(K)f(\gamma[\frac{m}{N}]\langle x\rangle^K_{\mathcal{G}}(1)))^T,0_{1\times 2^m})^T
	\\=&(f^r_{\gamma[\frac{m}{N}^-]}(x)^T,0_{1\times 2^m})^T
	\end{split}
	\end{equation*}
	for each $x\in\mathbb{B}_{\gamma[\frac{m}{N}]}^r\cap\mathbb{B}_{\gamma[\frac{m}{N}^-]}^r\cap\mathbb{R}$.
	
	Then $\mathcal{F}$ is axially compatible (see Definition \ref{df-hssy} (3)).
	
	\textbf{4). We will prove that $\mathcal{F}$ is initially compatible, in this step.}
	
	For each $x\in\mathbb{B}_{\gamma[0]}^r\cap\mathbb{R}$, and $I\in\mathbb{S}$, according to (\ref{eq-mrk}), (\ref{eq-gfg}) and (\ref{eq-2fgz}), we have
	\begin{equation*}
	\begin{split}
	f^r_{\gamma[0]}(x)=&\mathcal{G}^f_{\gamma[0]\langle x\rangle}
	\\=&
	\left(\begin{matrix}
	1&I\\1&-I
	\end{matrix}\right)^{-1}
	\left(\begin{matrix}
	f(\gamma[0]\langle x\rangle^I_{\mathcal{G}}(1))\\f(\gamma[0]\langle x\rangle^{-I}_{\mathcal{G}}(1))
	\end{matrix}\right)
	\\=&\frac{1}{2}
	\left(\begin{matrix}
	1&1\\-I&I
	\end{matrix}\right)
	\left(\begin{matrix}
	f(\gamma[0]\langle x\rangle^I_{\mathcal{G}}(1))\\f(\gamma[0]\langle x\rangle^{I}_{\mathcal{G}}(1))
	\end{matrix}\right)
	\\=&(f(\gamma[0]\langle x\rangle^I_{\mathcal{G}}(1)),0)^T.
	\end{split}
	\end{equation*}
	Then $\mathcal{F}$ is initially compatible.
	
	In summary, $\mathcal{F}$ is a holomorphic stem function system.
\end{proof}

\begin{defn}\label{df-2sds}
	Let $\mathcal{G}=(G,\pi,x)$ be a slice-domain over $\mathbb{H}$ with distinguished point, and $r$ be a stem radius system of $\mathcal{G}$. We set
	\begin{equation*}
	\mathcal{HS}^r(\mathcal{G}):=\{(F,r,\mathcal{P}^\infty_{\mathbb{C}}(\mathcal{G}))\in\mathcal{HS}(\mathcal{G})|F:\mathcal{P}^\infty_{\mathbb{C}}(\mathcal{G})\rightarrow\mathcal{SH}^\infty(\mathbb{B}^r_{\mathcal{P}^\infty_{\mathbb{C}}(\mathcal{G})})\ is\ a\ map\},
	\end{equation*}
	and
	\begin{equation*}
	\mathcal{I}_{\mathcal{G}}^r:=\mathcal{I}_{\mathcal{G}}|_{\mathcal{HS}^r(\mathcal{G})}:\mathcal{HS}^r(\mathcal{G})\rightarrow \mathcal{SR}(G),
	\end{equation*}
	where $\mathcal{HS}(\mathcal{G})$ is defined in Definition \ref{df-hss}.
\end{defn}

Let $\mathcal{G}=(G,\pi,x_0)$ be a slice-domain over $\mathbb{H}$ with distinguished point, $r$ be a stem radius system of $\mathcal{G}$, and
\begin{equation*}
\mathcal{F}_\imath=(F_\imath,r,\mathcal{P}^\infty_{\mathbb{C}}(\mathcal{G}))\in\mathcal{HS}(\mathcal{G}),\qquad\imath=1,2.
\end{equation*}
We define $F_1+F_2:\mathcal{P}^\infty_{\mathbb{C}}(\mathcal{G})\rightarrow\mathcal{SH}^{\infty}(\mathbb{B}_{\mathcal{P}^\infty_{\mathbb{C}}(\mathcal{G})}^r)$ by
\begin{equation*}
(F_1+F_2)(\gamma):=F_1(\gamma)+F_2(\gamma).
\end{equation*}
We can prove that $(F_1+F_2,r,\mathcal{P}^\infty_{\mathbb{C}}(\mathcal{G}))$ is a holomorphic stem system, by direct verification.

We call $(F_1+F_2,r,\mathcal{P}^\infty_{\mathbb{C}}(\mathcal{G}))$ the sum of $\mathcal{F}_1$ and $\mathcal{F}_2$, denoted by $\mathcal{F}_1+\mathcal{F}_2$.

\begin{prop}\label{pr-is}
	Let $\mathcal{G}=(G,\pi,x_0)$ be a slice-domain over $\mathbb{H}$ with distinguished point, and $r$ be a stem radius system of $\mathcal{G}$. Then $(\mathcal{HS}^r(\mathcal{G}),+)$ is an Abelian group.
	
	Moreover, if $\mathcal{G}$ is $\infty$-axially symmetric, then $\mathcal{I}_{\mathcal{G}}^r$ is a group isomorphic between $(\mathcal{HS}^r(\mathcal{G}),+)$ and $(\mathcal{SR}(G),+)$.
\end{prop}

\begin{proof}
	1. Let $\mathcal{F}_\imath=(F_\imath,r,\mathcal{P}^\infty_{\mathbb{C}}(\mathcal{G}))\in\mathcal{HS}^r(G)$, $\imath=1,2$. Thanks to (\ref{eq-gfg}), (\ref{eq-2fkr}) and (\ref{eq-2fgz}), we have
	\begin{equation*}
	\begin{split}
	&(\mathcal{I}_{\mathcal{G}}^r(\mathcal{F}_1)+\mathcal{I}_{\mathcal{G}}^r(\mathcal{F}_2))\circ\pi|_{(\mathcal{G}^r_\gamma)^K}^{-1}(x+yK_N)
	\\=&\mathcal{I}_{\mathcal{G}}^r(\mathcal{F}_1)\circ\pi|_{(\mathcal{G}^r_\gamma)^K}^{-1}(x+yK_N)+\mathcal{I}_{\mathcal{G}}^r(\mathcal{F}_2)\circ\pi|_{(\mathcal{G}^r_\gamma)^K}^{-1}(x+yK_N)
	\\=&\zeta(K)(F_1)_\gamma(x+yi)+\zeta(K)(F_2)_\gamma(x+yi)
	\\=&\zeta(K)((F_1)_\gamma+(F_2)_\gamma)(x+yi)
	\\=&\zeta(K)(F_1+F_2)_\gamma(x+yi)
	\\=&(F_1+F_2)^K_\gamma(x+yK_N)
	\end{split}
	\end{equation*}
	for each $N\in\mathbb{N}^+$, $K\in\mathbb{S}^N$, $\gamma\in\mathcal{P}^N_{\mathbb{C}}(\mathcal{G})$, and $x,y\in\mathbb{R}$ with $x+yi=\gamma(1)\in\mathbb{B}_\gamma^r$. Then
	\begin{equation*}
	(\mathcal{I}_{\mathcal{G}}^r(\mathcal{F}_1)+\mathcal{I}_{\mathcal{G}}^r(\mathcal{F}_2))\circ\pi|_{(\mathcal{G}^r_\gamma)^K}^{-1}=(F_1+F_2)^K_\gamma.
	\end{equation*}
	Notice that $\mathcal{I}_{\mathcal{G}}^r(\mathcal{F}_1)+\mathcal{I}_{\mathcal{G}}^r(\mathcal{F}_2)$ is a slice regular function on $G$, and according to Definitions \ref{df-hss} and \ref{df-2sds}, we have
	\begin{equation*}
	\mathcal{F}_1+\mathcal{F}_2=(F_1+F_2,r,\mathcal{P}^\infty_{\mathbb{C}}(\mathcal{G}))\in\mathcal{HS}^r(G)
	\end{equation*}
	and
	\begin{equation*}
	\mathcal{I}_{\mathcal{G}}^r(\mathcal{F}_1)+\mathcal{I}_{\mathcal{G}}^r(\mathcal{F}_2)=\mathcal{I}_{\mathcal{G}}^r(\mathcal{F}_1+\mathcal{F}_2).
	\end{equation*}
	
	According to $F_1+F_2=F_2+F_1$, it follows that $\mathcal{F}_1+\mathcal{F}_2=\mathcal{F}_2+\mathcal{F}_1$. Then $(\mathcal{HS}^r(\mathcal{G}),+)$ is an Abelian group.
	
	2. According to Theorem \ref{th-srhs}, $\mathcal{I}_{\mathcal{G}}^r$ is a surjection.
	
	3. If $\mathcal{I}_{\mathcal{G}}^r(\mathcal{F})=0_{\mathcal{SR}(G)}$, for some $\mathcal{F}=(F,r,\mathcal{P}^\infty_{\mathbb{C}}(\mathcal{G}))\in\mathcal{SR}^r(G)$, then
	\begin{equation*}
	\zeta(K)F_\gamma^r(x+yi)=0_{\mathcal{SR}(G)}\circ\pi|_{(\mathcal{G}^r_\gamma)^K}^{-1}(x+yK_N)=0,
	\end{equation*}
	for each $N\in\mathbb{N}^+$, $K\in\mathbb{S}^N$, $\gamma\in\mathcal{P}^N_{\mathbb{C}}(\mathcal{G})$, and $x,y\in\mathbb{R}$ with $x+yi\in\mathbb{B}_\gamma^r$. It follows that
	\begin{equation*}
	\begin{split}
	&F_\gamma^r(x+yi)
	\\=&\mathcal{M}(J)^{-1}\mathcal{M}(J)F_\gamma^r(x+yi)
	\\=&\mathcal{M}(J)^{-1}(\zeta(J_1)F_\gamma^r(x+yi),\zeta(J_2)F_\gamma^r(x+yi),...,\zeta(J_{2^N})F_\gamma^r(x+yi))^T
	\\=&0_{2^N\times 1}
	\end{split}
	\end{equation*}
	for each $N\in\mathbb{N}^+$, $I\in\mathbb{S}$, $J=\eta_N(I)$, $\gamma\in\mathcal{P}^N_{\mathbb{C}}(\mathcal{G})$, and $x,y\in\mathbb{R}$ with $x+yi\in\mathbb{B}_\gamma^r$. Then
	\begin{equation*}
	\mathcal{F}=0_{\mathcal{HS}^r(\mathcal{G})}.
	\end{equation*}
	It follows that $\mathcal{I}_{\mathcal{G}}^r$ is an injection.
\end{proof}

\section{$*$-product of slice regular functions}\label{sc-psrf}

The $*$-product is introduced in \cite{Gentili2008001}, and has been extended to axially symmetric slice domain in $\mathbb{H}$ (see \cite{Colombo2009001}). We will discuss $*$-product on $*$-preserving slice-domains over $\mathbb{H}$ with distinguished point, by the tensor product (following \cite{Ghiloni2011001}), in this section.

For each $N\in\mathbb{N}^+$, we consider the tensor product $\mathbb{C}^{\otimes N}\otimes\mathbb{H}$ over $\mathbb{R}$. Let $ab$ (or $a\cdot b$) be the product of $a$ and $b$, defined by
\begin{equation*}
ab:=\otimes_{\imath=1}^{N+1}(a_\imath b_\imath)=(a_1b_1)\otimes(a_2b_2)\otimes...\otimes(a_{N+1}b_{N+1}),
\end{equation*}
where $a=\otimes_{\imath=1}^{N+1}a_\imath\in\mathbb{C}^{\otimes N}\otimes\mathbb{H}$ and $b=\otimes_{\imath=1}^{N+1}b_\imath\in\mathbb{C}^{\otimes N}\otimes\mathbb{H}$. For each $q\in\mathbb{H}$, we set
\begin{equation*}
q_{\mathbb{H}}:=1^{\otimes N}\otimes q,
\end{equation*}
abbreviated as $q$ without ambiguity. We set
\begin{equation*}
i_{N,\jmath}:=1^{\otimes \jmath-1}\otimes i\otimes 1^{\otimes N+1-\jmath},\qquad\forall\ \jmath\in\{1,2,...,2^N\},
\end{equation*}
and
\begin{equation*}
i_N(m):=\prod_{\imath=N}^1(i_{N,\imath}i_{N,\imath-1})^{m_{\imath}},\qquad\forall\ m\in\{1,2,...,2^N\},
\end{equation*}
where $i_0:=1$, and $(m_Nm_{N-1}...m_1)_2$ is the binary number of $m-1$. We set
\begin{equation*}
\zeta_N(i):=(i_N(1),i_N(2),...,i_N(2^N)).
\end{equation*}
We also have
\begin{equation}\label{eq-tt}
i_{N,N}\zeta_N(i)=\zeta_N(i)\sigma_N,
\end{equation}
by the same method of Lemma \cite[Proposition \ref{le-ct}]{Dou2018001}.

For each $N\in\mathbb{N}^+$ and $a=(a_1,a_2,...,a_{2^N})^T\in M_{2^N\times 1}(\mathbb{H})$, we define a map $\mathcal{I}_N:M_{2^N\times 1}(\mathbb{H})\rightarrow\mathbb{C}^{\otimes N}\otimes\mathbb{H}$, by
\begin{equation}\label{eq-2ti}
\mathcal{I}_N:M_{2^N\times 1}(\mathbb{H})\rightarrow\mathbb{C}^{\otimes N}\otimes\mathbb{H},\qquad a\mapsto\zeta_N(i)a_{\mathbb{H}},
\end{equation}
where
\begin{equation*}
a_{\mathbb{H}}:=((a_1)_{\mathbb{H}},(a_2)_{\mathbb{H}},...,(a_{2^N})_{\mathbb{H}})^T.
\end{equation*}
Then
\begin{equation*}
\mathcal{I}_N(e_{N,\imath})=i_N(\imath),\qquad\forall\ \imath\in\{1,2,..,2^N\},
\end{equation*}
where
\begin{equation*}
e_{N,\imath}:=(0_{1\times(\imath-1)},1,0_{1\times(N-\imath+1)})^T.
\end{equation*}

We notice that $\{e_{N,\imath}\}_{\imath=1}^{2^N}$ is a basis over $\mathbb{H}$ of $M_{2^N\times 1}(\mathbb{H})$, and $\{i_N(\imath)\}_{\imath=1}^{2^N}$ a basis over $1^{\otimes N}\otimes\mathbb{H}\cong\mathbb{H}$ of $\mathbb{C}^{\otimes N}\otimes\mathbb{H}$. Then

\begin{prop}\label{pr-vt}
	 Let $N\in\mathbb{N}^+$. Then $\mathcal{I}_N$ is an isomorphism between $(M_{2^N\times 1}(\mathbb{H}),+,\cdot_{\mathbb{H}})$ and $(\mathbb{C}^{\otimes N}\otimes\mathbb{H},+,\cdot_{\mathbb{H}})$, where $\cdot_{\mathbb{H}}$ is the two-sided scalar multiplication between $\mathbb{H}$ (resp. $1^{\otimes N}\otimes\mathbb{H}$) and $M_{2^N\times 1}(\mathbb{H})$ (resp. $\mathbb{C}^{\otimes N}\otimes\mathbb{H}$).
	
	 Similarly $\mathcal{I}_N^{-1}$ exists, and $\mathcal{I}_N^{-1}$ is also an isomorphism between $(\mathbb{C}^{\otimes N}\otimes\mathbb{H},+,\cdot_{\mathbb{H}})$ and $(M_{2^N\times 1}(\mathbb{H}),+,\cdot_{\mathbb{H}})$.
\end{prop}

\begin{proof}
	This proposition is proved directly by the definition.
\end{proof}

Let $N\in\mathbb{N}^+$. According to (\ref{eq-tt}) and (\ref{eq-2ti}), for each $a\in M_{2^N\times 1}(\mathbb{H})$, we have
\begin{equation}\label{eq-tt2}
\begin{split}
i_{N,N}\mathcal{I}_N(a)=&i_{N,N}\zeta_N(i)a_{\mathbb{H}}
\\=&\zeta_N(i)\sigma_Na_{\mathbb{H}}
\\=&\zeta_N(i)(\sigma_Na)_{\mathbb{H}}
\\=&\mathcal{I}_N(\sigma_N a).
\end{split}
\end{equation}
For each $a\in\mathbb{C}^{\otimes N}\otimes\mathbb{H}$, we also have
\begin{equation}\label{eq-tt1}
\begin{split}
\mathcal{I}_N^{-1}(i_{N,N}a)=&\mathcal{I}_N^{-1}(i_{N,N}\mathcal{I}_N\circ\mathcal{I}_N^{-1}(a))
\\=&\mathcal{I}_N^{-1}(\mathcal{I}_N(\sigma_N\circ\mathcal{I}_N^{-1}(a)))
\\=&\sigma_N\circ\mathcal{I}_N^{-1}(a).
\end{split}
\end{equation}

\begin{defn}
	Let $N\in\mathbb{N}^+$ and $\Omega$ be a domain in $\mathbb{C}$. A function $f:\Omega\rightarrow\mathbb{C}^{\otimes N}\otimes\mathbb{H}$ is called tensor ($N$-)holomorphic or (finite-)holomorphic, if $f$ has continuous partial derivatives and satisfies
	\begin{equation*}
		(\frac{\partial}{\partial x}+i_{N,N}\frac{\partial}{\partial y})f(x+yi)=0_{\mathbb{H}}
	\end{equation*}
	for each $x,y\in\mathbb{R}$ with $x+yi\in\Omega$.
	
	For each $N\in\mathbb{N}^+$, we denote the set of all the tensor $N$-holomorphic function on $\Omega$, by $\mathcal{TH}^N(\Omega)$.
	
	We denote the set of all the tensor finite-holomorphic functions, by $\mathcal{TH}^{\infty}(\Omega)$, defined by
	\begin{equation*}
	\mathcal{TH}^{\infty}(\Omega):=\bigcup_{m\in\mathbb{N}^+}(\mathcal{TH}^m(\Omega)).
	\end{equation*}
\end{defn}

\begin{prop}\label{pr-shths}
	Let $N\in\mathbb{N}^+$ and $\Omega$ be a domain in $\mathbb{C}$. A function $f:\Omega\rightarrow\mathbb{C}^{\otimes N}\otimes\mathbb{H}$ is tensor $N$-holomorphic if and only if, $\mathcal{I}_N^{-1}\circ f$ is stem $N$-holomorphic.
\end{prop}

\begin{proof}
	If $f$ is tensor $N$-holomorphic, then according to (\ref{eq-tt1}),
	\begin{equation*}
	\begin{split}
	(\frac{\partial}{\partial x}+\sigma_N\frac{\partial}{\partial y})\mathcal{I}_N^{-1}\circ f(x+yi)
	=&\mathcal{I}_N^{-1}((\frac{\partial}{\partial x}+i_{N,N}\frac{\partial}{\partial y})f(x+yi))
	\\=&\mathcal{I}_N^{-1}(0_{\mathbb{H}})
	\\=&0_{2^N\times 1}
	\end{split}
	\end{equation*}
	for each $x,y\in\mathbb{R}$ with $x+yi\in\Omega$. Conversely, if $\mathcal{I}_N^{-1}\circ f$ is stem $N$-holomorphic, then according to (\ref{eq-tt2}),
	\begin{equation*}
	\begin{split}
	(\frac{\partial}{\partial x}+i_{N,N}\frac{\partial}{\partial y})f(x+yi)
	=&(\frac{\partial}{\partial x}+i_{N,N}\frac{\partial}{\partial y})\mathcal{I}_N\circ\mathcal{I}_N^{-1}\circ f((x+yi)
	\\=&\mathcal{I}_N((\frac{\partial}{\partial x}+\sigma_N\frac{\partial}{\partial y})\mathcal{I}_N^{-1}\circ f((x+yi))
	\\=&\mathcal{I}_N(0_{2^N\times 1})
	\\=&0_{\mathbb{H}}
	\end{split}
	\end{equation*}
	for each $x,y\in\mathbb{R}$ with $x+yi\in\Omega$.
\end{proof}

\begin{prop}\label{pr-thsh}
	Let $N\in\mathbb{N}^+$, $\Omega$ be a domain in $\mathbb{C}$, and $f,g:\Omega\rightarrow\mathbb{C}^{\otimes N}\otimes\mathbb{H}$ are two tensor $N$-holomorphic functions. Then $fg:\Omega\rightarrow\mathbb{C}^{\otimes N}\otimes\mathbb{H}$ is also a tensor $N$-holomorphic function, where $fg$ is defined by
	\begin{equation*}
	fg(z):=f(z)g(z),\qquad\forall\ z\in\Omega.
	\end{equation*}
\end{prop}

\begin{proof}
	We notice that $i_{N,N}\cdot a=a\cdot i_{N,N}$ for each $a\in\mathbb{C}^{\otimes N}\otimes\mathbb{H}$. Then
	\begin{equation*}
	\begin{split}
	&(\frac{\partial}{\partial x}+i_{N,N}\frac{\partial}{\partial y})fg(x+yi)
	\\=&((\frac{\partial}{\partial x}+i_{N,N}\frac{\partial}{\partial y})f(x+yi))g(x+yi)
	+f(x+yi)((\frac{\partial}{\partial x}+i_{N,N}\frac{\partial}{\partial y})g(x+yi))
	\\=&0_{\mathbb{H}}
	\end{split}
	\end{equation*}
	for each $x,y\in\mathbb{R}$ with $x+yi\in\Omega$.
\end{proof}

\begin{defn}
	Let $N\in\mathbb{N}^+$, $\Omega$ be a domain in $\mathbb{C}$, and $f,g:\Omega\rightarrow M_{2^N\times 1}(\mathbb{H})$ are two $M_{2^N\times 1}(\mathbb{H})$-value functions on $\Omega$. We call
	\begin{equation*}
	f*g:=\mathcal{I}_N^{-1}(\mathcal{I}_N(f)\mathcal{I}_N(g))
	\end{equation*}
	the slice product of $f$ and $g$. This product is called the $*$-product or slice product.
\end{defn}

For each $N\in\mathbb{N}$ and $a,b\in M_{2^N\times 1}(\mathbb{H})$, we set
\begin{equation}\label{eq-sm}
a*b:=\mathcal{I}_N^{-1}(\mathcal{I}_N(a)\mathcal{I}_N(b)).
\end{equation}

Let $N\in\mathbb{N}^+$, $\Omega$ be a domain in $\mathbb{C}$, and $f,g:\Omega\rightarrow M_{2^N\times 1}(\mathbb{H})$ are two $M_{2^N\times 1}(\mathbb{H})$-value functions on $\Omega$. Obviously,
\begin{equation*}
f*g(z)=f(z)*g(z),\qquad\forall\ z\in\Omega.
\end{equation*}

\begin{prop}\label{pr-sh}
	Let $N\in\mathbb{N}^+$, $\Omega$ is a domain in $\mathbb{C}$, and $f,g$ are two stem $N$-holomorphic function on $\Omega$. Then $f*g$ is also a stem $N$-holomorphic function.
\end{prop}

\begin{proof}
	This proposition follows immediately from Proposition \ref{pr-shths} and \ref{pr-thsh}.
\end{proof}

\begin{thm}\label{th-sp}
	Let $\mathcal{F}_1=(F_1,r,\Gamma)$ and $\mathcal{F}_2=(F_2,r,\Gamma)$ be two holomorphic stem systems. Then
	\begin{equation*}
	\mathcal{F}:=(F,r,\Gamma)
	\end{equation*}
	is also a holomorphic stem system, where $F:\Gamma\rightarrow\mathcal{SH}^{\infty}(\mathbb{B}_{\Gamma}^r)$ is the map with
	\begin{equation*}
	F_\gamma:=F(\gamma)\qquad\mbox{and}\qquad F_\gamma=(F_1)_\gamma*(F_2)_\gamma,\qquad\forall\ \gamma\in\Gamma.
	\end{equation*}
	
	We denote $F$ by $F_1*F_2$.
	
	We also denote $\mathcal{F}$ by $\mathcal{F}_1*\mathcal{F}_2$.
\end{thm}

\begin{proof}
	1). According to Proposition \ref{pr-sh}, $\mathcal{F}$ is locally holomorphic.
	
	2). For each $N\in\mathbb{N}^+$, $N$-part path $\gamma$ in $\Gamma$, $m\in\{1,2,...,N\}$, and $t_1,t_2\in[\frac{m-1}{N},\frac{m}{N}]$ with $t_1<t_2$, if there exists a real number $t\in[t_1,t_2]$ such that
	\begin{equation*}
	\gamma([t_1,t])\subset\mathbb{B}_{\gamma[t_1]}^r\qquad\mbox{and}\qquad\gamma([t,t_2])\subset\mathbb{B}_{\gamma[t_2^-]}^r.
	\end{equation*}
	Then
	\begin{equation*}
	(F_1)_{\gamma[t_1]}(z)=(F_1)_{\gamma[t_2^-]}(z)\qquad\mbox{and}\qquad (F_2)_{\gamma[t_1]}(z)=(F_2)_{\gamma[t_2^-]}(z)
	\end{equation*}
	for each $z\in\mathbb{B}_{\gamma[t_1]}^r\cap\mathbb{B}_{\gamma[t_2^-]}^r$. It follows that
	\begin{equation*}
	\begin{split}
	(F_1)_{\gamma[t_1]}(z)=&(F_1)_{\gamma[t_1]}(z)*(F_2)_{\gamma[t_1]}(z)
	\\=&(F_1)_{\gamma[t_2^-]}(z)*(F_2)_{\gamma[t_2^-]}(z)
	\\=&(F)_{\gamma[t_2^-]}(z)
	\end{split}
	\end{equation*}
	for each $z\in\mathbb{B}_{\gamma[t_1]}^r\cap\mathbb{B}_{\gamma[t_2^-]}^r$. Then $\mathcal{F}$ is locally compatible.
	
	3). According to (\ref{eq-2ti}) and (\ref{eq-sm}), for each $a,b\in M_{2^{N-1}\times 1}(\mathbb{H})$, we have
	\begin{equation}\label{eq-dsm}
	\begin{split}
	&(a^T,0_{1\times 2^{N-1}})^T*(b^T,0_{1\times 2^{N-1}})^T
	\\=&\mathcal{I}_N^{-1}(\mathcal{I}_N((a^T,0_{1\times 2^{N-1}})^T)\cdot\mathcal{I}_N((b^T,0_{1\times 2^{N-1}}))^T)
	\\=&\mathcal{I}_N^{-1}((\zeta_N(i)(a^T,0_{1\times 2^{N-1}})^T_{\mathbb{H}})\cdot(\zeta_N(i)(b^T,0_{1\times 2^{N-1}})^T_{\mathbb{H}}))
	\\=&\mathcal{I}_N^{-1}((\zeta_{N-1}(i)a_{\mathbb{H}})\cdot(\zeta_{N-1}(i)b_{\mathbb{H}}))
	\\=&\mathcal{I}_N^{-1}(\mathcal{I}_{N-1}\circ\mathcal{I}_{N-1}^{-1}(\mathcal{I}_{N-1}(a)\cdot\mathcal{I}_{N-1}(b)))
	\\=&\mathcal{I}_N^{-1}(\zeta_{N-1}(i)(a*b)_{\mathbb{H}})
	\\=&\mathcal{I}_N^{-1}(\zeta_N(i)((a*b)^T,0_{1\times 2^{N-1}})^T_{\mathbb{H}})
	\\=&((a*b)^T,0_{1\times 2^{N-1}})^T.
	\end{split}
	\end{equation}
	It follows that
	\begin{equation*}
	\begin{split}
	F_{\gamma[\frac{m}{N}]}(x)=&(F_1)_{\gamma[\frac{m}{N}]}(x)*(F_2)_{\gamma[\frac{m}{N}]}(x)
	\\=&(((F_1)_{\gamma[\frac{m}{N}^{-1}]}(x))^T,0_{1\times 2^{N-1}})^T*(((F_2)_{\gamma[\frac{m}{N}^{-1}]}(x))^T,0_{1\times 2^{N-1}})^T
	\\=&(((F_1)_{\gamma[\frac{m}{N}^{-1}]}(x)*(F_2)_{\gamma[\frac{m}{N}^{-1}]}(x))^T,0_{1\times 2^{N-1}})^T
	\\=&(F_{\gamma[\frac{m}{N}^{-1}]}(x)^T,0_{1\times 2^{N-1}})^T
	\end{split}
	\end{equation*}
	for each $N\in\mathbb{N}^+$, $m\in\{1,2,...,N-1\}$, $\gamma\in\Gamma$, and $x\in\mathbb{B}_{\gamma[\frac{m}{N}]}^r\cap\mathbb{B}_{\gamma[\frac{m}{N}^-]}^r\cap\mathbb{R}$. Then $\mathcal{F}$ is axially compatible.
	
	4). Since $\mathcal{F}_\imath$ is initially compatible, for each $\imath\in\{1,2\}$ and $\gamma\in\Gamma$, there exists a function $f_\imath:\mathbb{B}_{\gamma[0]}^r\cap\mathbb{R}\rightarrow\mathbb{H}$ such that
	\begin{equation*}
	(F_\imath)_{\gamma[0]}(x)=(f_\imath(x),0)^T,\qquad\forall\ x\in\mathbb{B}_{\gamma[0]}^r\cap\mathbb{R}.
	\end{equation*}
	Let $\alpha$ be the initial path of $\Gamma$, defined by (\ref{eq-2asx}). According to (\ref{eq-dsm}), we have
	\begin{equation*}
	\begin{split}
	F_\alpha(x)=&(F_1)_\alpha(x)*(F_2)_\alpha(x)
	\\=&(f_1(x),0)^T*(f_2(x),0)^T
	\\=&(f_1(x)*f_2(x),0)^T
	\end{split}
	\end{equation*}
	for each $x\in\mathbb{B}_\alpha^r\cap\mathbb{R}$. Then $\mathcal{F}$ is initially compatible. In summary, $\mathcal{F}$ is a holomorphic stem function system.
\end{proof}

\begin{defn}\label{df-psd}
	A slice-domain $\mathcal{G}=(G,\pi,x)$ over $\mathbb{H}$ with distinguished point is called $*$-preserving, if the following properties hold:	
	\begin{enumerate}
		\item $\mathcal{G}$ is $\infty$-axially symmetric.
		
		\item For each $f,g\in\mathcal{SR}(G)$,	there exists $F\in\mathcal{SR}(G)$ such that $F|_{G^x_{\mathbb{R}}}=f|_{G^x_{\mathbb{R}}}g|_{G^x_{\mathbb{R}}}$, where $G^x_{\mathbb{R}}$ is the connected component in $G_{\mathbb{R}}$ containing $x$.
	\end{enumerate}
	Thanks to the \cite[Identity Principle \ref{th-idh}]{Dou2018001}, we can denote
	\begin{equation*}
	f*g:=F.
	\end{equation*}
\end{defn}

\begin{prop}\label{pr-hsr}
	Let $\mathcal{G}=(G,\pi,x_0)$ be a $*$-preserving slice-domain over $\mathbb{H}$ with distinguished point, and $r$ be a stem radius system of $\mathcal{G}$. Then $(\mathcal{HS}^r(\mathcal{G}),+,*)$ is a ring with the identity element $(1_{\mathcal{SR}(G)})^r_\mathcal{G}$, where $1_{\mathcal{SR}(G)}:G\rightarrow\mathbb{H}$ is the identity element of $(\mathcal{SR}(G),+,*)$, defined by
	\begin{equation*}
	1_{\mathcal{SR}(G)}(q):=1,\qquad\forall\ q\in G.
	\end{equation*}
\end{prop}

\begin{proof}
	\textbf{1). $\mathcal{HS}^r(\mathcal{G})$ is closed under the multiplication $*$.}
	
	Let
	\begin{equation*}
	\mathcal{F}_\imath=(F_\imath,r,\mathcal{P}^\infty_{\mathbb{C}}(\mathcal{G}))\in\mathcal{HS}^r(\mathcal{G}),\qquad\imath=1,2,3.
	\end{equation*}
	According to Theorem \ref{th-sp}, $\mathcal{F}_1*\mathcal{F}_2$ is a holomorphic stem system. And thanks to Theorem \ref{th-sfsr}, there exists a slice-domain $\mathcal{G}'=(G',\pi',x_0')$ over $\mathbb{H}$ with distinguished point such that
	\begin{equation*}
	\mathcal{F}_1*\mathcal{F}_2\prec\mathcal{G}'.
	\end{equation*}
	Let $\alpha$ be the initial path of $\Gamma$, defined by (\ref{eq-2asx}). According to Definition \ref{df-hss} (1), we have
	\begin{equation*}
	\mathcal{P}^\infty_{\mathbb{C}}(\mathcal{G})\subset\mathcal{P}^\infty_{\mathbb{C}}(\mathcal{G}').
	\end{equation*}
	Obviously, there exists a stem radius system $r'$ of $\mathcal{G}'$ with
	\begin{equation*}
	r'|_{\mathcal{P}^\infty_{\mathbb{C}}(\mathcal{G})}=r.
	\end{equation*}
	Let $f_\imath:\mathbb{B}_\alpha^r\rightarrow\mathbb{H}$ be a function such that
	\begin{equation}\label{eq-2fir}
	(F_\imath)_\alpha^r=(f_\imath,0)^T,
	\end{equation}
	where $\imath=1,2$. We set
	\begin{equation*}
	\mathcal{G}^r_\alpha:=\bigcup_{I\in\mathbb{S}}(\mathcal{G}^r_\alpha)^I,\qquad(\mbox{resp.}\quad(\mathcal{G}')^r_\alpha:=\bigcup_{I\in\mathbb{S}}((\mathcal{G}')^r_\alpha)^I)
	\end{equation*}
	where $(\mathcal{G}^r_\alpha)^I$ (resp. $((\mathcal{G}')^r_\alpha)^I$) is a domain in $G_I$ (resp. $G'_I$) defined by Definition \ref{df-srs}. It is trivial to prove that $\mathcal{G}^r_\alpha$ (resp. $(\mathcal{G}')^r_\alpha$) is a domain in $G$ (resp. $G'$) and
	\begin{equation*}
	\pi|_{\mathcal{G}^r_\alpha}:\mathcal{G}^r_\alpha\rightarrow B_\mathbb{H}(\pi(x_0),r_\alpha)\qquad(\mbox{resp.}\quad\pi'|_{(\mathcal{G}')^r_\alpha}:(\mathcal{G}')^r_\alpha\rightarrow B_\mathbb{H}(\pi(x_0),r_\alpha))
	\end{equation*}
	is a slice-homeomorphism, by \cite[Proposition \ref{pr-ul}]{Dou2018001}. Thanks to (\ref{eq-2fkr}), (\ref{eq-2fir}) and Definition \ref{df-hss} (2), for each $x\in\mathbb{B}_\alpha^r\cap\mathbb{R}$, $K\in\mathbb{S}$, and $\imath\in\{1,2\}$, we have
	\begin{equation}\label{eq-ie1}
	\begin{split}
	\mathcal{I}_{\mathcal{G}}(\mathcal{F}_\imath)\circ\pi|_{\mathcal{G}_\alpha^r}^{-1}(x)
	=&\zeta(K)(\mathcal{F}_\imath)^r_\alpha
	\\=&\zeta(K)(f_\imath(x),0)^T
	\\=&(1,K)(f_\imath(x),0)^T
	\\=&f_\imath(x)
	\end{split}
	\end{equation}
	and
	\begin{equation*}
	\begin{split}
	\mathcal{I}_{\mathcal{G}}(\mathcal{F}_1*\mathcal{F}_2)\circ\pi|_{\mathcal{G}_\alpha^r}^{-1}(x)
	=&\zeta(K)(f_1f_2(x),0)^T
	\\=&f_1f_2(x).
	\end{split}
	\end{equation*}

	Similarly, we have
	\begin{equation}\label{eq-2igf}
		\mathcal{I}_{\mathcal{G}'}(\mathcal{F}_1*\mathcal{F}_2)\circ\pi'|_{(\mathcal{G}')_\alpha^r}^{-1}(x)
		=f_1f_2(x)
	\end{equation}
	
	We notice that $\mathcal{I}_{\mathcal{G}}(\mathcal{F}_1)$ and $\mathcal{I}_{\mathcal{G}}(\mathcal{F}_2)$ are slice regular functions on $G$, and since $\mathcal{G}$ is $*$-preserving, it follows that there exists a slice regular function $f$ on $G$ such that
	\begin{equation*}
	f|_{G^x_\mathbb{R}}=(\mathcal{I}_{\mathcal{G}}(\mathcal{F}_1)|_{G^x_\mathbb{R}}\cdot\mathcal{I}_{\mathcal{G}}(\mathcal{F}_2)|_{G^x_\mathbb{R}}
	\end{equation*}
	
	According to (\ref{eq-ie1}) and (\ref{eq-2igf}), we have
	\begin{equation}\label{eq-isf}
	\begin{split}
	\mathcal{I}_{\mathcal{G}'}(\mathcal{F}_1*\mathcal{F}_2)\circ\pi|_{(\mathcal{G}')_\alpha^r}^{-1}(x)
	=&f_1f_2(x)
	\\=&\mathcal{I}_{\mathcal{G}}(\mathcal{F}_1)\circ\pi|_{\mathcal{G}_\alpha^r}^{-1}(x)\cdot\mathcal{I}_{\mathcal{G}}(\mathcal{F}_2)\circ\pi|_{\mathcal{G}_\alpha^r}^{-1}(x)
	\\=&(\mathcal{I}_{\mathcal{G}}(\mathcal{F}_1)\cdot\mathcal{I}_{\mathcal{G}}(\mathcal{F}_2))\circ\pi|_{\mathcal{G}_\alpha^r}^{-1}(x)
	\\=&f\circ\pi|_{\mathcal{G}_\alpha^r}^{-1}(x)
	\end{split}
	\end{equation}
	for each $x\in\mathbb{B}_\alpha^r\cap\mathbb{R}$. According to \cite[Identity Principle \ref{th-dps}]{Dou2018001}, it is clear that
	\begin{equation}\label{eq-2ifg}
	\mathcal{I}_{\mathcal{G}'}(\mathcal{F}_1*\mathcal{F}_2)\circ\pi|_{(\mathcal{G}')_\alpha^r}^{-1}=f\circ\pi|_{\mathcal{G}_\alpha^r}^{-1}:\mathbb{B}_0\rightarrow\mathbb{H},
	\end{equation}
	where
	\begin{equation*}
	\mathbb{B}_0:=\bigcup_{I\in\mathbb{S}}(\mathbb{B}_\alpha^r)^I=B_\mathbb{H}(x_0,r_\alpha).
	\end{equation*}
	We notice that $(\mathbb{B}_0,id_{\mathbb{B}_0},\pi(x_0))$ is a slice-domain over $\mathbb{H}$ with distinguished point,
	\begin{equation*}
	\mathcal{B}\prec\mathcal{G}\qquad\mbox{and}\qquad\mathcal{B}\prec\mathcal{G}',
	\end{equation*}
	and since (\ref{eq-2ifg}), it follows that
	\begin{equation}\label{eq-2igf1}
	\mathcal{I}_{\mathcal{G}'}(\mathcal{F}_1*\mathcal{F}_2)|_{\mathbb{B}_0}=\mathcal{I}_{\mathcal{G}'}(\mathcal{F}_1*\mathcal{F}_2)\circ\pi|_{(\mathcal{G}')_\alpha^r}^{-1}=f\circ\pi|_{\mathcal{G}_\alpha^r}^{-1}=f|_{\mathbb{B}_0}.
	\end{equation}
	Let $\widehat{\mathcal{G}}=(\widehat{G},\widehat{\pi},\widehat{x_0})$ be a slice-domain of existence of $f$ with respect to $\mathcal{G}$ and
	\begin{equation*}
	\widehat{f}:\widehat{G}\rightarrow\mathbb{H}
	\end{equation*}
	be the slice regular extension of $f$. Thanks to (\ref{eq-2igf1}), $\widehat{f}$ is also a slice regular extension of $\mathcal{I}_{\mathcal{G}'}(\mathcal{F}_1*\mathcal{F}_2)$. Then for each $N\in\mathbb{N}^+$, $\gamma\in\mathcal{P}^\infty_{\mathbb{C}}(\mathcal{G})\in\mathcal{HS}^r(\mathcal{G})$ and $K\in\mathbb{S}^N$,
	\begin{equation*}
	(F_1*F_2)^K_\gamma
	=\mathcal{I}_{\mathcal{G}'}(\mathcal{F}_1*\mathcal{F}_2)\circ\pi'|_{((\mathcal{G}')^r_\gamma)^K}^{-1}
	=\widehat{f}\circ\widehat{\pi}|_{(\widehat{\mathcal{G}}^r_\gamma)^K}^{-1}
	=f\circ\pi|_{(\mathcal{G}^r_\gamma)^K}^{-1}.
	\end{equation*}
	And according to Definition \ref{df-hss}, $f$ is the slice regular function induced by $\mathcal{F}_1*\mathcal{F}_2$ on $\mathcal{G}$, it follows that
	\begin{equation*}
	\mathcal{F}_1*\mathcal{F}_2\in\mathcal{HS}(\mathcal{G}).
	\end{equation*}
	And since Definition \ref{df-2sds},
	\begin{equation*}
	\mathcal{F}_1*\mathcal{F}_2\in\mathcal{HS}^r(\mathcal{G}).
	\end{equation*}
	It follows that $\mathcal{HS}^r(\mathcal{G})$ is closed under the multiplication $*$.
	
	\textbf{2). The multiplication $*$ is distributive with respect to the addition $+$.}
	
	For each $N\in\mathbb{N}^+$ and $a,b,c\in M_{2^N\times 1}(\mathbb{H})$, we have
	\begin{equation*}
	\begin{split}
	(a+b)*c=&\mathcal{I}_N^{-1}\circ\mathcal{I}_N((a+b)*c)
	\\=&\mathcal{I}_N^{-1}(\mathcal{I}_N(a+b)\cdot\mathcal{I}_N(c))
	\\=&\mathcal{I}_N^{-1}(\mathcal{I}_N(a)\cdot\mathcal{I}_N(c)+\mathcal{I}_N(b)\cdot\mathcal{I}_N(c))
	\\=&\mathcal{I}_N^{-1}(\mathcal{I}_N(a*c)+\mathcal{I}_N(b*c))
	\\=&\mathcal{I}_N^{-1}(\mathcal{I}_N(a*c+b*c))
	\\=&a*c+b*c.
	\end{split}
	\end{equation*}
	It follows that,
	\begin{equation*}
	((F_1)_\gamma+(F_2)_\gamma)*(F_3)_\gamma=(F_1)_\gamma*(F_3)_\gamma+(F_2)_\gamma*(F_3)_\gamma,\qquad\forall\ \gamma\in\mathcal{P}^\infty_{\mathbb{C}}(\mathcal{G}).
	\end{equation*}
	Thence
	\begin{equation*}
	(F_1+F_2)*F_3=F_1*F_3+F_2*F_3
	\end{equation*}
	and
	\begin{equation*}
	(\mathcal{F}_1+\mathcal{F}_2)*\mathcal{F}_3=\mathcal{F}_1*\mathcal{F}_3+\mathcal{F}_2*\mathcal{F}_3.
	\end{equation*}
	Similarly, we also have
	\begin{equation*}
	\mathcal{F}_1*(\mathcal{F}_2+\mathcal{F}_3)=\mathcal{F}_1*\mathcal{F}_2+\mathcal{F}_1*\mathcal{F}_3.
	\end{equation*}
	
	\textbf{3). $(1_{\mathcal{SR}(G)})^r_\mathcal{G}$ is the multiplicative identity of $(\mathcal{SR}(G),+,*)$.}
	
	We write
	\begin{equation*}
	(1_{\mathcal{SR}(G)})^r_\mathcal{G}=(F_0,r,\mathcal{P}^\infty_{\mathbb{C}}(\mathcal{G}))=:\mathcal{F}_0,
	\end{equation*}
	then for each $N\in\mathbb{N}^+$, $I\in\mathbb{S}$, $J=\eta_N(I)$, $\gamma\in\mathcal{P}^N_{\mathbb{C}}(\mathcal{G})$, and $x,y\in\mathbb{R}$ with $x+yi\in\mathbb{B}^r_{\gamma}$,
	\begin{equation*}
	\begin{split}
	&(F_0)_\gamma(x+yi)
	\\=&\mathcal{M}(J)^{-1}\mathcal{M}(J)(F_0)_\gamma(x+yi)
	\\=&\mathcal{M}(J)^{-1}(\zeta(J_1)(F_0)_\gamma(x+yi),\zeta(J_2)(F_0)_\gamma(x+yi),...,\zeta(J_{2^N})(F_0)_\gamma(x+yi))^T
	\\=&\mathcal{M}(J)^{-1}(a_1,a_2,...,a_{2^N})^T
	\\=&\mathcal{M}(J)^{-1}(1,1,...,1)_{1\times 2^N}^T
	\\=&\mathcal{M}(J)^{-1}\mathcal{M}(J)(1,0_{1\times 2^N-1})^T
	\\=&(1,0_{1\times 2^N-1})^T
	\end{split}
	\end{equation*}
	where
	\begin{equation*}
	a_\imath:=f\circ\pi|_{(\mathcal{G}^r_\gamma)^{J_\imath}}(x+yJ_{\imath,N}).
	\end{equation*}
	We notice that, for each $N\in\mathbb{N}^+$, $\gamma\in\mathcal{P}^N_{\mathbb{C}}(\mathcal{G})$, and $x,y\in\mathbb{R}$ with $x+yi\in\mathbb{B}^r_{\gamma}$,
	\begin{equation*}
	\begin{split}
	&(F_0)_\gamma(x+yi)*(F_1)_\gamma(x+yi)
	\\=&\mathcal{I}_N^{-1}(\mathcal{I}_N((F_0)_\gamma(x+yi))\cdot\mathcal{I}_N((F_1)_\gamma(x+yi)))
	\\=&\mathcal{I}_N^{-1}((\zeta_N(i)(1,0_{1\times 2^N-1})^T)\cdot\mathcal{I}_N((F_1)_\gamma(x+yi)))
	\\=&\mathcal{I}_N^{-1}(1_{\mathbb{H}}\cdot\mathcal{I}_N((F_1)_\gamma(x+yi)))
	\\=&(F_1)_\gamma(x+yi),
	\end{split}
	\end{equation*}
	therefore
	\begin{equation*}
	F_0*F_1=F_1\qquad\mbox{and}\qquad\mathcal{F}_0*\mathcal{F}_1=\mathcal{F}_1.
	\end{equation*}
	Similarly, we also have $\mathcal{F}_1*\mathcal{F}_0=\mathcal{F}_1$. Then $\mathcal{F}_0$ is the identity element of $(\mathcal{HS}^r(\mathcal{G}),+,*)$.
	
	\textbf{4). $(\mathcal{HS}^r(\mathcal{G}),+,*)$ is a ring.}
	
	According to Proposition \ref{pr-is}, $(\mathcal{HS}^r(\mathcal{G}),+)$ is an abelian group. Thanks to 1) and 3), $(\mathcal{HS}^r(\mathcal{G}),*)$ is a monoid. And since 2), $(\mathcal{HS}^r(\mathcal{G}),+,*)$ is a ring.
	
	\textbf{5). $1_{\mathcal{SR}(G)}$ is the identity element of $(\mathcal{SR}(G),+,*)$.}
	
	We notice that
	\begin{equation*}
	1_{\mathcal{SR}(G)}(q)=1,\qquad q\in G^{x_0}_\mathbb{R}.
	\end{equation*}
	Consequently, for each slice regular function $f$ on $G$,
	\begin{equation*}
	(1_{\mathcal{SR}(G)}*f)|_{G^{x_0}_\mathbb{R}}=(1_{\mathcal{SR}(G)}\cdot f)|_{G^{x_0}_\mathbb{R}}=f|_{G^{x_0}_\mathbb{R}}.
	\end{equation*}
	And since \cite[Identity Principle \ref{th-idh}]{Dou2018001}, we have
	\begin{equation*}
	1_{\mathcal{SR}(G)}*f=f.
	\end{equation*}
	Similarly, we also have
	\begin{equation*}
	f*1_{\mathcal{SR}(G)}=f.
	\end{equation*}
	It follows that $1_{\mathcal{SR}(G)}$ is the identity element of $(\mathcal{SR}(G),+,*)$.
\end{proof}

\begin{thm}
	Let $\mathcal{G}=(G,\pi,x)$ be a $*$-preserving slice-domain over $\mathbb{H}$ with distinguished point, and $r$ be a stem radius system of $\mathcal{G}$. Then $\mathcal{I}_{\mathcal{G}}^r$ is a ring isomorphic between $(\mathcal{HS}^r(\mathcal{G}),+,*)$ and $(\mathcal{SR}(G),+,*)$.
\end{thm}

\begin{proof}
	For each $f,g\in\mathcal{SR}(G)$, $f^r_\mathcal{G}*g^r_\mathcal{G}$ is a holomorphic stem system, by Theorem \ref{th-srhs} and Theorem \ref{th-sp}. And since Proposition \ref{pr-is},
	\begin{equation*}
	h:=\mathcal{I}_{\mathcal{G}}^r(f^r*g^r)
	\end{equation*}
	is a slice regular function on $G$. Let $\alpha$ be the initial path of $\mathcal{P}^{\infty}_{\mathbb{C}}(\mathcal{G})$, defined by (\ref{eq-2asx}).
	
	For each $I\in\mathbb{S}$ and $x\in(\mathbb{B}^r_\alpha)^I\cap\mathbb{R}=(\mathbb{B}^r_\alpha)_{\mathbb{R}}$, we have
	\begin{equation*}
	\begin{split}
	f^r_\alpha(x)=&\left(\begin{matrix}1&I\\1&-I\end{matrix}\right)^{-1}\left(\begin{matrix}f(x)\\f(x)\end{matrix}\right)
	\\=&\frac{1}{2}\left(\begin{matrix}1&1\\-I&I\end{matrix}\right)\left(\begin{matrix}f(x)\\f(x)\end{matrix}\right)
	\\=&(f(x),0)^T
	\end{split}
	\end{equation*}
	Then for each $K\in\mathbb{S}$ and $x\in(\mathbb{B}^r_\alpha)_{\mathbb{R}}$,
	\begin{equation*}
	\begin{split}
	h\circ\pi|^{-1}_{(\mathcal{G}^r_\alpha)^K}(x)=&\zeta(K)(f^r_\alpha*g^r_\alpha(x))
	\\=&\zeta(K)(f^r_\alpha(x)*g^r_\alpha(x))
	\\=&\zeta(K)((f(x),0)^T*(g(x),0)^T)
	\\=&\zeta(K)(f(x)g(x),0)^T
	\\=&(fg)\circ\pi|^{-1}_{(\mathcal{G}^r_\alpha)^K}(x)
	\\=&(f*g)\circ\pi|^{-1}_{(\mathcal{G}^r_\alpha)^K}(x).
	\end{split}
	\end{equation*}
	Thanks to \cite[Identity Principle \ref{th-idh}]{Dou2018001}, it follows that
	\begin{equation*}
	h=f*g.
	\end{equation*}
	According to Proposition \ref{pr-is},
	\begin{equation*}
	\begin{split}
	(\mathcal{I}^r_{\mathcal{G}})^{-1}(f)*(\mathcal{I}^r_{\mathcal{G}})^{-1}(g)
	=&f^r*g^r
	\\=&\mathcal{I}_{\mathcal{G}}^{-1}\circ\mathcal{I}_{\mathcal{G}}(f^r*g^r)
	\\=&\mathcal{I}_{\mathcal{G}}^{-1}(h)
	\\=&\mathcal{I}_{\mathcal{G}}^{-1}(f*g)
	\\=&(f*g)^r
	\\=&(\mathcal{I}^r_{\mathcal{G}})^{-1}(f*g).
	\end{split}
	\end{equation*}
	And since Proposition \ref{pr-is} and Proposition \ref{pr-hsr}, $\mathcal{I}_{\mathcal{G}}^{-1}$ is a ring isomorphic. Then $\mathcal{I}_{\mathcal{G}}$ is also a ring isomorphic.
\end{proof}

Let $\mathcal{G}=(G,\pi,x_0)$ be a $*$-preserving slice-domain over $\mathbb{H}$ with distinguished point, and $f$ be a slice regular function on $G$. For each $I,J\in\mathbb{S}$ with $I\perp J$, there exist $f_\imath:G_\mathbb{R}^{x_0}\rightarrow\mathbb{H}$, $\imath=0,1,2,3$, such that
\begin{equation*}
f|_{G_\mathbb{R}^{x_0}}=f_0+f_1I+f_2J+f_3IJ.
\end{equation*}
We set
\begin{equation*}
g_4:=\frac{1}{2}(f-I*f*I)\in\mathcal{{SR}(G)}\quad\mbox{and}\quad g_0:=\frac{1}{2}(g_4-J*g_4*J)\in\mathcal{{SR}(G)}.
\end{equation*}
We notice that,
\begin{equation*}
\begin{split}
g_4|_{G_\mathbb{R}^{x_0}}=&\frac{1}{2}(f_0+f_1I+f_2J+f_3IJ)-\frac{1}{2}(f_0II+f_1III+f_2IJI+f_3IIJI)
\\=&\frac{1}{2}(f_0+f_1I+f_2J+f_3IJ)+\frac{1}{2}(f_0+f_1I-f_2J-f_3IJ)
\\=&f_0+f_1I,
\end{split}
\end{equation*}
and
\begin{equation*}
\begin{split}
g_0|_{G_\mathbb{R}^{x_0}}=\frac{1}{2}(f_0+f_1I)+\frac{1}{2}(f_0-f_1I)=f_0.
\end{split}
\end{equation*}

Similarly, there exists $g_\imath\in\mathcal{SR}(G)$, such that $g_\imath|_{G_\mathbb{R}^{x_0}}=f_\imath$ for each $\imath\in\{1,2,3\}$. Thence
\begin{equation*}
g_0-g_1I-g_2J-g_3IJ\in\mathcal{SR}(G).
\end{equation*}
We call $g_0-g_1I-g_2J-g_3IJ$ the regular conjugate of $f$, denoted by $f^c$.

We notice that
\begin{equation*}
(f*f^c)|_{G_\mathbb{R}^{x_0}}=f|_{G_\mathbb{R}^{x_0}}\overline{f|_{G_\mathbb{R}^{x_0}}}=\overline{f|_{G_\mathbb{R}^{x_0}}}f|_{G_\mathbb{R}^{x_0}}=(f^c*f)|_{G_\mathbb{R}^{x_0}}.
\end{equation*}
We call $f*f^c$ the symmetrization of $f$, denoted by $f^s$.

Now, we have some propositions with respect to regular conjugates and symmetrizations of slice regular functions, following \cite[Section 5]{Colombo2009001}).

\begin{prop}
	Let $\mathcal{G}=(G,\pi,x_0)$ be a $*$-preserving slice-domain over $\mathbb{H}$ with distinguished point, and $f$ be a slice regular function on $G$. Then
	\begin{equation*}
	(f*g)^c=g^c*f^c\qquad\mbox{and}\qquad(f*g)^s=(g*f)^s.
	\end{equation*}
\end{prop}
\begin{proof}
	We notice that
	\begin{equation*}
		\begin{split}
		(f*g)^c|_{G_\mathbb{R}^{x_0}}&=\overline{(fg)|_{G_\mathbb{R}^{x_0}}}
		\\&=\overline{g|_{G_\mathbb{R}^{x_0}}}\cdot\overline{f|_{G_\mathbb{R}^{x_0}}}
		\\&=(g^c*f^c)|_{G_\mathbb{R}^{x_0}},
		\end{split}
	\end{equation*}
	and since \cite[Identity Principle \ref{th-idh}]{Dou2018001}, it follows that
	\begin{equation*}
	(f*g)^c=g^c*f^c.
	\end{equation*}
	Similarly, since
	\begin{equation*}
	\begin{split}
	(f*g)^s|_{G_\mathbb{R}^{x_0}}
	&=((f*g)*(f*g)^c)|_{G_\mathbb{R}^{x_0}}
	\\&=(f*g)|_{G_\mathbb{R}^{x_0}}\cdot(f*g)^c|_{G_\mathbb{R}^{x_0}}
	\\&=(f|_{G_\mathbb{R}^{x_0}}g|_{G_\mathbb{R}^{x_0}})(\overline{g|_{G_\mathbb{R}^{x_0}}}\cdot\overline{f|_{G_\mathbb{R}^{x_0}}})
	\\&=(f|_{G_\mathbb{R}^{x_0}}\overline{f|_{G_\mathbb{R}^{x_0}}})(g|_{G_\mathbb{R}^{x_0}}\overline{g|_{G_\mathbb{R}^{x_0}}})
	\\&=(g|_{G_\mathbb{R}^{x_0}}f|_{G_\mathbb{R}^{x_0}})(\overline{f|_{G_\mathbb{R}^{x_0}}}\cdot\overline{g|_{G_\mathbb{R}^{x_0}}})
	\\&=(g*f)^s|_{G_\mathbb{R}^{x_0}},
	\end{split}
	\end{equation*}
	we have
	\begin{equation*}
	(f*g)^s=(g*f)^s.
	\end{equation*}
\end{proof}

\begin{defn}
	Let $N\in\mathbb{N}^+$. A slice-domain $\mathcal{G}=(G,\pi,x)$ over $\mathbb{H}$ with distinguished point is called $N$-part if, for each $q\in G$, there exists an $N$-part path $\gamma$ in $G$ from $x$ to $q$.
\end{defn}

\begin{prop}\label{pr-asp}
	Let $\mathcal{G}=(G,\pi,x_0)$ be an $*$-preserving $1$-part slice-domain over $\mathbb{H}$ with distinguished point, and $f$, $g$ be slice regular functions on $G$. Then
	\begin{equation*}
	(f*g)^s=f^sg^s,
	\end{equation*}
	where $f^sg^s$ is defined by
	\begin{equation*}
	f^sg^s(q):=f^s(q)g^s(q),\qquad\forall\ q\in G.
	\end{equation*}
\end{prop}

\begin{proof}
	For each $I\in\mathbb{S}$ and $q\in G_I$, there exists a domain $U$ in $G_I$ containing $q$. Then there exists $p$ in $U$ with $\pi(p)\notin\mathbb{R}$, and a path $\alpha$ in $U$ from $p$ to $q$. Since $\mathcal{G}$ is $1$-part, there exists a path $\beta$ in $G_I$ from $x_0$ to $p$. Thus $\beta\alpha$ is a path in $G_I$ from $x_0$ to $q$. Consequently, $G_I$ is path-connected, and $(G_I,\pi|_{G_I},x_0)$ is a slice-domain over $\mathbb{C}$ with distinguished point.
	
	We notice that, for each $I\in\mathbb{S}$, $f^sg^s$ is a $\mathbb{C}_I$-valued holomorphic function on $G_I$. It follows that $f^sg^s$ is slice regular function on $G$. Thanks to
	\begin{equation*}
	\begin{split}
	(f*g)^s|_{G_\mathbb{R}^{x_0}}&=(f|_{G_\mathbb{R}^{x_0}}g|_{G_\mathbb{R}^{x_0}})(\overline{g|_{G_\mathbb{R}^{x_0}}}\cdot\overline{f|_{G_\mathbb{R}^{x_0}}})
	\\&=(f|_{G_\mathbb{R}^{x_0}}\overline{f|_{G_\mathbb{R}^{x_0}}})(g|_{G_\mathbb{R}^{x_0}}\overline{g|_{G_\mathbb{R}^{x_0}}})
	\\&=(f*f^c)|_{G_\mathbb{R}^{x_0}}\cdot(g*g^c)|_{G_\mathbb{R}^{x_0}}
	\\&=(f^sg^s)|_{G_\mathbb{R}^{x_0}},
	\end{split}
	\end{equation*}
	and since \cite[Identity Principle \ref{th-idh}]{Dou2018001}, it follows that
	\begin{equation*}
	(f*g)^s=f^sg^s.
	\end{equation*}
\end{proof}

\begin{rmk}\label{rm-asp}
	Let $\Omega$ be an axially symmetric slice domain in $\mathbb{H}$ (see \cite{Colombo2009001}), and $x_0\in\Omega\cap{\mathbb{R}}$. Then $(\Omega,id_{\Omega},x_0)$ is a (schlicht \cite[Definition \ref{df-1asd}]{Dou2018001}) $\infty$-symmetric $1$-part slice-domain over $\mathbb{H}$ with distinguished point. Moreover, $(\Omega,id_{\Omega},x_0)$ is $*$-preserving.
\end{rmk}

\begin{proof}
	This remark follows immediately from \cite[Page 1803]{Colombo2009001}.
\end{proof}

\begin{defn}
	Let $\mathcal{G}$ be an $\infty$-axially symmetric slice-domain over $\mathbb{H}$ with distinguished point. $\mathcal{G}$ is called $1$-strong symmetric, if for each $I\in\mathbb{S}$ and $\alpha,\beta\in\mathcal{P}_\mathcal{G}^1(\mathcal{G})$ with $\alpha^I_\mathcal{G}(1)=\beta^I_\mathcal{G}(1)$, we have
	\begin{equation*}
	\alpha^K_\mathcal{G}(1)=\beta^K_\mathcal{G}(1),\qquad\forall\ K\in\mathbb{S}.
	\end{equation*}
\end{defn}

\begin{prop}
	Let $\mathcal{G}=(G,\pi,x_0)$ be an $*$-preserving slice-domain over $\mathbb{H}$ with distinguished point, and $f$, $g$ be slice regular functions on $G$. If $\mathcal{G}$ is $1$-strong symmetric, then
	\begin{equation}\label{eq-smdm}
	f*g(\gamma_{\mathcal{G}}^I(1))=f(\gamma_{\mathcal{G}}^I(1))g(\gamma_{\mathcal{G}}^{f(\gamma_{\mathcal{G}}^I(1))^{-1}If(\gamma_{\mathcal{G}}^I(1))}(1))
	\end{equation}
	for each $\gamma\in\mathcal{P}^1_{\mathbb{C}}(\mathcal{G})$, and $I\in\mathbb{S}$.
\end{prop}

\begin{proof}
	Let $r$ be a stem radius system of $\mathcal{G}$ and $I\in\mathbb{S}$. We define a function $\phi_I$ on $G_I^{x_0}$, by
	\begin{equation}\label{eq-2pig}
	\phi_I(\gamma_{\mathcal{G}}^I(1)):=f(\gamma_{\mathcal{G}}^I(1))g(\gamma_{\mathcal{G}}^{f(\gamma_{\mathcal{G}}^I(1))^{-1}If(\gamma_{\mathcal{G}}^I(1))}(1))
	\end{equation}
	for each $\gamma\in\mathcal{P}^1_{\mathbb{C}}(\mathcal{G})$, where $G_I^{x_0}$ is the connected component in $G_I$ containing $x_0$. Since $\mathcal{G}$ is $1$-strong symmetric, then when $p:=\gamma_{\mathcal{G}}^I(1)$ is fixed, $\gamma_{\mathcal{G}}^{f(p)^{-1}If(p)}(1)$ is not depend on the choice of $\gamma$. Hence for each $p\in G_I^{x_0}$, $f*g(p)$ is well defined, by (\ref{eq-2pig}).
	
	We notice that
	\begin{equation*}
	(1,I)\sigma_1=I(1,I)=(1,I)\diag(I,I),
	\end{equation*}
	and
	\begin{equation*}
	F\cdot(1,J)=(F,FF^{-1}IF)=(1,I)\diag(F,F),
	\end{equation*}
	it is clear that
	\begin{equation*}
	\begin{split}
	&(\frac{\partial}{\partial x}+I\frac{\partial}{\partial y})\phi_I\circ\pi_{(\mathcal{G}_\gamma^r)^I}^{-1}(x+yI)
	\\=&(\frac{\partial}{\partial x}+I\frac{\partial}{\partial y})(f\circ\pi_{(\mathcal{G}^r_\gamma)^I}^{-1}(x+yI)\cdot g\circ\pi_{(\mathcal{G}^r_\gamma)^J}^{-1}(x+yJ))
	\\=&(\frac{\partial}{\partial x}+I\frac{\partial}{\partial y})(F\cdot(1,J)g^r_\gamma(x+yi))
	\\=&(\frac{\partial}{\partial x}+I\frac{\partial}{\partial y})((1,I)\diag(F,F)g^r_\gamma(x+yi))
	\\=&(1,I)(\frac{\partial}{\partial x}+\sigma_1\frac{\partial}{\partial y})(\diag(F,F)g^r_\gamma(x+yi))
	\\=&(1,I)((\frac{\partial}{\partial x}+\diag(I,I)\frac{\partial}{\partial y})\diag(F,F))g^r_\gamma(x+yi)
	\\&+(1,I)\diag(F,F)((\frac{\partial}{\partial x}+\sigma_1\frac{\partial}{\partial y})g^r_\gamma(x+yi))
	\\=&0
	\end{split}
	\end{equation*}
	for each $\gamma\in\mathcal{P}^1_{\mathbb{C}}(\mathcal{G})$, and $x,y\in\mathbb{R}$ with $x+yi\in\mathbb{B}_\gamma^r$, where
	\begin{equation*}
	F:=f\circ\pi_{(\mathcal{G}^r_\gamma)^I}^{-1}(x+yI)\qquad\mbox{and}\qquad J:=F^{-1}IF.
	\end{equation*}
	Then $\phi_I$ is a holomorphic function on $G_I^{x_0}$. And since
	\begin{equation*}
	\phi_I(q)=f(q)g(q)=f*g(q),\qquad\forall\ q\in G_\mathbb{R}^{x_0}.
	\end{equation*}
	It follows that
	\begin{equation*}
	\phi_I=(f*g)|_{G^{x_0}_I},\qquad\forall\ I\in\mathbb{S}.
	\end{equation*}
	And we notice that, for each $I\in\mathbb{S}$ and $\gamma\in\mathcal{P}^1_{\mathbb{C}}(\mathcal{G})$,
	\begin{equation*}
	\gamma^I_{\mathcal{G}}(1)\in G_I^{x_0},
	\end{equation*}
	then (\ref{eq-smdm}) holds.
\end{proof}

\begin{prop}\label{pr-iesr}
	Let $\mathcal{G}=(G,\pi,x_0)$ be a $*$-preserving slice-domain over $\mathbb{H}$ with distinguished point, and $f$ be a slice regular function on $G$. If $f^s(q)\neq 0$, for each $q\in G$. Then there exists an inverse element $f^{-*}$ of $f$ in $(\mathcal{SR}(G),*)$, given by
	\begin{equation*}
	f^{-*}(q):=f^s(q)^{-1}*f^c(q),\qquad\forall\ q\in G.
	\end{equation*}
	We call $f^{-*}$ the regular reciprocal of $f$.
\end{prop}

\begin{proof}
	We define a function
	\begin{equation*}
	(f^s)^{-1}:G\rightarrow\mathbb{H},\qquad q\mapsto f^s(q)^{-1},\qquad\forall\ q\in G.
	\end{equation*}
	We notice that $(f^s)^{-1}$ is a slice regular function on $G$. Then
	\begin{equation*}
	\begin{split}
	f^{-*}*f(q)=&f^s(q)^{-1}f^c(q)f(q)\\=&(f^c(q)f(q))^{-1}f^c(q)f(q)\\=&1\\=&1_{\mathcal{SR}(G)}(q)
	\end{split}
	\end{equation*}
	for each $q\in G_{\mathbb{R}}^{x_0}$. According to \cite[Identity Principle \ref{th-idh}]{Dou2018001}, we have
	\begin{equation*}
	f^{-*}*f=1_{\mathcal{SR}(G)}.
	\end{equation*}
	Similarly
	\begin{equation*}
	\begin{split}
	f*f^{-*}(q)=&f(q)f^s(q)^{-1}f^c(q)
	\\=&f(q)(f^c(q)f(q))^{-1}f^c(q)
	\\=&1
	\\=&1_{\mathcal{SR}(G)}(q)
	\end{split}
	\end{equation*}
	for each $q\in G^{x_0}_\mathbb{R}$. According to \cite[Identity Principle \ref{th-idh}]{Dou2018001},
	\begin{equation*}
	f*f^{-*}=1_{\mathcal{SR}(G)}.
	\end{equation*}
\end{proof}

\section{Power series on $*$-preserving Riemann slice-domain}

The power series of slice regular function on quaternions is studied in \cite{Gentili2012001}. There are a new series expansion over quaternions in \cite{Stoppato2012001} and series over real alternative *-algebra in \cite{Ghiloni0214002}. In this section, we consider the power series expansions of slice regular functions on $*$-preserving slice-domains over $\mathbb{H}$ with distinguished point.

\begin{defn}
	Let $\Omega$ be a slice-domain in $\mathbb{H}$, and $f:\Omega\rightarrow\mathbb{H}$ be a function with partial derivatives. The $I$-derivative $\partial_If:\Omega_I\rightarrow\mathbb{H}$ of $f$ is defined by
	\begin{equation*}
	\partial_If(x+yI):=\frac{1}{2}(\frac{\partial}{\partial x}-I\frac{\partial}{\partial y})f|_{\Omega_I}(x+yI)
	\end{equation*}
	for each $I\in\mathbb{S}$ with $\Omega_I\neq\varnothing$, and $x,y\in\mathbb{R}$ with $x+yI\in\Omega_I$. The slice derivative of $f$ is the function $f^{(1)}=\partial_cf:\Omega\rightarrow\mathbb{H}$ defined by $\partial_I f$ on $\Omega_I$ for all $I\in\mathbb{S}$.
	
	We set
	\begin{equation*}
	g^{(0)}:=g
	\end{equation*}
	for each quaternion-valued function $g$ defined on $\Omega$.
	
	Let $n\in\mathbb{N}^+$. If $f$ has $n$-th order partial derivatives. We set
	\begin{equation*}
	f^{(n)}:=(\partial_c)^n f.
	\end{equation*}
\end{defn}

This definition is well defined because
\begin{equation*}
\partial_I f(x+0I)=\frac{\partial}{\partial x} f|_{\Omega_{\mathbb{R}}}(x)=\partial_J f(x+0J),\qquad\forall\ I,J\in\mathbb{S}\ \mbox{and}\ x\in\Omega_\mathbb{R}.
\end{equation*}

\begin{defn}\label{df-1}
	Let $\mathcal{G}=(G,\pi,x_0)$ be a slice-domain over $\mathbb{H}$ with distinguished point, $r$ be a stem radius system of $\mathcal{G}$, and $f$ be a slice regular function on $G$. We define $f^{(n)}:G\rightarrow\mathbb{H}$, by
	\begin{equation}\label{eq-2fng}
	f^{(n)}|_{(\mathcal{G}^r_\gamma)^K}:=(f\circ\pi|_{(\mathcal{G}^r_\gamma)^K}^{-1})^{(n)}
	\end{equation}
	for each $n\in\mathbb{N}$, $N\in\mathbb{N}^+$, $\gamma\in\mathcal{P}^N_{\mathbb{C}}(\mathcal{G})$, and $K\in\mathbb{S}$.
	
	We call $f^{(n)}$ the $n$-th derivative of $f$.
\end{defn}

Since $f^{(n)}(q)$ only depends on $f_q$, for each $q\in\mathbb{H}$, Definition \ref{df-1} is well defined, and $f^{(n)}$ does not depend on the choice of $r$.

We denote the ``$\sigma$-ball" (see \cite[Page 121]{Gentili2012001}) by
\begin{equation*}
\Sigma(p,r):=\{x+yJ\in\mathbb{H}:x,y\in\mathbb{R},x\pm yI\in B_I(p,r)\}.
\end{equation*}

\begin{defn}
	Let $\mathcal{G}=(G,\pi,x_0)$ be a $*$-preserving slice-domain over $\mathbb{H}$ with distinguished point, and $f$ be a slice regular function on $G$. We set
	\begin{equation*}
	f^{*0}:=1_{\mathcal{SR}(G)}\qquad\mbox{and}\qquad f^{*(n+1)}:=f^{*n}*f,\qquad\forall\ n\in\mathbb{N}.
	\end{equation*}
\end{defn}

Let $p\in\mathbb{H}$, and $\mathcal{G}=(G,\pi,x_0)$ be a slice-domain over $\mathbb{H}$ with distinguished point. We denote a function
\begin{equation*}
h_p:\mathbb{H}\rightarrow\mathbb{H},\qquad q\mapsto q-p,\qquad\forall\ q\in\mathbb{H}.
\end{equation*}
Obviously, $h_p$ is a slice regular function on $\mathbb{H}$, and
\begin{equation*}
\mathcal{G}_0:=(\mathbb{H},id_\mathbb{H},0)
\end{equation*}
is an $\infty$-axially symmetric and $1$-part slice domain over $\mathbb{H}$ with distinguished point. According to Proposition \ref{rm-asp}, $\mathcal{G}_0$ is $*$-preserving. Then $h^{*n}$ is a slice-regular function on $\mathbb{H}$ for each $n\in\mathbb{N}$. We notice that $(\mathbb{H},\id_\mathbb{H},\pi(x_0))$ is also a slice domain over $\mathbb{H}$ with distinguished point with
\begin{equation*}
\mathcal{G}\prec(\mathbb{H},id_\mathbb{H},\pi(x_0)),
\end{equation*}
and $\pi:G\rightarrow\mathbb{H}$ is the fiber preserving map from $\mathcal{G}$ to $(\mathbb{H},id_\mathbb{H},\pi(x_0))$. Let $U$ be a domain in $G$, such that $\pi|_U:U\rightarrow\pi(U)$ is a slice-homeomorphism. We denote
\begin{equation}\label{eq-2qpn}
(q'-p')^{*n}_{\pi|_U}:=h_p^{*n}|_G(q'),
\end{equation}
where
\begin{equation*}
q':=\pi|_U^{-1}(q)\qquad\mbox{and}\qquad p':=\pi|_U^{-1}(p)
\end{equation*}
for each $p,q\in\mathbb{H}$.

\begin{thm}\label{th-t}
	Let $\mathcal{G}=(G,\pi,x_0)$ be a slice-domain over $\mathbb{H}$ with distinguished point, and $f$ be a slice regular function on $G$. If $I\in\mathbb{S}$, $r>0$, $q_0\in G_I$, and $U$ is a domain in $G$ containing $q_0$ with $\pi|_U:U\rightarrow\Sigma(\pi(q_0),r)$ be a slice-homeomorphism, then
	\begin{equation}\label{eq-tl}
	f(q)=\sum_{n=0}^{+\infty}\frac{1}{n!}(q-q_0)^{*n}_{\pi|_U} f^{(n)}(q_0),\qquad\forall\ q\in U.
	\end{equation}
\end{thm}

\begin{proof}
	If $q\in U_I$, according to (\ref{eq-2fng}) and (\ref{eq-2qpn}), we have
	\begin{equation*}
	\begin{split}
	f(q)=&f\circ\pi|_U^{-1}(\pi(q))
	\\=&\sum_{n=0}^{+\infty}\frac{1}{n!}(\pi(q)-\pi(q_0))^n(f\circ\pi|_U^{-1})^{(n)}(\pi(q_0))
	\\=&\sum_{n=0}^{+\infty}\frac{1}{n!}(q-q_0)^{*n}_{\pi|_U} f^{(n)}(q_0).
	\end{split}
	\end{equation*}
	According to \cite[Theorem 8]{Gentili2012001}, the series in (\ref{eq-tl}) converges in $U$. And since \cite[Identity Principle \ref{th-idh}]{Dou2018001}, (\ref{eq-tl}) holds.
\end{proof}

Let $\mathcal{G}=(G,\pi,x_0)$ be a $*$-preserving slice-domain over $\mathbb{H}$ with distinguished point, $r$ be a stem radius system of $\mathcal{G}$, $f$, $g$ be slice regular functions on $G$, and $\alpha$ be the initial path of $\mathcal{P}^\infty_{\mathbb{C}}(\mathcal{G})$. Thanks to (\ref{eq-2fng}), we have
\begin{equation*}
\begin{split}
	(f*g)^{(n)}\circ\pi|_{(\mathcal{G}_\alpha^r)^I}^{-1}(x)=&((f*g)\circ\pi|_{(\mathcal{G}_\alpha^r)^I}^{-1})^{(n)}(x)
	\\=&((f\circ\pi|_{(\mathcal{G}_\alpha^r)^I}^{-1}\cdot g\circ\pi|_{(\mathcal{G}_\alpha^r)^I}^{-1})^{(n)}(x)
	\\=&(\frac{\partial}{\partial x})^n(f\circ\pi|_{(\mathcal{G}_\alpha^r)^I}^{-1}\cdot g\circ\pi|_{(\mathcal{G}_\alpha^r)^I}^{-1})(x)
	\\=&\sum_{m=0}^n\binom{n}{m}(f\circ\pi|_{(\mathcal{G}_\alpha^r)^I}^{-1})^{(m)}(g\circ\pi|_{(\mathcal{G}_\alpha^r)^I}^{-1})^{(n-m)}(x)
	\\=&\sum_{m=0}^n(\binom{n}{m}(f^{(m)})*(g^{(n-m)}))\circ\pi|_{(\mathcal{G}_\alpha^r)^I}^{-1}(x)
\end{split}
\end{equation*}
for each $x\in\mathbb{B}^r_\alpha\cap\mathbb{R}$, $n\in\mathbb{N}$ and $I\in\mathbb{S}$. And since \cite[Identity Principle \ref{th-idh}]{Dou2018001}, we have
\begin{equation}\label{eq-2fsg}
(f*g)^{(n)}=\sum_{m=0}^n\binom{n}{m}f^{(m)}*g^{(n-m)},\qquad\forall\ n\in\mathbb{N}.
\end{equation}

\begin{prop}
	Let $\mathcal{G}=(G,\pi,x_0)$ be a $*$-preserving slice-domain over $\mathbb{H}$ with distinguished point, and $f$, $g$ be slice regular functions on $G$. If $I\in\mathbb{S}$, $r>0$, $q_0\in G_I$, and $U$ is a domain in $G$ containing $q_0$ with
	\begin{equation*}
	\pi_U:U\rightarrow\Sigma(\pi(q_0),r)
	\end{equation*}
	be a slice-homeomorphism, then
	\begin{equation*}
	f*g(q)=\sum_{n=0}^{+\infty}(\frac{1}{n!}(q-q_0)^{*n}_{\pi|_U} \sum_{m=0}^n\binom{n}{m}f^{(m)}*g^{(n-m)}(q_0)),\qquad\forall\ q\in U.
	\end{equation*}
\end{prop}

\begin{proof}
	This proposition follows immediately from Theorem \ref{th-t} and (\ref{eq-2fsg}).
\end{proof}

\begin{defn}
	Let $\Omega$ be a domain in $\mathbb{C}$, $N\in\mathbb{N}^+$, and $F:\Omega\rightarrow\mathbb{C}^{\otimes N}\otimes\mathbb{H}$ be a map. If $F$ has $n$-th order partial derivatives, then we call $F^{(n)}$ be the $n$-th derivative of the stem function $F$, defined by
	\begin{equation*}
	F^{(n)}:=(\partial_{i_{N,N}})^n F,\qquad\forall\ n\in\mathbb{N},
	\end{equation*}
	where
	\begin{equation*}
	\partial_{i_{N,N}}:=\frac{1}{2}(\frac{\partial}{\partial x}-i_{N,N}\frac{\partial}{\partial y}).
	\end{equation*}
\end{defn}

\begin{defn}
	Let $\Omega$ be a domain in $\mathbb{C}$, $N\in\mathbb{N}^+$, and $F:\Omega\rightarrow M_{2^N\times 1}(\mathbb{H})$ be a map. If $F$ has $n$-th order partial derivatives, then we call $F^{(n)}$ be the $n$-th derivative of the stem function $F$, defined by
	\begin{equation*}
	F^{(n)}:=(\partial_{\sigma_N})^n F,\qquad\forall\ n\in\mathbb{N},
	\end{equation*}
	where
	\begin{equation*}
	\partial_{\sigma_N}:=\frac{1}{2}(\frac{\partial}{\partial x}-\sigma_N\frac{\partial}{\partial y}).
	\end{equation*}
\end{defn}

For each $N\in\mathbb{N}^+$ and $x,y\in\mathbb{R}$ with $x+yi=z\in\mathbb{C}$, we set
\begin{equation*}
z^{[N]}:=x\mathbb{I}_{2^N}+y\sigma_N
\end{equation*}
and
\begin{equation*}
z^{(N)}:=x_\mathbb{H}+y_\mathbb{H}i_{N,N}.
\end{equation*}

\begin{thm}\label{th-str}
	Let $\mathcal{G}=(G,\pi,q)$ be a $\infty$-axially symmetric slice-domain over $\mathbb{H}$ with distinguished point, $r$ be a stem radius system of $\mathcal{G}$, and $f$ be a slice regular function on $G$. If $N\in\mathbb{N}^+$ and $\gamma\in\mathcal{P}^N_{\mathbb{C}}(\mathcal{G})$, then $f^r_\gamma$ and $\mathcal{I}_N(f^r_\gamma)$ have any order partial derivatives,
	\begin{equation}\label{eq-se}
	f^r_\gamma(z)=\sum_{n=0}^{+\infty}\frac{1}{n!}(z^{[N]}-z^{[N]}_0)^n(f^r_\gamma)^{(n)}(z_0),
	\end{equation}
	and
	\begin{equation}\label{eq-te}
	\mathcal{I}_N(f^r_\gamma)(z)=\sum_{n=0}^{+\infty}(\frac{1}{n!}(z^{(N)}-z_0^{(N)})^n\cdot(\mathcal{I}_N(f^r_\gamma))^{(n)}(z_0))
	\end{equation}
	for each $z\in\mathbb{B}^r_\gamma$, where $z_0:=\gamma(1)$. Moreover,
	\begin{equation}\label{eq-eee}
	\mathcal{I}_N(f^r_\gamma)^{(n)}=\mathcal{I}_N((f^r_\gamma)^{(n)})=\mathcal{I}_N((f^{(n)})^r_\gamma),\qquad\forall\ n\in\mathbb{N}.
	\end{equation}
\end{thm}

\begin{proof}
	1. Notice that $\mathcal{I}_N(f^r_\gamma):\mathbb{B}^r_{\gamma}\rightarrow\mathbb{C}^{\otimes N}\otimes\mathbb{H}$ is a tensor $N$-holomorphic function, i.e.,
	\begin{equation*}
		(\frac{\partial}{\partial x}+i_{N,N}\frac{\partial}{\partial y})\mathcal{I}_N(f^r_\gamma)(x+yi)=0_{\mathbb{H}}
	\end{equation*}
	for each $x,y\in\mathbb{R}$ with $x+yi\in\mathbb{B}^r_\gamma$. Let $I',J'\in\mathbb{S}$ with $I'\perp J'$. We write
	\begin{equation*}
		\mathcal{I}_N(f^r_\gamma)=\sum_{\jmath=1}^{2^{N-1}}(F_\jmath^0+F_\jmath^1 I'_\mathbb{H}+F_\jmath^2 J'_\mathbb{H}+F_\jmath^3 I'_\mathbb{H}J'_\mathbb{H})i_N(\jmath),
	\end{equation*}
	where
	\begin{equation*}
	F_\jmath^\imath:\mathbb{B}^r_\gamma\rightarrow\{1\}^{\otimes N-1}\otimes\mathbb{C}\otimes\{1\},\qquad\forall\ \jmath\in\{1,2,...,2^{N-1}\}\ \mbox{and}\ \imath\in\{0,1,2,3\}.
	\end{equation*}
	We notice that
	\begin{equation*}
		(\frac{\partial}{\partial x}+i_{N,N}\frac{\partial}{\partial y})F^\imath_\jmath(x+yi)=0_{\mathbb{H}},
	\end{equation*}
	and $F^\imath_\jmath$ has any order partial derivatives, for each $x,y\in\mathbb{R}$ with $x+yi\in\mathbb{B}^r_\gamma$, $\jmath\in\{1,2,...,2^{N-1}\}$ and $\imath\in\{0,1,2,3\}$. Note that
	\begin{equation*}
	\{1\}^{\otimes N-1}\otimes\mathbb{C}\otimes\{1\}=\mathbb{C}_{i_{N,N}},
	\end{equation*}
	where
	\begin{equation*}
	\mathbb{C}_{i_{N,N}}:=\{x_\mathbb{H}+y_\mathbb{H}i_{N,N}:x,y\in\mathbb{R}\}.
	\end{equation*}
	It follows that,
	\begin{equation*}
		F^\imath_\jmath(z)=\sum_{n=0}^{+\infty}(\frac{1}{n!}(z^{(N)}-z_0^{(N)})^n (F^\imath_\jmath)^{(n)}(z_0))
	\end{equation*}
	for each $z\in\mathbb{B}^r_\gamma$, $\jmath\in\{1,2,...,2^{N-1}\}$ and $\imath\in\{0,1,2,3\}$. Then
	\begin{equation}\label{eq-tte}
		\begin{split}
			\mathcal{I}_N(f^r_\gamma)(z)=&\sum_{\imath=0}^{3}\sum_{\jmath=1}^{2^{N-1}}\sum_{n=0}^{+\infty}(\frac{1}{n!}(z^{(N)}-z_0^{(N)})^n (F^\imath_\jmath)^{(n)}(z_0)K_\imath i_N(\jmath))
			\\=&\sum_{n=0}^{+\infty}(\frac{1}{n!}\sum_{\imath=0}^{3}\sum_{\jmath=1}^{2^{N-1}}((z^{(N)}-z_0^{(N)})^n (F^\imath_\jmath)^{(n)}(z_0)K_\imath i_N(\jmath)))
			\\=&\sum_{n=0}^{+\infty}(\frac{1}{n!}(z^{(N)}-z_0^{(N)})^n\cdot(\mathcal{I}_N(f^r_\gamma))^{(n)}(z_0)),
		\end{split}
	\end{equation}
	and $\mathcal{I}_N(f^r_\gamma)$ has any order partial derivatives, for each $z\in\mathbb{B}^r_\gamma$, where
	\begin{equation*}
	K_0:=1_\mathbb{H},\quad K_1:=I'_\mathbb{H},\quad K_2:=J'_\mathbb{H}\quad\mbox{and}\quad K_3:=I'_\mathbb{H}J'_\mathbb{H}.
	\end{equation*}
	Then (\ref{eq-te}) holds.
	
	2. We notice that $(\sigma_N)^2=-\mathbb{I}_{2^N}$, and since (\ref{eq-tt2}), it follows that
	\begin{equation}\label{eq-tese}
	\begin{split}
	\mathcal{I}_N((f^r_\gamma)^{(n)})=&\mathcal{I}_N((\frac{1}{2}(\frac{\partial}{\partial x}-\sigma_N\frac{\partial}{\partial y}))^nf^r_\gamma)
	\\=&(\frac{1}{2}(\frac{\partial}{\partial x}-i_{N,N}\frac{\partial}{\partial y}))^n\mathcal{I}_N(f^r_\gamma)
	\\=&\mathcal{I}_N(f^r_\gamma)^{(n)},
	\end{split}
	\end{equation}
	where $n=0$. And according Proposition \ref{pr-vt} and (\ref{eq-tte}), it follows that $f^r_\gamma$ has any order partial derivatives. Then (\ref{eq-tese}) holds for each $n\in\mathbb{N}$. And thanks to (\ref{eq-tt2}),
	\begin{equation*}
	\begin{split}
	&(z^{(N)}-z_0^{(N)})^n\cdot(\mathcal{I}_N(f^r_\gamma)^{(n)}(z_0))
	\\=&((x-x_0)+(y-y_0)i_{N,N})^n\cdot\mathcal{I}_N((f^r_\gamma)^{(n)}(z_0))
	\\=&\mathcal{I}_N(((x-x_0)+(y-y_0)\sigma_N)^n(f^r_\gamma)^{(n)}(z_0))
	\\=&\mathcal{I}_N((z^{[N]}-z^{[N]}_0)^n(f^r_\gamma)^{(n)}(z_0))
	\end{split}
	\end{equation*}
	for each $z=x+yi\in\mathbb{B}^r_\gamma$ and $n\in\mathbb{N}$, where $x_0,y_0\in\mathbb{R}$ with $x_0+y_0i=z_0$. And since (\ref{eq-tte}), we have
	\begin{equation*}
	\begin{split}
	f^r_\gamma(z)=&\mathcal{I}_N^{-1}\circ\mathcal{I}_N(f^r_\gamma)(z)
	\\=&\mathcal{I}_N^{-1}(\sum_{n=0}^{+\infty}(\frac{1}{n!}(z^{(N)}-z_0^{(N)})^n\cdot(\mathcal{I}_N(f^r_\gamma))^{(n)}(z_0)))
	\\=&\mathcal{I}_N^{-1}(\sum_{n=0}^{+\infty}(\frac{1}{n!}\mathcal{I}_N((z^{[N]}-z^{[N]}_0)^n(f^r_\gamma)^{(n)}(z_0))))
	\\=&\sum_{n=0}^{+\infty}(\frac{1}{n!}(z^{[N]}-z^{[N]}_0)^n(f^r_\gamma)^{(n)}(z_0))
	\end{split}
	\end{equation*}
	for each $z\in\mathbb{B}^r_\gamma$. Then (\ref{eq-se}) holds.
	
	3. Let $I\in\mathbb{S}$ and we set
	\begin{equation*}
	J:=\eta_N(I).
	\end{equation*}
	According to \cite[Proposition \ref{pr-ct}]{Dou2018001}, (\ref{eq-2fkr}), (\ref{eq-2fpi}) and (\ref{eq-tl}), we have
	\begin{equation}\label{eq-tt3}
		\begin{split}
			f^r_\gamma(z)=&\mathcal{M}(J)^{-1}\mathcal{M}(J)f^r_\gamma(z)
			\\=&\mathcal{M}(J)^{-1}(\zeta(J_1)f^r_\gamma(z),\zeta(J_2)f^r_\gamma(z),...,\zeta(J_{2^N})f^r_\gamma(z))^T
			\\=&\mathcal{M}(J)^{-1}
			\left(\begin{matrix}
				f\circ\pi|_{(\mathcal{G}^r_\gamma)^{J_1}}^{-1}(P_{J_{1,N}}(z))\\f\circ\pi|_{(\mathcal{G}^r_\gamma)^{J_2}}^{-1}(P_{J_{2,N}}(z))\\\vdots\\f\circ\pi|_{(\mathcal{G}^r_\gamma)^{J_{2^N}}}^{-1}(P_{J_{2^N,N}}(z))
			\end{matrix}\right)
			\\=&\mathcal{M}(J)^{-1}
			\left(\begin{matrix}
				\sum_{n=0}^{+\infty}\frac{1}{n!}P_{J_{1,N}}(z-z_0)^n(f\circ\pi|_{(\mathcal{G}^r_\gamma)^{J_1}}^{-1})^{(n)}(P_{J_{1,N}}(z_0))\\\sum_{n=0}^{+\infty}\frac{1}{n!}P_{J_{2,N}}(z-z_0)^n(f\circ\pi|_{(\mathcal{G}^r_\gamma)^{J_2}}^{-1})^{(n)}(P_{J_{2,N}}(z_0))\\\vdots\\\sum_{n=0}^{+\infty}\frac{1}{n!}P_{J_{2^N,N}}(z-z_0)^n(f\circ\pi|_{(\mathcal{G}^r_\gamma)^{J_{2^N}}}^{-1})^{(n)}(P_{J_{2^N,N}}(z_0)))
			\end{matrix}\right)
			\\=&\sum_{n=0}^{+\infty}\frac{1}{n!}\mathcal{M}(J)^{-1}((x-x_0)+(y-y_0)D_N(J))^n A
			\\=&\sum_{n=0}^{+\infty}\frac{1}{n!}((x-x_0)+(y-y_0)\sigma_N)^n\mathcal{M}(J)^{-1} A
			\\=&\sum_{n=0}^{+\infty}\frac{1}{n!}(z^{[N]}-z^{[N]}_0)^n(f^{(n)})^r_\gamma(z_0)
		\end{split}
	\end{equation}
	for each $z\in\mathbb{B}^r_\gamma$, where
	\begin{equation*}
		A:=\left(\begin{matrix}
			(f\circ\pi|_{(\mathcal{G}^r_\gamma)^{J_1}}^{-1})^{(n)}(P_{J_{1,N}}(z_0))\\(f\circ\pi|_{(\mathcal{G}^r_\gamma)^{J_2}}^{-1})^{(n)}(P_{J_{2,N}}(z_0))\\\vdots\\(f\circ\pi|_{(\mathcal{G}^r_\gamma)^{J_{2^N}}}^{-1})^{(n)}(P_{J_{2^N,N}}(z_0)))
		\end{matrix}\right).
	\end{equation*}
	And thanks to (\ref{eq-tt}) and (\ref{eq-2ti}),
	\begin{equation}\label{eq-tte1}
		\begin{split}
			\mathcal{I}_N(f^r_\gamma)(z)=&\zeta_N(i)(\sum_{n=0}^{+\infty}(\frac{1}{n!}(z^{[N]}-z^{[N]}_0)^n(f^{(n)})^r_\gamma(z_0)))_\mathbb{H}
			\\=&\sum_{n=0}^{+\infty}\frac{1}{n!}((z^{(N)}-z^{(N)}_0)^n\cdot\zeta_N(i)(f^{(n)})^r_\gamma(z_0)_\mathbb{H})
			\\=&\sum_{n=0}^{+\infty}\frac{1}{n!}((z^{(N)}-z^{(N)}_0)^n\cdot\mathcal{I}_N((f^{(n)})^r_\gamma)(z_0))
		\end{split}
	\end{equation}
	for each $z\in\mathbb{B}^r_\gamma$. And since (\ref{eq-tte}), then
	\begin{equation*}
	\mathcal{I}_N(f^r_\gamma))^{(n)}(z_0)=\mathcal{I}_N((f^{(n)})^r_\gamma)(z_0),\qquad\forall\ n\in\mathbb{N}.
	\end{equation*}
	And according to (\ref{eq-tte1}), we have
	\begin{equation*}
	\mathcal{I}_N(f^r_\gamma))^{(n)}(z)=\mathcal{I}_N((f^{(n)})^r_\gamma)(z),\qquad\forall\ z\in\mathbb{B}^r_\gamma\ \mbox{and}\ n\in\mathbb{N}.
	\end{equation*}
	Hence
	\begin{equation*}
		\mathcal{I}_N(f^r_\gamma)^{(n)}=\mathcal{I}_N((f^{(n)})^r_\gamma),\qquad\forall\ n\in\mathbb{N}.
	\end{equation*}
	And thanks to (\ref{eq-tese}), (\ref{eq-eee}) holds.
\end{proof}

\section{Final remarks}

We have developed some fundamental concepts and results of the theory of slice regular functions on Riemann slice-domains over quaternions. However, we can not discuss singular points of slice regular functions on Riemann slice-domains, e.g., the condition ``$f^s(q)\neq 0$" in Proposition \ref{pr-iesr} is not intrinsic. It is necessary and might be possible to introduce a ``slice-analytic space" with singularities.

\section*{Acknowledgement}

The authors would like to thank Dr. Xieping Wang for the discussion of the theory about Riemann domains over $\mathbb{C}^n$ and pointing out several inaccuracies in a draft of this paper.

\bibliographystyle{abbrv}
\bibliography{mybibfile}

\begin{thebibliography}{10}

\bibitem{Alpay2015001}
D.~Alpay, F.~Colombo, J.~Gantner, and I.~Sabadini.
\newblock A new resolvent equation for the {$S$}-functional calculus.
\newblock {\em J. Geom. Anal.}, 25(3):1939--1968, 2015.

\bibitem{Alpay2016001}
D.~Alpay, F.~Colombo, T.~Qian, and I.~Sabadini.
\newblock The {$H^\infty$} functional calculus based on the {$S$}-spectrum for
  quaternionic operators and for {$n$}-tuples of noncommuting operators.
\newblock {\em J. Funct. Anal.}, 271(6):1544--1584, 2016.

\bibitem{Alpay2012001}
D.~Alpay, F.~Colombo, and I.~Sabadini.
\newblock Schur functions and their realizations in the slice hyperholomorphic
  setting.
\newblock {\em Integral Equations Operator Theory}, 72(2):253--289, 2012.

\bibitem{Alpay2013001}
D.~Alpay, F.~Colombo, and I.~Sabadini.
\newblock Pontryagin-de {B}ranges-{R}ovnyak spaces of slice hyperholomorphic
  functions.
\newblock {\em J. Anal. Math.}, 121:87--125, 2013.

\bibitem{Bisi2012001}
C.~Bisi and C.~Stoppato.
\newblock The {S}chwarz-{P}ick lemma for slice regular functions.
\newblock {\em Indiana Univ. Math. J.}, 61(1):297--317, 2012.

\bibitem{Colombo2018001}
F.~Colombo and J.~Gantner.
\newblock Fractional powers of quaternionic operators and {K}ato's formula
  using slice hyperholomorphicity.
\newblock {\em Trans. Amer. Math. Soc.}, 370(2):1045--1100, 2018.

\bibitem{Colombo2010001}
F.~Colombo, G.~Gentili, and I.~Sabadini.
\newblock A {C}auchy kernel for slice regular functions.
\newblock {\em Ann. Global Anal. Geom.}, 37(4):361--378, 2010.

\bibitem{Colombo2009001}
F.~Colombo, G.~Gentili, I.~Sabadini, and D.~Struppa.
\newblock Extension results for slice regular functions of a quaternionic
  variable.
\newblock {\em Adv. Math.}, 222(5):1793--1808, 2009.

\bibitem{Colombo2015001}
F.~Colombo, R.~L\'avi\~vcka, I.~Sabadini, and V.~Sou\~vcek.
\newblock The {R}adon transform between monogenic and generalized slice
  monogenic functions.
\newblock {\em Math. Ann.}, 363(3-4):733--752, 2015.

\bibitem{Colombo2011001}
F.~Colombo and I.~Sabadini.
\newblock The quaternionic evolution operator.
\newblock {\em Adv. Math.}, 227(5):1772--1805, 2011.

\bibitem{Colombo2010003}
F.~Colombo, I.~Sabadini, and D.~C. Struppa.
\newblock Duality theorems for slice hyperholomorphic functions.
\newblock {\em J. Reine Angew. Math.}, 645:85--105, 2010.

\bibitem{Colombo2010002}
F.~Colombo, I.~Sabadini, and D.~C. Struppa.
\newblock An extension theorem for slice monogenic functions and some of its
  consequences.
\newblock {\em Israel J. Math.}, 177:369--389, 2010.

\bibitem{Dou2018001}
X.~Dou and G.~Ren.
\newblock Riemann slice-domains over quaternions {I}.
\newblock {\em in preparation}, 2018.

\bibitem{Fueter1934/35}
R.~Fueter.
\newblock Die {F}unktionentheorie der {D}ifferentialgleichungen {$\Theta u=0$}
  und {$\Theta\Theta u=0$} mit vier reellen {V}ariablen.
\newblock {\em Comment. Math. Helv.}, 7(1):307--330, 1934.

\bibitem{Gentili2014001}
G.~Gentili, S.~Salamon, and C.~Stoppato.
\newblock Twistor transforms of quaternionic functions and orthogonal complex
  structures.
\newblock {\em J. Eur. Math. Soc. (JEMS)}, 16(11):2323--2353, 2014.

\bibitem{Gentili2016001}
G.~Gentili, G.~Sarfatti, and D.~C. Struppa.
\newblock Ideals of regular functions of a quaternionic variable.
\newblock {\em Math. Res. Lett.}, 23(6):1645--1663, 2016.

\bibitem{Gentili2008001}
G.~Gentili and C.~Stoppato.
\newblock Zeros of regular functions and polynomials of a quaternionic
  variable.
\newblock {\em Michigan Math. J.}, 56(3):655--667, 2008.

\bibitem{Gentili2012001}
G.~Gentili and C.~Stoppato.
\newblock Power series and analyticity over the quaternions.
\newblock {\em Math. Ann.}, 352(1):113--131, 2012.

\bibitem{Gentili2006001}
G.~Gentili and D.~C. Struppa.
\newblock A new approach to {C}ullen-regular functions of a quaternionic
  variable.
\newblock {\em C. R. Math. Acad. Sci. Paris}, 342(10):741--744, 2006.

\bibitem{Gentili2007001}
G.~Gentili and D.~C. Struppa.
\newblock A new theory of regular functions of a quaternionic variable.
\newblock {\em Adv. Math.}, 216(1):279--301, 2007.

\bibitem{Gentili2008002}
G.~Gentili, D.~C. Struppa, and F.~Vlacci.
\newblock The fundamental theorem of algebra for {H}amilton and {C}ayley
  numbers.
\newblock {\em Math. Z.}, 259(4):895--902, 2008.

\bibitem{Ghiloni2013001}
R.~Ghiloni, V.~Moretti, and A.~Perotti.
\newblock Continuous slice functional calculus in quaternionic {H}ilbert
  spaces.
\newblock {\em Rev. Math. Phys.}, 25(4):1350006, 83, 2013.

\bibitem{Ghiloni2011001}
R.~Ghiloni and A.~Perotti.
\newblock Slice regular functions on real alternative algebras.
\newblock {\em Adv. Math.}, 226(2):1662--1691, 2011.

\bibitem{Ghiloni0214002}
R.~Ghiloni and A.~Perotti.
\newblock Power and spherical series over real alternative {$^*$}-algebras.
\newblock {\em Indiana Univ. Math. J.}, 63(2):495--532, 2014.

\bibitem{Ghiloni20171001}
R.~Ghiloni, A.~Perotti, and C.~Stoppato.
\newblock Singularities of slice regular functions over real alternative
  {$^*$}-algebras.
\newblock {\em Adv. Math.}, 305:1085--1130, 2017.

\bibitem{Ghiloni2018001}
R.~Ghiloni and V.~Recupero.
\newblock Slice regular semigroups.
\newblock {\em Trans. Amer. Math. Soc.}, 370(7):4993--5032, 2018.

\bibitem{Ren2017002}
G.~Ren and X.~Wang.
\newblock Julia theory for slice regular functions.
\newblock {\em Trans. Amer. Math. Soc.}, 369(2):861--885, 2017.

\bibitem{Stoppato2012001}
C.~Stoppato.
\newblock A new series expansion for slice regular functions.
\newblock {\em Adv. Math.}, 231(3-4):1401--1416, 2012.

\bibitem{Zhang1997001}
F.~Zhang.
\newblock Quaternions and matrices of quaternions.
\newblock {\em Linear Algebra Appl.}, 251:21--57, 1997.

\end{thebibliography}

\end{document}